\documentclass[11pt]{amsart}

\usepackage[utf8]{inputenc}
\usepackage{mathtools}
\usepackage{amsmath}
\usepackage{amsfonts}
\usepackage{fancyhdr}
\usepackage{amssymb}
\usepackage{amsthm}
\usepackage[margin= 1 in]{geometry}

\usepackage[colorlinks=true, pdfstartview=FitV, linkcolor=blue, citecolor=blue, urlcolor=blue]{hyperref}

\usepackage{xcolor}

\usepackage{tikz}
\usetikzlibrary{cd}
\usepackage{ytableau}
\usepackage{makecell}
\usepackage{bm}

\usepackage[colorinlistoftodos]{todonotes}

\newcommand{\defn}[1]{{\color{darkred}\emph{#1}}} 
\definecolor{darkblue}{rgb}{0.0,0,0.7} 
\definecolor{darkred}{rgb}{0.7,0,0} 
\definecolor{darkgreen}{rgb}{0, .6, 0} 

\newcommand{\HVT}{\mathsf{HVT}}
\newcommand{\MVT}{\mathsf{MVT}}
\newcommand{\SVT}{\mathsf{SVT}}
\newcommand{\SSYT}{\mathsf{SSYT}}
\newcommand{\RPP}{\mathsf{RPP}}
\newcommand{\ev}{\mathsf{ev}}

\newcommand{\sfA}{\mathsf{A}}
\newcommand{\sfL}{\mathsf{L}}
\newcommand{\sfH}{\mathsf{H}}
\newcommand{\shape}{\mathsf{shape}}

\newcommand*\circled[1]{\tikz[baseline=(char.base)]{\node[shape=circle,draw,inner sep=2pt] (char) {$#1$};}}

\newtheorem{theorem}{Theorem}[section]
\newtheorem{lemma}[theorem]{Lemma}
\newtheorem{corollary}[theorem]{Corollary}

\newtheorem{definition}[theorem]{Definition}
\newtheorem{proposition}[theorem]{Proposition}
\newtheorem{example}[theorem]{Example}
\newtheorem{remark}[theorem]{Remark}
\numberwithin{equation}{section}

\title{Uncrowding algorithm for hook-valued tableaux}

\author[J.~Pan]{Jianping Pan}
\address[J. Pan]{Department of Mathematics, UC Davis, One Shields Ave., Davis, CA 95616-8633, U.S.A.\\
current address: Department of Mathematics, NC State University, Raleigh, NC 27695-8205, U.S.A.}
\email{jpan9@ncsu.edu}

\author[J.~Pappe]{Joseph Pappe}
\address[J. Pappe]{Department of Mathematics, UC Davis, One Shields Ave., Davis, CA 95616-8633, U.S.A.}
\email{jhpappe@ucdavis.edu}

\author[W.~Poh]{Wencin Poh}
\address[W. Poh]{Department of Mathematics, UC Davis, One Shields Ave., Davis, CA 95616-8633, U.S.A.}
\email{wpoh@ucdavis.edu}

\author[A.~Schilling]{Anne Schilling}
\address[A. Schilling]{Department of Mathematics, UC Davis, One Shields Ave., Davis, CA 95616-8633, U.S.A.}
\email{anne@math.ucdavis.edu}

\date{\today}

\date{\today}

\begin{document}

\begin{abstract}
Whereas set-valued tableaux are the combinatorial objects associated to stable Grothen\-dieck polynomials, 
hook-valued tableaux are associated to stable canonical Grothendieck polynomials. In this paper, we define 
a novel uncrowding algorithm for hook-valued tableaux. The algorithm ``uncrowds'' the entries in the arm of the
hooks and yields a set-valued tableau and a column-flagged increasing tableau. We prove that our uncrowding algorithm
intertwines with crystal operators. An alternative uncrowding algorithm that ``uncrowds'' the entries in the leg instead of the
arm of the hooks is also given. As an application of uncrowding, we obtain various expansions of the canonical Grothendieck 
polynomials.
\end{abstract}

\maketitle 

\section{Introduction}
	
Set-valued tableaux play an important role in the $K$-theory of the Grassmannian.
They form a generalization of semi-standard Young tableaux, where boxes may contain sets of integers rather
than just integers~\cite{Buch.2002}. In particular, the \defn{stable symmetric Grothendieck polynomial} indexed by the partition
$\lambda$ is the generating function of set-valued tableaux 
\begin{equation}
\label{equation.G stable}
	G_\lambda(x; \beta) = \sum_{T \in \SVT(\lambda)} \beta^{|T|-|\lambda|} x^{\mathsf{weight}(T)},
\end{equation}
where $\SVT(\lambda)$ is the set of \defn{set-valued tableaux} of shape $\lambda$ and $\mathsf{weight}(T)$ is the vector
with $i$-th entry being the number of $i$ in $T$. Here $|T|$ is the number of entries in $T$ and $|\lambda|$ is the size of $\lambda$. 
Stable symmetric Grothendieck polynomials $G_\lambda$ can be viewed
as a $K$-theory analogue of the Schur functions $s_\lambda$ (while the Grothendieck polynomial is an analog of the
Schubert polynomial~\cite{LS.1983}). Buch~\cite{Buch.2002} also described the structure coefficients $c_{\lambda \mu}^\nu$, which
is the coefficient of $G_{\nu}$ in the expansion of $G_\lambda G_\mu$ in terms of set-valued tableaux, generalizing the 
Littlewood--Richardson rule for Schur functions.

The Grassmannian $\operatorname{Gr}(k,\mathbb{C}^n)$ of $k$-planes in $\mathbb{C}^n$ has a fundamental duality 
isomorphism
\[
	\operatorname{Gr}(k,\mathbb{C}^n) \cong \operatorname{Gr}(n-k,\mathbb{C}^n).
\]
This implies that the structure constants have the symmetry $c_{\lambda\mu}^\nu = c_{\lambda'\mu'}^{\nu'}$, where
$\lambda'$ denotes the conjugate of the partition $\lambda$ (see for example~\cite[Example 9.20]{Fulton.1997}). Hence one expects a ring 
homomorphism on the completion
of the ring of symmetric function defined on the basis of stable symmetric Grothendieck polynomials
$\tau(G_\lambda) = G_{\lambda'}$. The standard involutive ring automorphism $\omega$ defined on the Schur basis
by $\omega(s_\lambda) = s_{\lambda'}$ does not have this property~\cite{LP.2007}
\[
	\omega(G_\lambda) = J_\lambda \neq G_{\lambda'},
\]
where $J_\lambda$ is the weak symmetric Grothendieck polynomial.

Yeliussizov~\cite{Yeliussizov.2017} introduced a new family of \defn{canonical stable Grothendieck polynomials}
$G_\lambda(x;\alpha,\beta)$ such that
\[
	\omega(G_\lambda(x;\alpha,\beta)) = G_{\lambda'}(x;\beta,\alpha).
\]
Combinatorially, the canonical stable Grothendieck polynomials can be expressed as generating functions of
\defn{hook-valued tableaux}. In a hook-valued tableau, each box contains a semistandard Young tableau of hook shape,
which is weakly increasing in rows and strictly increasing in columns. More precisely
\[
	G_\lambda(x;\alpha,\beta) = \sum_{T\in \HVT(\lambda)} \alpha^{a(T)} \beta^{\ell(T)} x^{\mathsf{weight}(T)},
\]
where $\HVT(\lambda)$ is the set of hook-valued tableaux of shape $\lambda$, $a(T)$ is the sum of all arm lengths 
and $\ell(T)$ is the sum of all leg lengths of the hook tableaux in $T$.

A hook-valued tableau $T$ is a set-valued tableau when all hook tableaux entries are single columns or equivalently
$a(T)=0$. Hence $G_\lambda(x;\alpha,\beta)$ specializes to $G_\lambda(x;\beta)$ for $\alpha=0$. 
Similarly, a hook-valued tableau $T$ is a multiset-valued tableau when all hook tableaux entries are single 
rows or equivalently $\ell(T)=0$. Hence $G_\lambda(x;\alpha,\beta)$ specializes to $J_{\lambda}(x;\alpha)$ 
for $\beta=0$. 

\smallskip

In this paper, we describe a novel \defn{uncrowding algorithm} on hook-valued tableaux
(see Definitions~\ref{def: uncrowd_bump}, \ref{def: uncrowd_insert} and~\ref{def: uncrowd_map}).
The uncrowding algorithm on set-valued tableaux was originally developed by Buch~\cite[Theorem 6.11]{Buch.2002}
to give a bijective proof of Lenart's Schur expansion of symmetric stable Grothendieck polynomials~\cite{Lenart.2000}.
This uncrowding algorithm takes as input a set-valued tableau and produces a semistandard Young tableau (using the RSK bumping 
algorithm to uncrowd cells that contain more than one integer) and a flagged increasing tableau~\cite{Lenart.2000} (also known as an elegant 
filling~\cite{LP.2007,BandlowMorse.2012,Patrias.2016}), which serves as a recording tableau.

Chan and Pflueger~\cite{ChanPflueger.2019} provide an expansion of stable Grothendieck polynomials indexed by skew
partitions in terms of skew Schur functions. Their proof uses a generalization of the uncrowding algorithm of 
Lenart~\cite{Lenart.2000}, Buch~\cite{Buch.2002}, and Reiner, Tenner and Yong~\cite{RTY.2018} to skew shapes. Their analysis is motivated
geometrically by identifying Euler characteristics of Brill--Noether varieties up to sign as counts of set-valued standard 
tableaux. The uncrowding algorithm was also used in the analysis of $K$-theoretic analogues of the Hopf algebras 
of symmetric functions, quasisymmetric functions, noncommutative symmetric functions, and of
the Malvenuto--Reutenauer Hopf algebra of permutations~\cite{LP.2007,BandlowMorse.2012,Patrias.2016}.
In~\cite{GZJ.2020}, a vertex model for canonical Grothendieck polynomials and their duals was studied, which was used to
derive Cauchy identities.

An important property of the uncrowding algorithm on set-valued tableaux is that it intertwines with crystal 
operators~\cite{MPS.2018} (see also~\cite{MPPS.2020}). The crystal structure on a combinatorial set is
the combinatorial shadow of a (quantum) group representation (see for example~\cite{HongKang.2002,BumpSchilling.2017}).
A crystal structure on hook-valued tableaux was recently introduced by Hawkes and Scrimshaw~\cite{HS.2019}.
Our novel uncrowding map on hook-valued tableaux yields a set-valued tableau and a recording tableau.
We prove that it intertwines with crystal operators (see Proposition~\ref{proposition.main} and Theorem~\ref{theorem.main}).
This was stated as an open problem in~\cite{HS.2019}.

The paper is organized as follows. In Section~\ref{section.hook-valued tableaux}, we review the definition of
semistandard hook-valued tableaux of~\cite{Yeliussizov.2017} and the crystal structure on them~\cite{HS.2019}.
In Section~\ref{section.uncrowding}, we define the new uncrowding map on hook-valued tableaux and prove
that it intertwines with the crystal operators and other properties. We also give a variant of the uncrowding algorithm
on hook-valued tableaux. In Section~\ref{section.application}, we consider applications of the uncrowding algorithm, in particular 
expansions of the canonical Grothendieck polynomials using techniques developed in~\cite{BandlowMorse.2012}.

\subsection*{Acknowledgments}
We are grateful to Graham Hawkes and Travis Scrimshaw for discussions.

This work was partially supported by NSF grant DMS--1764153. 
JiP was partially supported by NSF grant DMS--1700814. 
AS was partially supported by NSF grant DMS--1760329.

\section{Hook-valued tableaux}
\label{section.hook-valued tableaux}

In Section~\ref{section.hook}, we define hook-valued tableaux~\cite{Yeliussizov.2017} and in
Section~\ref{section.crystal} we review the crystal structure on hook-valued tableaux as
introduced in~\cite{HS.2019}.

\subsection{Hook-valued tableaux}
\label{section.hook}

A semistandard Young tableau $U$ of hook shape is a tableau of the form
\[U = \raisebox{4em}{\begin{ytableau}
\ell_p\\
\vdots\\
\ell_1\\
x & a_1 & \ldots & a_q
\end{ytableau}} \; ,\]
where the integer entries weakly increase from left to right and strictly increase from bottom to top. 
Note that we use French notation for Young diagrams and tableaux throughout the paper.
In this case, $\sfH(U)=x$ is called the \defn{hook entry} of $U$, 
$\sfL(U)=(\ell_1,\ell_2,\ldots,\ell_p)$ is the \defn{leg} of $U$, 
and $\sfA(U)=(a_1,a_2,\ldots,a_q)$ is the \defn{arm} of $U$. 
Both the arm and the leg of $U$ are allowed to be empty.
Additionally, the \defn{extended leg} of $U$ is defined as 
$\sfL^+(U)=(x,\ell_1,\ell_2,\ldots,\ell_p)$.
We denote by $\max(U)$ (resp. $\min(U)$) the maximal (resp. minimal) entry in $U$.

\begin{definition} \cite{Yeliussizov.2017}
	Fix a partition $\lambda$. 
	A \defn{semistandard hook-valued tableau} (or \defn{hook-valued tableau} for short) $T$ of shape $\lambda$ is a 
	filling of the Young diagram for $\lambda$ with (nonempty) semistandard Young tableaux of hook shape such that:
	\begin{enumerate}
		\item[(i)] $\max(A)\leqslant \min(B)$ whenever the cell containing $A$ is in the same row, but left of the cell
		 containing $B$;
		\item[(ii)] $\max(A)<\min(C)$ whenever the cell containing $A$ is in the same column, but below the cell 
		containing $C$.
	\end{enumerate}
	The set of all hook-valued tableaux of shape $\lambda$ (respectively, with entries at most $m$) is 
	denoted by $\HVT(\lambda)$ (respectively, $\HVT^m(\lambda)$).
	
	Given a hook-valued tableau $T$, its \defn{arm excess} is the total number of integers in the arms 
	of all cells of $T$, while its \defn{leg excess} is the total number of integers in the legs of all cells of $T$.
\end{definition}

\begin{remark}
	In the special case when a hook-valued tableau has arm excess 0, it is also called a \defn{set-valued tableau}. 
	Similarly, a \defn{multiset-valued tableau} is a hook-valued tableau with leg excess 0.
	We use the notation $\SVT(\lambda)$ (resp. $\SVT^m(\lambda)$) and $\MVT(\lambda)$ (resp. $\MVT^m(\lambda)$) 
	for the set of all set-valued tableaux of shape $\lambda$ (resp. with entries at most $m$) and the set 
	of all multiset-valued tableaux of shape $\lambda$ (resp. with entries at most $m$), respectively.
\end{remark}

\subsection{Crystal structure on hook-valued tableaux}
\label{section.crystal}

Hawkes and Scrimshaw~\cite{HS.2019} defined a crystal structure on hook-valued tableaux.
We review their definition here.

\begin{definition}[\cite{HS.2019}, Definition 4.1]
	Let $C$ be a hook-valued tableau of column shape. 
	The column reading word $R(C)$ is obtained by reading the extended leg in each cell from top to bottom, followed by
	reading all of the remaining entries, arranged in a weakly increasing order.
	
	For a hook-valued tableau $T$, its column reading word is formed by concatenating the column 
	reading words of all of its columns, read from left to right, that is,
	\[
		R(T)= R(C_1)R(C_2)\ldots R(C_\ell),
	\]
	where $\ell$ is the number of columns of $T$ and $C_i$ is the $i$th column 
	of $T$. 
\end{definition}

\begin{example}\label{eg: column_reading}
	Let $T$ be the hook-valued tableau
    \[{\def\mc#1{\makecell[lb]{#1}}
    T =\,{\begin{array}[lb]{*{3}{|l}|}\cline{1-2}
\mc{4\\33}&\mc{5}\\\cline{1-3}
\mc{2\\11}&\mc{4\\334}&\mc{4445}\\\cline{1-3}
\end{array}\,.}}\]
	The column reading words for the columns of $T$ are respectively 432113, 54334 and 4445, so that 
	\[
		R(T)=432113543344445.
	\]
\end{example}

\begin{definition} \cite[Definition 4.3]{HS.2019}
\label{def: hvt_crystal}
	Let $T\in \HVT^m(\lambda)$. For any $1\leqslant i<m$, we employ the following pairing rules.
	Assign $-$ to every $i$ in $R(T)$ and assign $+$ to every $i+1$ in 
	$R(T)$. Then, successively pair each $+$ that is 
	adjacent and to the left of a $-$, removing all paired signs until nothing can be paired.
	
	The \defn{operator $f_i$} acts on $T$ according to the following rules in the given order. If there is no unpaired~$-$, 
	then $f_i$ annihilates $T$. Otherwise, locate the cell $c$ with entry the hook-valued tableau $B=T(c)$ containing 
	the $i$ corresponding to the rightmost unpaired $-$.
	\begin{enumerate}
		\item [(M)] If there is an $i+1$ in the cell above $c$ with entry $B^\uparrow$, 
		then $f_i$ removes an $i$ from $\sfA(B)$ and adds $i+1$ to $\sfA(B^\uparrow)$.
		\item [(S)] Otherwise, if there is a cell to the right of $c$ with entry $B^\rightarrow$, such that it contains an $i$ in 
		$\sfL^+(B^\rightarrow)$, then $f_i$ removes the $i$ from 
		$\sfL^+(B^\rightarrow)$ and adds $i+1$ to $\sfL(B)$.
		\item [(N)] Else, $f_i$ changes the $i$ in $B$ into an $i+1$.
	\end{enumerate}
	
	Similarly, the \defn{operator $e_i$} acts on $T$ according to the following rules in the given order. If there is no 
	unpaired $+$, then $e_i$ annihilates $T$. Otherwise, locate the cell $c$ with entry the hook-valued tableau $B=T(c)$
	containing the entry $i+1$ corresponding to the leftmost unpaired $+$.
	\begin{enumerate}
		\item [(M)] If there is an $i$ in the cell below $c$ with entry $B^\downarrow$, then $e_i$ removes the $i+1$ from 
		$\sfA(B)$ and adds $i$ to $\sfA(B^\downarrow)$.
		\item [(S)] Otherwise, if there is a cell to the left of $c$ with entry $B^\leftarrow$, such that it contains an $i+1$ 
		in $\sfL(B^\leftarrow)$, then $e_i$ removes the $i+1$ from 
		$\sfL(B^\leftarrow)$ and adds $i$ to $\sfL^+(B)$.
		\item [(N)] Else, $e_i$ changes the $i+1$ in $B$ into an $i$.
	\end{enumerate}
	Based on the pairing procedure above, $\varphi_i(T)$ is the number of unpaired $-$, whereas $\varepsilon_i(T)$ is 
	the number of unpaired $+$. 
\end{definition}

We remark that the definition of crystal operators on $\HVT$ specializes to the definition on $\SVT$ in \cite{MPS.2018} 
or the one on $\MVT$ in \cite{HS.2019} when the arm excess or leg excess of the tableaux is set to 0, respectively. 

\begin{example}
	\label{eg:crystal_semistd_multiset_valued_tableau}
	Consider the following hook-valued tableau $T$:
    \[
        {\def\mc#1{\makecell[lb]{#1}}
        {T =\, \begin{array}[lb]{*{2}{|l}|}\cline{1-2}
        \mc{4\\34}&\mc{5\\4}\\\cline{1-2}
        \mc{2\\11}&\mc{3\\233}\\\cline{1-2}
        \end{array}.}}  
    \]

    Then, $e_3$ annihilates $T$, whereas
\[
    {\def\mc#1{\makecell[lb]{#1}}
{e_1(T) =\, \begin{array}[lb]{*{2}{|l}|}\cline{1-2}
\mc{4\\34}&\mc{5\\4}\\\cline{1-2}
\mc{11}&\mc{3\\2\\133}\\\cline{1-2}
\end{array}}\,,\quad
f_1(T) =\, \begin{array}[lb]{*{2}{|l}|}\cline{1-2}
    \mc{4\\34}&\mc{5\\4}\\\cline{1-2}
    \mc{2\\12}&\mc{3\\233}\\\cline{1-2}
    \end{array}\,,\quad
f_3(T) =\,  \begin{array}[lb]{*{2}{|l}|}\cline{1-2}
    \mc{4\\34}&\mc{5\\44}\\\cline{1-2}
    \mc{2\\11}&\mc{3\\23}\\\cline{1-2}
    \end{array}\,.
}
\]
\end{example}

For a given cell $(r, c)$ in row $r$ and column $c$ in a hook-valued tableau $T$, 
let $L_T(r,c)$ be the leg of $T(r, c)$, let $\sfA_{T}(r, c)$ 
be arm of $T(r, c)$, let $H_T(r, c)$ be the hook entry of $T(r, c)$, and let 
$L^{+}_{T}(r,c)$ be the extended leg of $T(r,c)$.

\section{Uncrowding map on hook-valued tableaux}
\label{section.uncrowding}

In Section~\ref{section.uncrowding SVT}, we first review the uncrowding map on set-valued tableaux.
In Section~\ref{section.uncrowding HVT}, we give a new uncrowding map on hook-valued tableaux and prove
some of its properties in Section~\ref{section.uncrowding properties}. The relation to the uncrowding map on
multiset-valued tableaux is given in Section~\ref{section.uncrowding MVT}. In Section~\ref{section.crowding},
we give the inverse of the uncrowding map on hook-valued tableaux, called the crowding map. 
In Section~\ref{section.alternative uncrowding}, an alternative definition of the uncrowding map on
hook-valued tableaux is provided.

\subsection{Uncrowding map on set-valued tableaux}
\label{section.uncrowding SVT}

For set-valued tableaux, there exists an uncrowding operator, 
which maps a set-valued tableau to a pair of tableaux, one being 
a semistandard Young tableau and the other a flagged increasing tableau
(see for example~\cite{Lenart.2000,Buch.2002,BandlowMorse.2012,RTY.2018}). 
In this setting, the uncrowding operator intertwines with the crystal operators on set-valued 
tableaux and semistandard Young tableaux, respectively \cite{MPS.2018}.

Consider partitions $\lambda$, $\mu$ with $\lambda\subseteq\mu$ and $\lambda_1=\mu_1$. 
A \defn{flagged increasing tableau} (introduced in~\cite{Lenart.2000} and called (strict) \defn{elegant fillings} by various
authors~\cite{LP.2007,BandlowMorse.2012,Patrias.2016})
is a row and column strict filling of the skew shape $\mu/\lambda$ such that 
the positive integer entries in the $i$-th row of the tableau are at most $i-1$ for all $1\leqslant i\leqslant \ell(\mu)$, 
where $\ell(\mu)$ is the length of partition $\mu$. In particular, the bottom row is empty. 
The set of all flagged increasing tableaux is denoted by $\mathcal{F}$. The set of all flagged increasing tableaux
of shape $\mu/\lambda$ with $\lambda_1=\mu_1$ is denoted by $\mathcal{F}(\mu/\lambda)$.

We now review the uncrowding operation on set-valued tableaux. We call a cell in a set-valued tableau
a \defn{multicell} if it contains more than one letter.

\begin{definition}
\label{definition.uncrowding SVT}
Define the \defn{uncrowding operation} on $T\in \SVT(\lambda)$ as follows.
First identify the topmost row $r$ in $T$ with a multicell. Let $x$ be the largest letter in row $r$ that lies in a multicell; 
remove $x$ from the cell and perform RSK row bumping with $x$ into the rows above. The resulting tableau, whose 
shape differs from $\lambda$ by the addition of one cell, is the output of this operation.

The \defn{uncrowding map} on set-valued tableaux
\begin{equation}
\label{equation.U SVT}
	\mathcal{U}_{\SVT}: \, \SVT(\lambda) \longrightarrow \bigsqcup_{\mu\supseteq \lambda}\SSYT(\mu)\times \mathcal{F}(\mu/\lambda)
\end{equation}
is defined as follows. Let $T\in \SVT(\lambda)$ with leg excess $\ell$.
\begin{enumerate}
    \item Initialize $P_0 = T$ and $Q_0 = F_0$, where $F_0$ is the unique flagged increasing tableau of shape 
    $\lambda/\lambda$.
    \item For each $1\leqslant i \leqslant \ell$, $P_i$ is obtained from $P_{i-1}$ by applying the uncrowding 
    operation. Let $C$ be the cell in $\shape(P_i)/\shape(P_{i-1})$. If $C$ is in row $r'$, then $F_i$ is obtained from
    $F_{i-1}$ by adding cell $C$ with entry $r'-r$.
    \item Set $\mathcal{U}_{\SVT}(T) = (P,F) := (P_\ell,F_\ell)$.
\end{enumerate}
\end{definition}

It was proved in~\cite[Section 6]{Buch.2002} that $\mathcal{U}_{\SVT}$ in~\eqref{equation.U SVT} is a bijection.
Monical, Pechenik and Scrimshaw~\cite{MPS.2018} proved that $\mathcal{U}_{\SVT}$ intertwines with the crystal operators
on set-valued tableaux (see also~\cite{MPPS.2020}).
A similar uncrowding algorithm for multiset-valued tableaux was given in~\cite[Section 3.2]{HS.2019}.

\subsection{Uncrowding map on hook-valued tableaux}
\label{section.uncrowding HVT}

In~\cite{HS.2019}, the authors ask for an uncrowding map for hook-valued
tableaux which intertwines with the crystal operators. Here we provide such an uncrowding map by uncrowding the arm
excess in a hook-valued tableaux to obtain a set-valued tableaux. An alternative obtained by uncrowding the leg excess first
is given in Section~\ref{section.uncrowding MVT}. 

\begin{definition} \label{def: uncrowd_bump}
The \defn{uncrowding bumping} $\mathcal{V}_{b} \colon \HVT \rightarrow \HVT$ is defined by the following algorithm:
\renewcommand\labelenumi{(\arabic{enumi})}
\renewcommand\theenumi\labelenumi
\begin{enumerate}
    \item \label{def: uncrowd_bump_1} Initialize $T$ as the input.
    \item \label{def: uncrowd_bump_2} If the arm excess of T equals zero, return T.
    \item \label{def: uncrowd_bump_3} Else, find the rightmost column that contains a cell with nonzero arm 
    excess. Within this column, find the cell with the largest value in its 
    arm. (In French notation this is the topmost cell with nonzero arm 
    excess in the specified column.) Denote the row index and column index of 
    this cell by $r$ and $c$, respectively. Denote the cell as $(r,c)$,  its 
    rightmost arm entry by $a$, and its largest leg entry by $\ell$.
    \item \label{def: uncrowd_bump_4} Look at the column to the right of $(r,c)$ (i.e. column $c+1$) and find the smallest 
    number that is greater than or equal to $a$. 
        \begin{itemize}
            \item If no such number exists, attach an empty cell to the top of column $c+1$ and label the cell as $(\tilde{r}, c+1)$,
             where $\tilde{r}$ is its row index. Let $k$ be the empty character.
            \item If such a number exists, label the value as $k$ and the cell containing $k$ as  $(\tilde{r},c+1)$ where $\tilde{r}$ is the 
            cell's row index.
        \end{itemize}
    We now break into cases:
    \begin{enumerate}
        \item \label{def: uncrowd_bump_a} If $\tilde{r} \not = r$, then remove 
        $a$ from $\sfA_{T}(r, c)$, replace $k$ with $a$, and attach $k$ to the 
        arm of $\sfA_{T}(\tilde{r},c+1)$.
        \item \label{def: uncrowd_bump_b} If $\tilde{r} = r$ then remove $(a, 
        \ell] \cap \sfL_{T}(r, c)$ from $\sfL_{T}(r,c)$ where $(a, \ell] = \{a+1, a+2, \ldots, \ell\}$, remove $a$ from 
        $\sfA_{T}(r,c)$, insert $(a, \ell] \cap \sfL_{T}(r, c)$ into $\sfL_{T}(\tilde{r},c+1)$, replace the hook entry 
        of $(\tilde{r}, c+1)$ with $a$, and attach $k$ to $\sfA_{T}(\tilde{r}, c+1)$.
    \end{enumerate}
    \item \label{def: uncrowd_bump_5} Output the resulting tableau.
\end{enumerate}
See Figures~\ref{fig: vb-a} and \ref{fig: vb-b} for illustration.
\begin{figure}[h!]
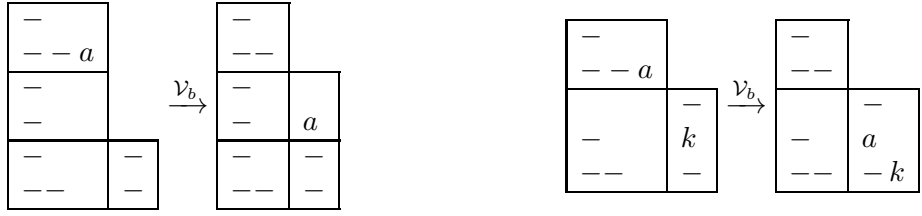

\begin{minipage}{0.45\textwidth}
	\[
		{\def\mc#1{\makecell[lb]{#1}}
		{\begin{array}[lb]{*{2}{|l}|}\cline{1-1}
		\mc{-\\--a}\\\cline{1-1}
		\mc{- \\-}\\\cline{1-2}
		\mc{- \\ - -}&\mc{-\\-}\\\cline{1-2}
		\end{array}} \xrightarrow[]{\mathcal{V}_b}{\begin{array}[lb]{*{2}{|l}|}\cline{1-1}
			\mc{-\\--}\\\cline{1-2}
			\mc{- \\-}&\mc{a}\\\cline{1-2}
			\mc{- \\ - -}&\mc{-\\-}\\\cline{1-2}
			\end{array}}}
	\]
\end{minipage}
\begin{minipage}{0.45\textwidth}
	\[
		{\def\mc#1{\makecell[lb]{#1}}
		{\begin{array}[lb]{*{2}{|l}|}\cline{1-1}
		\mc{-\\--a}\\\cline{1-2}
		\mc{- \\- -}&\mc{-\\k\\-}\\\cline{1-2}
		\end{array}} \xrightarrow[]{\mathcal{V}_b}{\begin{array}[lb]{*{2}{|l}|}\cline{1-1}
			\mc{-\\--}\\\cline{1-2}
			\mc{- \\- -}&\mc{-\\a\\-\,k}\\\cline{1-2}
			\end{array}}}
	\]
\end{minipage}
\caption{When $\tilde{r}\neq r$. Left: $(\tilde{r},c+1)$ is a new cell; Right: $(\tilde{r},c+1)$ is an existing cell.}
\label{fig: vb-a}
\end{figure}
\begin{figure}[h!]
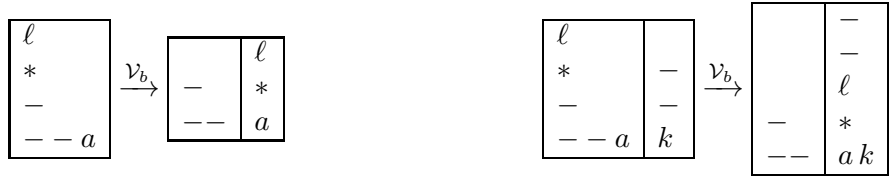

	\begin{minipage}{0.45\textwidth}
		\[
			{\def\mc#1{\makecell[lb]{#1}}
			{\begin{array}[lb]{*{2}{|l}|}\cline{1-1}
			\mc{\ell\\\ast \\-\\--a}\\\cline{1-1}
			\end{array}} \xrightarrow[]{\mathcal{V}_b}{\begin{array}[lb]{*{2}{|l}|}\cline{1-2}
				\mc{- \\--}&\mc{\ell\\\ast\\a}\\\cline{1-2}
				\end{array}}}
		\]
	\end{minipage}
	\begin{minipage}{0.45\textwidth}
		\[
			{\def\mc#1{\makecell[lb]{#1}}
			{\begin{array}[lb]{*{2}{|l}|}\cline{1-2}
			\mc{\ell\\\ast \\-\\--a}&\mc{-\\-\\k}\\\cline{1-2}
			\end{array}} \xrightarrow[]{\mathcal{V}_b}{\begin{array}[lb]{*{2}{|l}|}\cline{1-2}
				\mc{- \\--}&\mc{-\\-\\\ell\\\ast\\a \,k}\\\cline{1-2}
				\end{array}}}
		\]
	\end{minipage}
	\caption{When $\tilde{r}= r$. Left: $(r,c+1)$ is a new cell; Right: $(r,c+1)$ is an existing cell.}
	\label{fig: vb-b}
	\end{figure}
\end{definition}

\begin{lemma} \label{lem: well_defined_v}
The map $\mathcal{V}_b$ is well-defined. More precisely, for $T \in \HVT$ we have $\mathcal{V}_{b}(T) \in \HVT$.
\end{lemma}

\begin{proof}
It suffices to check that $\mathcal{V}_{b}$ preserves the semistandardness condition of both the entire hook-valued 
tableau and the filling within each cell.  We break into two cases depending on whether Step~\ref{def: uncrowd_bump_a} 
or~\ref{def: uncrowd_bump_b} in Definition~\ref{def: uncrowd_bump} is applied.

\begin{description}
	\item[Case 1] Assume Step~\ref{def: uncrowd_bump_a} is applied. To verify 
	semistandardness within each cell, it suffices to check cells $(r, c)$ and $(\tilde{r}, c+1)$. 
	The semistandardness within cell $(r, c)$ is clearly preserved as the only change to the hook-shaped tableau in
	cell $(r, c)$ is that an entry was removed from $\sfA_{T}(r, c)$. We now check the semistandardness condition 
	within cell $(\tilde{r}, c+1)$. We have that $\mathcal{V}_{b}$ either 
	created the cell $(\tilde{r}, c+1)$ and inserted the number $a$ in it or 
	$\mathcal{V}_{b}$ replaced $k$ with $a$ and appended $k$ to the arm of cell
	$(\tilde{r}, c+1)$. In both cases, the tableau in cell 
	$(\tilde{r}, c+1)$ is a semistandard hook-shaped tableau. In the second case this is true since $k$ is weakly greater 
	than $\sfH_{T}(\tilde{r}, c+1)$ and $k$ is the smallest number weakly greater than 
	$a$ in column $c+1$.

	We now check the semistandardness of the entire tableau. Note that 
	it suffices to check the semistandardness in row $\tilde{r}$ and column 
	$c+1$. Since $\tilde{r} < r$, the semistandardness in row $\tilde{r}$ is 
	preserved as $a$ is larger than every number in $(\tilde{r},c)$ and $k$ 
	remains in the same cell. Also, the semistandardness in column $c+1$ is 
	preserved as $k$ is chosen to be the smallest number in column $c+1$ that 
	is weakly greater than $a$.

	\item[Case 2] Assume Step~\ref{def: uncrowd_bump_b} is applied. The 
	semistandardness within cell $(r, c)$ is clearly preserved as the only 
	change to $(r, c)$ is that entries from $\sfL_{T}(r, c)$ and 
	$\sfA_{T}(r,c)$ are removed. We now check the semistandardness condition within 
	cell $(r, c+1)$. If $(a, \ell] \cap \sfL_{T}(r,c) = \emptyset$, then $a$ is weakly larger 
	than all elements of $(r,c)$. In this case, the semistandardness within 
	cell $(r, c+1)$ follows from the argument in Case 1. If $(a, \ell] \cap 
	\sfL_{T}(r,c) \not= \emptyset$, then $a$ is not weakly larger than all 
	elements of $(r,c)$. After applying $\mathcal{V}_{b}$ the semistandardness 
	condition in the leg of $(r, c+1)$ will still hold as $a < x < z$ for all 
	$x \in (a, \ell] \cap \sfL_{T}(r,c)$, where $z$ is the smallest value 
	in $\sfL_{T}(r,c+1)$. Similarly, the semistandardness condition in the arm of 
	$(r,c+1)$ holds as $a < k$ or $k$ is the empty character. Thus, the 
	semistandardness condition in each cell is preserved. The semistandardness 
	of row $r$ is preserved as all numbers strictly greater than $a$ in $(r,c)$ 
	are moved to $(r,c+1)$ along with $a$. The semistandardness condition 
	within column $c+1$ is preserved as every number in $(r+1, c+1)$ is 
	strictly greater than $\ell$ and every number in $(r-1, c+1)$ is strictly 
	less than $a$.
\end{description}
\end{proof}

\begin{definition} \label{def: uncrowd_insert}
The \defn{uncrowding insertion} $\mathcal{V} \colon \HVT \rightarrow \HVT$ is defined as
$\mathcal{V}(T) = \mathcal{V}_b^d(T)$, where the integer $d\geqslant 1$ is minimal such that 
$\shape(\mathcal{V}_b^d(T))/\shape(\mathcal{V}_b^{d-1}(T)) \neq \emptyset$ or 
$\mathcal{V}_b^d(T) = \mathcal{V}_b^{d-1}(T)$.
\end{definition}

A \defn{column-flagged increasing tableau} is a tableau whose transpose is a flagged increasing tableau. 
Let $\hat{\mathcal{F}}$ denote the set of all column-flagged increasing tableaux. Let $\hat{\mathcal{F}}(\mu/\lambda)$ denote the set of 
all column-flagged increasing tableaux of shape $\mu/\lambda$.

\begin{definition} \label{def: uncrowd_map}
Let $T \in \HVT(\lambda)$ with arm excess $\alpha$. The \defn{uncrowding map}
\[
	\mathcal{U} \colon \HVT(\lambda) \rightarrow \bigsqcup_{\mu\supseteq \lambda}\SVT(\mu) \times \hat{\mathcal{F}}(\mu/\lambda)
\]
is defined by the following algorithm:
\renewcommand\labelenumi{(\arabic{enumi})}
\renewcommand\theenumi\labelenumi
\begin{enumerate}
	\item Let $P_{0}=T$ and let $Q_{0}$ be the column-flagged increasing tableau of shape 
	$\lambda/\lambda$.
	\item
	For $1 \leqslant i \leqslant \alpha$, let $P_{i+1} = \mathcal{V}(P_{i})$. Let $c$ be the index of the rightmost 
	column of $P_{i}$ containing a cell with nonzero arm excess and let $\tilde{c}$ be the column index of the cell
	$\shape(P_{i+1})/\shape(P_i)$. Then $Q_{i+1}$ is obtained from $Q_{i}$ by appending the cell 
	$\shape(P_{i+1})/\shape(P_i)$ to $Q_{i}$ and filling this cell with $\tilde{c}- c$.
\end{enumerate}
Define $\mathcal{U}(T) = (P(T), Q(T)) := (P_{\alpha}, Q_{\alpha})$.
\end{definition}

\begin{example}\label{eg: uncrowding}
    Let $T$ be the hook-valued tableau
\[
    {\def\mc#1{\makecell[lb]{#1}}
{\begin{array}[lb]{*{3}{|l}|}\cline{1-1}
\mc{8\\67}\\\cline{1-2}
\mc{5\\4\\233}&\mc{66}\\\cline{1-3}
\mc{1}&\mc{2\\11}&\mc{7\\5}\\\cline{1-3}
\end{array}}}
    \]
    Then, we obtain the following sequence of tableaux $\mathcal{V}_b^i(T)$ for $0\leqslant i\leqslant 2=d$ when computing
    the first uncrowding insertion:

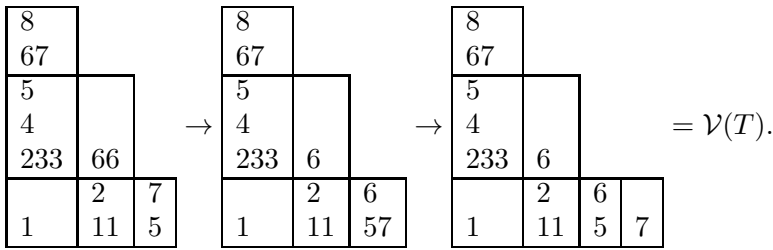
\[
    {\def\mc#1{\makecell[lb]{#1}}
{\begin{array}[lb]{*{3}{|l}|}\cline{1-1}
\mc{8\\67}\\\cline{1-2}
\mc{5\\4\\233}&\mc{66}\\\cline{1-3}
\mc{1}&\mc{2\\11}&\mc{7\\5}\\\cline{1-3}
\end{array}\rightarrow
\begin{array}[lb]{*{3}{|l}|}\cline{1-1}
    \mc{8\\67}\\\cline{1-2}
    \mc{5\\4\\233}&\mc{6}\\\cline{1-3}
    \mc{1}&\mc{2\\11}&\mc{6\\57}\\\cline{1-3}
    \end{array}\rightarrow
    \begin{array}[lb]{*{4}{|l}|}\cline{1-1}
        \mc{8\\67}\\\cline{1-2}
        \mc{5\\4\\233}&\mc{6}\\\cline{1-4}
        \mc{1}&\mc{2\\11}&\mc{6\\5}&\mc{7}\\\cline{1-4}
        \end{array} = \mathcal{V}(T).
}}  
\]
    Continuing with the remaining uncrowding insertions, we obtain the following sequences of tableaux for the
    uncrowding map:
    \[
    \begin{aligned}
    &   {\def\mc#1{\makecell[lb]{#1}}
    {\begin{array}[lb]{*{3}{|l}|}\cline{1-1}
    \mc{8\\67}\\\cline{1-2}
    \mc{5\\4\\233}&\mc{66}\\\cline{1-3}
    \mc{1}&\mc{2\\11}&\mc{7\\5}\\\cline{1-3}
    \end{array}}}
    & \rightarrow\; &
    {\def\mc#1{\makecell[lb]{#1}}
{\begin{array}[lb]{*{4}{|l}|}\cline{1-1}
\mc{8\\67}\\\cline{1-2}
\mc{5\\4\\233}&\mc{6}\\\cline{1-4}
\mc{1}&\mc{2\\11}&\mc{6\\5}&\mc{7}\\\cline{1-4}
\end{array}}}
    & \rightarrow\; &
    {\def\mc#1{\makecell[lb]{#1}}
{\begin{array}[lb]{*{5}{|l}|}\cline{1-1}
\mc{8\\67}\\\cline{1-2}
\mc{5\\4\\233}&\mc{6}\\\cline{1-5}
\mc{1}&\mc{1}&\mc{2\\1}&\mc{6\\5}&\mc{7}\\\cline{1-5}
\end{array}}}
    &\rightarrow\\
    &{\def\mc#1{\makecell[lb]{#1}}
    {\begin{array}[lb]{*{5}{|l}|}\cline{1-2}
    \mc{6}&\mc{8\\7}\\\cline{1-2}
    \mc{5\\4\\233}&\mc{6}\\\cline{1-5}
    \mc{1}&\mc{1}&\mc{2\\1}&\mc{6\\5}&\mc{7}\\\cline{1-5}
    \end{array}}}
    & \rightarrow\; &
    {\def\mc#1{\makecell[lb]{#1}}
{\begin{array}[lb]{*{5}{|l}|}\cline{1-2}
\mc{6}&\mc{8\\7}\\\cline{1-3}
\mc{23}&\mc{5\\4\\3}&\mc{6}\\\cline{1-5}
\mc{1}&\mc{1}&\mc{2\\1}&\mc{6\\5}&\mc{7}\\\cline{1-5}
\end{array}}}
    & \rightarrow\; &
    {\def\mc#1{\makecell[lb]{#1}}
{\begin{array}[lb]{*{6}{|l}|}\cline{1-2}
\mc{6}&\mc{8\\7}\\\cline{1-3}
\mc{2}&\mc{3}&\mc{5\\4\\3}\\\cline{1-6}
\mc{1}&\mc{1}&\mc{2\\1}&\mc{6\\5}&\mc{6}&\mc{7}\\\cline{1-6}
\end{array}}} & = P(T),
    \end{aligned}\]
    \[
    \ytableausetup{notabloids,boxsize=1.5em}
    \begin{aligned}
    &\raisebox{5mm}{\begin{ytableau}
        *(gray)\\
        *(gray) & *(gray)\\
        *(gray) & *(gray) & *(gray)
        \end{ytableau}} & \rightarrow\; &
    \raisebox{5mm}{\begin{ytableau}
        *(gray)\\
        *(gray) & *(gray)\\
        *(gray) & *(gray) & *(gray) & 2
        \end{ytableau}} & \rightarrow\; &
    \raisebox{5mm}{\begin{ytableau}
        *(gray)\\
        *(gray) & *(gray)\\
        *(gray) & *(gray) & *(gray) & 2 & 3
        \end{ytableau}} & \rightarrow \\
    &\raisebox{5mm}{\begin{ytableau}
        *(gray) & 1\\
        *(gray) & *(gray)\\
        *(gray) & *(gray) & *(gray) & 2 & 3
        \end{ytableau}} & \rightarrow\; &
    \raisebox{5mm}{\begin{ytableau}
        *(gray) & 1\\
        *(gray) & *(gray) & 2\\
        *(gray) & *(gray) & *(gray) & 2 & 3
        \end{ytableau}} & \rightarrow\; &
    \raisebox{5mm}{\begin{ytableau}
        *(gray) & 1\\
        *(gray) & *(gray) & 2\\
        *(gray) & *(gray) & *(gray) & 2 & 3 & 5
        \end{ytableau}} & = Q(T).
    \end{aligned}\]

\end{example}
\begin{corollary} \label{well_defined_p}
Let $T \in \HVT$. Then $P(T)$ is a set-valued tableau.
\end{corollary}

\begin{proof}
By Lemma \ref{lem: well_defined_v} and Definition \ref{def: uncrowd_insert}, we 
have that $\mathcal{V}(T)$ is a hook-valued tableau. Note that if 
the arm excess of $T$ is nonzero, then the arm excess of $\mathcal{V}(T)$ is 
one less than that of $T$. Since $P(T) = \mathcal{V}^{\alpha}(T)$, where 
$\alpha$ is the arm excess of $T$, we have that the arm excess of $P(T)$ is 
zero. Thus, $P(T)$ is a set-valued tableau.
\end{proof}

\begin{definition}
Let $T\in \HVT$ and let $d$ be minimal such that $\mathcal{V}(T) = \mathcal{V}_{b}^d(T)$. 
The \defn{insertion path} $p$ of $T\to \mathcal{V}(T)$ is defined as follows:
\begin{itemize}
	\item If $d = 0$, set $p = \emptyset$.
	\item Otherwise, let $(r_0,c_0)$ be the rightmost and topmost cell of $T$ containing 
	a cell with nonzero arm excess. For all 
	$1 \leqslant j \leqslant d$, let $c_{j} = c_{0}+j$ and let $r_{j} = 
	\tilde{r}$ be
	$\tilde{r}$ in Definition~\ref{def: uncrowd_bump} when $\mathcal{V}_{b}$ is applied to 
	$\mathcal{V}_{b}^{j-1}(T)$. Set $p = ((r_{0}, c_{0}), (r_{1}, c_{1}), \ldots, (r_{d}, c_{d}))$.
\end{itemize}
\end{definition}

\begin{lemma} \label{well_defined_q}
Let $T \in \HVT$. Then $Q(T)$ is a column-flagged increasing tableau.
\end{lemma}

\begin{proof}
By construction, the positive integer entries in column $i$ of $Q(T)$ 
are at most $i-1$. Let $m$ be the smallest nonnegative integer such that 
$\mathcal{V}^{m}(T) = P(T)$. Let $p^{i} = ((r_{0}^{i}, c_{0}^{i}), (r_{1}^{i}, c_{1}^{i}), 
\ldots, (r_{d_i}^{i}, c_{d_i}^{i}))$ for $0\leqslant i<m$ be the insertion path of
$\mathcal{V}^i(T) \to \mathcal{V}^{i+1}(T)$. Since $c_{0}^{i+1} \leqslant c_{0}^{i}$ for all 
$0\leqslant i <m$, the entries in each row of $Q(T)$ are strictly increasing.
To check that the entries in each column of $Q(T)$ are strictly increasing, it 
suffices to show that if $c_{0}^{i+1} = c_{0}^{i}$ then $p^{i+1}$ lies weakly 
below $p^{i}$. In other words, it suffices to check that $c_{0}^{i+1} = 
c_{0}^{i}$ implies that $r_{j}^{i+1}\leqslant r_{j}^{i}$ for all $0 \leqslant j \leqslant d_i$.
We prove this by induction on $j$. Note that $r_{0}^{i+1} \leqslant r_{0}^{i}$ by the 
definition of $\mathcal{U}$. Assume by induction that $r_{j}^{i+1} \leqslant r_{j}^{i}$. This 
implies that the $a$ when applying $\mathcal{V}_{b}$ to 
$\mathcal{V}_{b}^{j}(\mathcal{V}^{i}(T))$ is weakly smaller than the $a$ 
when applying $\mathcal{V}_{b}$ to $\mathcal{V}_{b}^{j}(\mathcal{V}^{i-1}(T))$. 
Thus, we must have $r_{j+1}^{i+1} \leqslant r_{j+1}^{i}$.
\end{proof}

\subsection{Properties of the uncrowding map}
\label{section.uncrowding properties}

Let $T$ be a hook-valued tableau. Define $R_{i}(T)$ as the induced subword of $R(T)$ consisting only
of the letters $i$ and $i+1$. In the next lemma, we use the same notation as in Definition~\ref{def: uncrowd_bump}. Furthermore,
two words are \defn{Knuth equivalent} if one can be transformed to the other by a sequence of Knuth equivalences on three consecutive letters
\[
	xzy \equiv zxy \quad \text{for $x\leqslant y<z$,} \qquad yxz \equiv yzx \quad \text{for $x<y\leqslant z$.}
\]

\begin{lemma}\label{lem: knuth}
For $T\in \HVT$, $R_{i}(T) = R_{i}(\mathcal{V}_{b}(T))$ unless $T$ satisfies one of the following three conditions:
\begin{enumerate}
	\item[(a)] $a = i$ or $a = i+1$ and column $c+1$ contains both an $i$ and an $i+1$,
	\item[(b)] $\tilde{r} = r$, $i \in (a, \ell] \cap \sfL_{T}(r, c)$, $k = i$, and 
	column 
	$c+1$ contains an $i+1$,
	\item[(c)] $\tilde{r} = r$, $a = i$, $i+1 \in (a, \ell] \cap \sfL_{T}(r, c)$, 
	and 
	$(r, c)$ contains another $i$ besides $a$.
\end{enumerate}
Moreover, $R_{i}(T)$ is Knuth equivalent to $R_{i}(\mathcal{V}_{b}(T))$.
\end{lemma}

\begin{proof}
Let $R_{i}(T) = r_{1} r_{2} \ldots r_{m}$. We break into cases based on the value of $a$.

\medskip \noindent
\textbf{Case 1:} Assume $a \not = i, i+1$. \\
	Assume Step~\ref{def: uncrowd_bump_a} is applied by $\mathcal{V}_{b}$. If 
	$k \not = i, i+1$, then $R_{i}(T) = R_i(\mathcal{V}_{b}(T))$ as the 
	position of all letters $i$ and $i+1$ remains the same. Let $k = i$. We have 
	that $k$ is the only $i$ in column $c+1$. Hence, when $k$ gets bumped from 
	$\sfL_{T}(\tilde{r}, c+1)$ and appended to $\sfA_{T}(\tilde{r}, c+1)$, the 
	relative position of $k$ to the other letters $i$ and $i+1$ in $R_{i}(T)$ does 
	not change. Thus, $R_{i}(T) = R_i(\mathcal{V}_{b}(T))$. Let $k = i+1$. Note 
	that column $c+1$ cannot have a cell containing an $i$ as $k$ is the 
	smallest number weakly greater than $a$. Hence, moving $k$ from 
	$\sfL_{T}(\tilde{r}, c+1)$ to $\sfA_{T}(\tilde{r}, c+1)$ will not change 
	$R_i(T)$. Therefore, we once again have that $R_{i}(T) = 	R_i(\mathcal{V}_{b}(T))$.

	Assume Step~\ref{def: uncrowd_bump_b} is applied by $\mathcal{V}_{b}$. 
	Consider the subcase when $(a, \ell] \cap \sfL_{T}(r, c) = \emptyset$. By a 
	similar argument to the previous paragraph, we have that $R_{i}(T) = 
	R_i(\mathcal{V}_{b}(T))$. Next, consider the subcase when $i+1 \in (a, \ell] 
	\cap \sfL_{T}(r, c)$. This implies that $a < i$ and the only time $i+1$ 
	occurs in column $c$ is in $\sfL_{T}(r, c)$. Note that if an $i$ exists in column 
	$c$, it must be contained in $\sfL_{T}(r, c)$. We also have that $k \geqslant i+1$ 
	or $k$ is the empty character
	and no cell in column $c+1$ contains an $i$. Thus, removing $(a, \ell] \cap 
	\sfL_{T}(r, c)$ from $\sfL_{T}(r, c)$, replacing $k$ with $(a, \ell] \cap 
	\sfL_{T}(r, 	c)$ in $\sfL_{T}(r, c+1)$, and appending $k$ to $\sfA_{T}(r, c+1)$ does not 
	change $R_i(T)$. Therefore $R_{i}(T) = R_i(\mathcal{V}_{b}(T))$. Let $i \in 
	(a, \ell] \cap \sfL_{T}(r, c)$ and $i+1 \not \in (a, \ell] \cap \sfL_{T}(r, c)$. 
	Note that the only place $i+1$ can occur in column $c$ is as $\sfH_{T}(r+1, c)$ 
	and the only place $i$ can occur in column $c$ is in $\sfL_{T}(r, c)$. This 
	implies that removing $(a, \ell] \cap \sfL_{T}(r, c)$ from $\sfL_{T}(r, c)$, 
	replacing $k$ with $(a, \ell] \cap \sfL_{T}(r, c)$ in $\sfL_{T}(r, c+1)$ 
	and appending $k$ to $\sfA_{T}(r, c+1)$ will not change $R_i(T)$ unless both 
	$i+1$ and $i$ show up in column $c+1$. This can only occur when $k = i$ which 
	implies that $R_i(T) = r_{1} \ldots i \,\, i+1 \,\, k \ldots r_{m}$ and 
	$R_i(\mathcal{V}_{b}(T)) = r_{1} \dots i+1 \, \, i \, \, k \ldots r_{m}$. 
	We see that $R_i(T)$ and $R_i(\mathcal{V}_{b}(T))$ only differ by a Knuth 
	relation implying they are Knuth equivalent. Assume that $i, i+1 \not \in 
	(a, \ell] \cap \sfL_{T}(r, c) \not = \emptyset$. If $a > i+1$ the positions of all 
	letters $i$ and $i+1$ remain the same after $\mathcal{V}_{b}$ is applied. If $a 
	< i$, then the positions of all letters $i$ and $i+1$ also remain the same 
	unless $k = i$ or $k = i+1$. In both of these special subcases, it can be checked
	that still $R_{i}(T) = R_i(\mathcal{V}_{b}(T))$.

\medskip \noindent
\textbf{Case 2:} Assume $a = i$. \\
	Assume Step~\ref{def: uncrowd_bump_a} is applied by $\mathcal{V}_{b}$. If 
	column $c+1$ does not contain both an $i$ and an $i+1$, then we have
	$R_{i}(T) = R_i(\mathcal{V}_{b}(T))$. However, if both an $i$ and an $i+1$ are in column 
	$c+1$, then $R_{i}(T) = r_{1} \dots i \, \, i+1 \, \, i \ldots r_{m}$ and 
	$R_{i}(\mathcal{V}_{b}(T)) = r_{1} \dots i+1 \, \, i \, \, i \ldots r_{m}$ which are Knuth equivalent.

	Assume Step~\ref{def: uncrowd_bump_b} is applied by $\mathcal{V}_{b}$. 
	Consider the subcase when $(a, \ell] \cap \sfL_{T}(r, c) = \emptyset$. By a 
	similar argument to the previous paragraph, we have that $R_{i}(T) = 
	R_i(\mathcal{V}_{b}(T))$ unless both an $i$ and an $i+1$ are in column 
	$c+1$ in which case $R_{i}(T)$ and $R_{i}(\mathcal{V}_{b}(T))$ are only 
	Knuth equivalent. Consider the subcase given by $i+1 \in (a, \ell] \cap 
	\sfL_{T}(r, c)$. Note that no cell in column $c+1$ can contain an $i$, the 
	only cell that could contain an $i+1$ in column $c+1$ is $(r, c+1)$, and 
	the only cell containing letters $i$ or $i+1$ in column $c$ is $(r, c)$. This 
	implies that it suffices to look at the changes to $(r, c)$ and $(r, c+1)$. 
	We see that $R_{i}(T) = r_{1} \ldots i +1 \,\, \underbrace{i \ldots i \,
	a}_{\gamma} \ldots r_{m}$ and $R_{i}(\mathcal{V}_{b}(T)) = r_{1} \ldots 
	\underbrace{i \ldots i}_{\gamma-1} \,\, i+1 \,\, a $ where $\gamma \geqslant 1$ 
	is the number of letters $i$ in cell $(r, c)$ including $a$. We see that 
	$R_{i}(T)$ and $R_{i}(\mathcal{V}_{b}(T))$ are Knuth equivalent. Consider 
	the subcase when $i+1\not \in (a, \ell] \cap \sfL_{T}(r, c) \not = \emptyset$. 
	We have that both $i$ and $i+1$ cannot be in a cell in column $c+1$ and an 
	$i+1$ cannot be in column $c$. Thus applying $\mathcal{V}_{b}$ does not 
	change $R_{i}(T)$ giving us that $R_{i}(T) = R_i(\mathcal{V}_{b}(T))$.

\medskip\noindent
\textbf{Case 3:} Assume $a=i+1$.\\
	Assume Step~\ref{def: uncrowd_bump_a} is applied by $\mathcal{V}_{b}$. 
	If column $c+1$ does not contain both $i$ and $i+1$, then we have 
	that $R_{i}(T) = R_i(\mathcal{V}_{b}(T))$. However, if both $i$ and $i+1$ occur
	in column $c+1$, then $R_{i}(T) = r_{1} \dots i+1 \, \, i+1 \, \, i \ldots r_{m}$ 
	and $R_{i}(\mathcal{V}_{b}(T)) = r_{1} \dots i+1 \, \, i \, \, i+1 \ldots r_{m}$ which are 
	Knuth equivalent.

	Assume Step~\ref{def: uncrowd_bump_b} is applied by $\mathcal{V}_{b}$. If 
	$(a, \ell] \cap \sfL_{T}(r, c) = \emptyset$, then $R_{i}(T) = 
	R_i(\mathcal{V}_{b}(T))$ unless both $i$ and $i+1$ occur in column 
	$c+1$. In this exceptional case, we have that $R_{i}(T)$ and 
	$R_{i}(\mathcal{V}_{b}(T))$ are only Knuth equivalent by a similar argument 
	to the previous paragraph. If $(a, \ell] \cap \sfL_{T}(r, c) \not = \emptyset$, 
	then $k > i+1$ or $k$ is the empty character
	and no cell in column $c+1$ contains an $i+1$. Thus applying 
	$\mathcal{V}_{b}$ does not change $R_{i}(T)$ giving us that $R_{i}(T) = 
	R_i(\mathcal{V}_{b}(T))$.
\end{proof}

\begin{remark}
In general, the full reading words are not Knuth equivalent under the 
uncrowding map. For example, take the following hook-valued tableau $T$, which 
uncrowds to a set-valued tableau $S$: 
\[	
{\def\mc#1{\makecell[lb]{#1}}
	{T = \begin{array}[lb]{*{2}{|l}|}\cline{1-2}
	\mc{4\\3\\2\\1 2}&\mc{5\\4}\\\cline{1-2}
	\end{array}} \rightarrow
	{\begin{array}[lb]{*{3}{|l}|}\cline{1-3}
	\mc{2\\1}&\mc{4\\3\\2}&\mc{5\\4}\\\cline{1-3}
	\end{array} = S.}
	}
	\]
The reading word changed from $4321254$ to $2143254$, which are not Knuth equivalent.
\end{remark}

\begin{proposition}
\label{proposition.main}
Let $T \in \HVT$. 
\begin{enumerate}
\item If $f_i(T) = 0$, then $f_i(P(T)) = 0$.
\item If $e_i(T) = 0$, then $e_i(P(T)) = 0$.
\end{enumerate}
\end{proposition}

\begin{proof}
Since $P(T) = \mathcal{V}_{b}^s(T)$ for some $s \in \mathbb{N}$ and Knuth equivalence 
is transitive, we have that $R_{i}(T)$ is Knuth equivalent to $R_{i}(P(T))$ by the previous lemma. 
As $f_i(T) = 0$, we have that every $i$ in $R_{i}(T)$ is $i$-paired with an $i+1$ to its left. 
This property is preserved under Knuth equivalence giving us that $f_i(P(T)) = 0$.
The same reasoning implies (2).
\end{proof}

\begin{lemma}\label{intertwine_lemma}
Let $T \in \HVT$. 
\begin{enumerate}
\item
If $f_i(T) \not = 0$, then $f_i(\mathcal{V}_{b}(T)) = \mathcal{V}_{b}(f_i(T)) \not = 0$.
\item
If $e_i(T) \not = 0$, then $e_i(\mathcal{V}_{b}(T)) = \mathcal{V}_{b}(e_i(T)) \not = 0$.
\end{enumerate}
\end{lemma}

\begin{proof}
We are going to prove (1). Part (2) follows since $e_i$ and $f_i$ are partial inverses.

Let $a$, $\ell$, $k$, $r$, $c$, and $\tilde{r}$ be defined as in Definition \ref{def: uncrowd_bump} 
when $\mathcal{V}_{b}$ is applied to $T$. Similarly, define $a'$, $\ell'$, 
$k'$, $r'$, $c'$, and $\tilde{r}'$ for when $\mathcal{V}_{b}$ is applied to 
$f_{i}(T)$. Let $R_{i}(T) = r_{1} r_{2} \ldots r_{m}$ and $R_{i}(\mathcal{V}_{b}(T)) = r_{1}' r_{2}' \ldots r_{m}'$
be the corresponding reading words. Let $(\hat{r}, 
\hat{c})$ denote the cell containing the rightmost unpaired $i$ in $T$, where 
$\hat{r}$ and $\hat{c}$ are its row and column index respectively. We break 
into cases based on the position of $(\hat{r}, \hat{c})$ to $(r,c)$.
\begin{description}
	\item[Case 1] Assume $(\hat{r}, \hat{c}) = (r, c)$. We break into subcases based on how $f_i$ acts on $T$.
		\begin{itemize}
			\item Assume that $(r+1, c)$ contains an $i+1$.\\
As every entry in $(r,c)$ must be strictly smaller than the values in $(r+1, 
c)$ and $(r,c)$ must contain an $i$, we have that $\ell = i$ or $a = i$.
If $\ell = i$, then $\ell$ is $i$-paired with the $i+1$ in $(r+1,c)$. 
Hence $a$ is always equal to $i$ and $a$ must correspond to the 
rightmost unpaired $i$ of $T$. Thus, $f_{i}$ acts on $T$ by removing $a$ from 
$(r, c)$ and appending an $i+1$ to $\sfA_{T}(r+1, c)$. Note that $(a, \ell] 
\cap 
\sfL_{T}(r, c) = \emptyset$ implying $\mathcal{V}_{b}$ acts on $T$ by removing 
$a$ 
from $\sfA_{T}(r, c)$, replacing $k$ in $(\tilde{r}, c+1)$ with $a$, and 
appending 
$k$ to $\sfA_{T}(\tilde{r}, c+1)$ where $\tilde{r} \leqslant r$. We break into 
subcases 
based upon where the values of $i$ and $i+1$ are in column $c+1$ utilizing the 
fact that column $c+1$ cannot contain an $i$ without an $i+1$ (since the arm excess of cell
$(r+1,c)$ is zero and cell $(r,c)$ contains the rightmost unpaired $i$).\\

Assume that column $c+1$ does not contain an $i$. Since $a$ corresponds to 
the rightmost unpaired $i$ in $T$ and column $c+1$ does not contain an $i$, we 
have that the rightmost unpaired $i$ in $\mathcal{V}_{b}(T)$ is precisely $a$ 
in the cell $(\tilde{r}, c+1)$. Note that $(\tilde{r}+1, c+1)$ does not contain 
an $i+1$ in $\mathcal{V}_{b}(T)$ as $k \geqslant i+1$ or $k$ is the empty character. 
Similarly, we have that $(\tilde{r}, c+2)$ does not contain an $i$. Thus, $f_i$ 
acts on $\mathcal{V}_{b}(T)$ by changing $a$ to an $i+1$ in $(\tilde{r}, c+1)$. 
We now consider $\mathcal{V}_{b}(f_{i}(T))$. When applying $\mathcal{V}_{b}$ to 
$f_{i}(T)$, $a'$ is precisely the $i+1$ appended to $\sfA_{T}(r+1, c)$ and $k'$ 
is the same as $k$. Since $\tilde{r}' = \tilde{r} < r+1$, we have that 
$\mathcal{V}_{b}$ acts on $f_{i}(T)$ by removing $i+1$ from 
$\sfA_{f_{i}(T)}(r+1, 
c)$, replacing $k$ with an $i+1$ in $(\tilde{r}, c+1)$, and appending $k$ to 
$\sfA_{f_{i}(T)}(\tilde{r}, c+1)$. We see that $f_i(\mathcal{V}_{b}(T)) = 
\mathcal{V}_{b}(f_i(T))$.\\

Assume that column $c+1$ contains both an $i$ and an $i+1$ in the same cell. 
Note that this implies that $k = i$. Since $a$ is the rightmost unpaired $i$ 
in $T$ and the only cell in column $c+1$ that contained an $i+1$ or an $i$ is 
$(\tilde{r}, c+1)$, we have that the rightmost unpaired $i$ in 
$\mathcal{V}_{b}(T)$ is the $i$ appended to $\sfA_{T}(\tilde{r}, c+1)$. Since 
$(\tilde{r}, c+1)$ contains an $i+1$, we have that $(\tilde{r}+1, c+1)$ cannot 
contain an $i+1$ and $(\tilde{r}, c+2)$ cannot contain an $i$. Thus, $f_{i}$ 
acts on $\mathcal{V}_{b}(T)$ by changing the $i$ in 
$\sfA_{\mathcal{V}_{b}(T)}(\tilde{r}, c+1)$ to an $i+1$. We now consider 
$\mathcal{V}_{b}(f_{i}(T))$. When applying $\mathcal{V}_{b}$ to $f_{i}(T)$, 
$a'$ is precisely the $i+1$ appended to $\sfA_{T}(r+1, c)$ and $k'$ is the 
$i+1$ 
in $(\tilde{r}, c+1)$. Since $\tilde{r}' = \tilde{r} < r+1$, we have that 
$\mathcal{V}_{b}$ acts on $f_{i}(T)$ by removing $i+1$ from 
$\sfA_{f_{i}(T)}(r+1, c)$, replacing $i+1$ in $(\tilde{r}, c+1)$ with the $i+1$ from 
$\sfA_{f_{i}(T)}(r+1, c)$, and appending an $i+1$ to 
$\sfA_{f_{i}(T)}(\tilde{r}, 
c+1)$. We see that $f_i(\mathcal{V}_{b}(T)) = \mathcal{V}_{b}(f_i(T))$.\\

Assume that column $c+1$ contains both an $i$ and an $i+1$ in different cells. 
Note that this implies that $k = i$. Since $a$ corresponds to the rightmost 
unpaired $i$ in $R_{i}(T)$ and the only $i+1$ and $i$ in column $c+1$ are in 
cells $(\tilde{r}+1, c+1)$ and $(\tilde{r}, c+1)$ respectively, we have that the 
rightmost unpaired $i$ in $R_i(\mathcal{V}_{b}(T))$ corresponds to the $i$ 
appended to $\sfA_{T}(\tilde{r}, c+1)$. By assumption, we have that 
$(\tilde{r}+1, c+1)$ contains an $i+1$. Thus, $f_{i}$ acts on 
$\mathcal{V}_{b}(T)$ by removing the $i$ from 
$\sfA_{\mathcal{V}_{b}(T)}(\tilde{r}, c+1)$ and appending an $i+1$ to 
$\sfA_{\mathcal{V}_{b}(T)}(\tilde{r}+1, c+1)$. We now consider 
$\mathcal{V}_{b}(f_{i}(T))$. When applying $\mathcal{V}_{b}$ to $f_{i}(T)$, 
$a'$ is precisely the $i+1$ appended to $\sfA_{T}(r+1, c)$ and $k'$ is the 
$i+1$ 
in cell $(\tilde{r}+1, c+1)$. If $\tilde{r}' = r+1$, then $i+1$ is weakly 
larger than every value in $(r+1, c)$. Thus, either $(a', \ell'] \cap 
\sfL_{f_{i}(T)}(r+1, c) = \emptyset$ or $\tilde{r}' < r+1$. This implies that 
$\mathcal{V}_{b}$ acts on $f_{i}(T)$ by removing $i+1$ from 
$\sfA_{f_{i}(T)}(r+1, 
c)$, replacing the $i+1$ in $\sfH_{f_{i}(T)}(\tilde{r}+1, c+1)$ with the $i+1$ 
removed from $\sfA_{f_{i}(T)}(r+1, c)$, and appending an $i+1$ to 
$\sfA_{f_{i}(T)}(\tilde{r}+1, c+1)$. We see that $f_i(\mathcal{V}_{b}(T)) = 
\mathcal{V}_{b}(f_i(T))$.\\

\item Assume that $(r+1, c)$ does not contain an $i+1$ and $(r, c+1)$ contains an $i$.\\
Under these assumptions, we have that no cell in column $c$ can contain an 
$i+1$. This implies that column $c+1$ must contain an $i+1$. The cell $(r+1, 
c+1)$ cannot have an $i+1$ as this would force $(r+1, c)$ to also have an 
$i+1$. Thus, $(r, c+1)$ must contain an $i+1$ in its leg. By our assumption we 
have that $f_{i}$ acts on $T$ by removing the $i$ from $(r, c+1)$ and appending 
an $i+1$ to $\sfL_{T}(r, c)$. We break into subcases according to where the 
rightmost unpaired $i$ sits inside the cell $(r, c)$. If the rightmost unpaired 
$i$ is in $\sfH_{T}(r, c)$, then $a\geqslant i$ which would either contradict the hook 
entry being the rightmost unpaired $i$ or cell $(r, c+1)$ containing an $i$. 
Thus, we only need to consider the subcases where the rightmost unpaired $i$ is 
either in the leg or arm of $(r, c)$.\\

Assume that the rightmost unpaired $i$ is in $\sfL_{T}(r, c)$ for this entire 
paragraph. This implies that $\ell = i$. Since $(r, c+1)$ contains an $i$, we have 
that $a < i$. If $\tilde{r} < r$, then $\mathcal{V}_{b}$ acts on $T$ by 
removing $a$ from $(r,c)$, replacing $k$ with $a$ in $(\tilde{r}, c+1)$, and 
appending $k$ to $\sfA_{T}(\tilde{r}, c+1)$. Since $a, k < i$, we have that 
$\mathcal{V}_{b}$ does not change position of the rightmost unpaired $i$. Note 
that $(r+1, c)$ still does not contain an $i+1$ while $(r, c+1)$ still contains 
an $i$. Thus, $f_{i}$ acts on $\mathcal{V}_{b}(T)$ by removing the $i$ from 
$(r, c+1)$ and appending an $i+1$ to $\sfL_{\mathcal{V}_{b}(T)}(r, c)$. We now 
consider $\mathcal{V}_{b}(f_{i}(T))$. Note that $(r', c')$, $a'$, and $k'$ are 
the same as $(r, c)$, $a$, and $k$ respectively. Thus, $\mathcal{V}_{b}$ acts in 
the same way as before. This gives us that $f_i(\mathcal{V}_{b}(T)) = 
\mathcal{V}_{b}(f_i(T))$. If $\tilde{r} = r$, then $k$ is precisely the $i$ in 
cell $(r, c+1)$. We see that $\mathcal{V}_{b}$ acts on $T$ by removing $(a, 
i]\cap \sfL_{T}(r, c)$ from $\sfL_{T}(r, c)$ and $a$ from $\sfA_{T}(r,c)$, 
replacing 
$k$ with $\left((a, i]\cap \sfL_{T}(r, c)\right) \cup \{ a \}$, and appending $k$ to 
$\sfA_{T}(r+1, c)$. Since there is an $i+1$ in $\sfL_{\mathcal{V}_{b}(T)}(r, 
c+1)$, we see that the rightmost unpaired $i$ in $\mathcal{V}_{b}(T)$ is precisely $k$ 
in $\sfA_{\mathcal{V}_{b}(T)}(r, c+1)$. Note that $(r+1, c+1)$ does not contain an 
$i+1$ and $(r, c+2)$ does not contain an $i$ because $(r, c+1)$ contains an 
$i+1$. Thus, $f_i$ acts on $\mathcal{V}_{b}(T)$ by changing the $i$ in 
$\sfA_{\mathcal{V}_{b}(T)}(r, c+1)$ to an $i+1$. We now consider 
$\mathcal{V}_{b}(f_{i}(T))$. We have that $a'$ is the same as $a$ and $k'$ is 
the $i+1$ in $(r, c+1)$. We have $(a', \ell'] \cap \sfL_{f_{i}(T)}(r', c') = \{ i+1 
\} \cup ((a, i]\cap \sfL_{T}(r, c))$. This implies that $\mathcal{V}_{b}$ acts on 
$f_{i}(T)$ by removing $\{ i+1 \} \cup ((a, i]\cap \sfL_{T}(r, c))$ from 
$\sfL_{f_{i}(T)}(r, c)$ and $a$ from $\sfA_{f_{i}(T)}(r, c)$, replacing $i+1$ 
with 
$\{ i+1 \} \cup ((a, i]\cap \sfL_{T}(r, c)) \cup \{ a \}$ in $(r, c+1)$, and 
appending an $i+1$ to $\sfA_{f_{i}(T)}(r, c+1)$. We see that 
$f_i(\mathcal{V}_{b}(T)) = \mathcal{V}_{b}(f_i(T))$.\\

Assume that the rightmost unpaired $i$ is in $\sfA_{T}(r, c)$. This implies 
that 
$a = i$ and forces $a$ to correspond to the rightmost unpaired $i$. We also 
have that $k$ is the $i$ in $(r, c+1)$. Since $i$ is weakly greater than all 
values in $(r, c)$, we have that $(a, \ell]\cap \sfL_{T}(r, c) = \emptyset$. 
Thus, 
$\mathcal{V}_{b}$ acts on $T$ by removing $a$ from $(r, c)$, replacing $k$ with 
$a$ in $(r, c+1)$, and appending $k$ to $\sfA_{T}(r, c+1)$. Since $a$ was the 
rightmost unpaired $i$ in $T$ and cell $(r, c+1)$ contains an $i+1$ in its leg, 
we have that the rightmost unpaired $i$ in $\mathcal{V}_{b}(T)$ is $k$ in 
$\sfA_{\mathcal{V}_{b}(T)}(r, c+1)$. As $i+1$ is in $(r, c+1)$, we have that 
$(r+1, c+1)$ cannot contain an $i+1$ and $(r, c+2)$ cannot contain an $i$. This 
implies that $f_i$ acts on $\mathcal{V}_{b}(T)$ by changing the $i$ in 
$\sfA_{\mathcal{V}_{b}(T)}(r, c+1)$ to an $i+1$. We now consider 
$\mathcal{V}_{b}(f_{i}(T))$. We have that $a'$ is the same as $a$ and $k'$ is 
equal to the $i+1$ in $(r, c+1)$. Note that $(a',  \ell'] \cap \sfL_{T}(r, c) = 
\{ i+1 \}$. This implies that $\mathcal{V}_{b}$ acts on $f_{i}(T)$ by removing 
$i+1$ from $\sfL_{f_{i}(T)}(r, c)$ and $a$ from $\sfA_{f_{i}(T)}(r, c)$, 
replacing the $i+1$ in $(r, c+1)$ with $\{i+1, a\}$, and appending an $i+1$ to 
$\sfA_{f_{i}(T)}(r, c+1)$. We see that $f_i(\mathcal{V}_{b}(T)) = 
\mathcal{V}_{b}(f_i(T))$.\\

\item Assume that $(r+1, c)$ does not contain an $i+1$ and $(r, c+1)$ does not contain an $i$.\\
We break into subcases based on where the rightmost unpaired $i$ sits inside $(r, c)$.\\

Assume that the rightmost unpaired $i$ is in the hook entry of $(r, c)$ for the 
remainder of this paragraph. Note that this implies that $a > i$ and the 
rightmost unpaired $i$ in $\mathcal{V}_{b}(T)$ is still the hook entry of $(r, 
c)$. We see that $\mathcal{V}_{b}$ does not insert an $i+1$ into $(r+1, c)$ nor 
an $i$ into $(r, c+1)$. This implies that $f_{i}$ acts on $T$ and 
$\mathcal{V}_{b}(T)$ in the same way by changing the hook entry of $(r, c)$ 
into an $i+1$. Next, we note that $(r', c')$, $a'$, $k'$, and $(a',  \ell'] 
\cap 
\sfL_{f_{i}(T)}(r', c')$ are the same as $(r, c)$, $a$, $k$, and $(a, \ell] 
\cap 
\sfL_{T}(r, c)$ respectively. Thus, $\mathcal{V}_{b}$ acts on $T$ and 
$f_{i}(T)$ 
in the same manner without affecting the hook entry of $(r, c)$. Therefore, we 
have that the actions of $f_{i}$ and $\mathcal{V}_{b}$ on $T$ are independent 
and $f_i(\mathcal{V}_{b}(T)) = \mathcal{V}_{b}(f_i(T))$.\\

Assume that the rightmost unpaired $i$ is in the leg of $(r,c)$ for the 
remainder of this paragraph. This implies that $a \not = i$. First, we assume 
that $a > i$ or $\tilde{r} < r$. Under this extra assumption, we observe that 
the action of $\mathcal{V}_{b}$ does not change the position of the rightmost 
unpaired $i$. Also, $\mathcal{V}_{b}$ does not insert an $i+1$ into $(r+1, c)$ 
nor an $i$ into $(r, c+1)$. We see that $f_{i}$ acts on $T$ and 
$\mathcal{V}_{b}(T)$ in the same way by changing the $i$ in the leg of $(r, c)$ 
into an $i+1$. Next, we note that $(r', c')$, $a'$, and $k'$ are the same as 
$(r, c)$, $a$, and $k$ respectively. If $a >i$, we have that $a \geqslant i+1$ 
implying that $(a', \ell'] \cap \sfL_{f_{i}(T)}(r', c') = (a, \ell] \cap 
\sfL_{T}(r, 
c)$. 
Thus, either $(a', \ell'] \cap \sfL_{f_{i}(T)}(r', c') = (a, \ell] \cap 
\sfL_{T}(r, 
c)$ 
or 
$\tilde{r} < r$. This implies that $\mathcal{V}_{b}$ acts on $T$ and $f_{i}(T)$ 
in the same manner and does not affect the $i$ or $i+1$ in the leg of $(r, c)$. 
Therefore, we have that the actions of $f_{i}$ and $\mathcal{V}_{b}$ on $T$ are 
independent and $f_i(\mathcal{V}_{b}(T)) = \mathcal{V}_{b}(f_i(T))$. Next, 
assume that $\tilde{r} = r$ and $a < i$. This implies that $(a, \ell] \cap 
\sfL_{T}(r, c) \not = \emptyset$ as $i \in (a, \ell]\cap \sfL_{T}(r, c)$. We 
have 
that 
$\mathcal{V}_{b}$ acts on $T$ by removing $(a, \ell]\cap \sfL_{T}(r, c)$ from $\sfL_{T}(r, c)$ and $a$ from $\sfA_{T}(r, c)$, replacing $k$ with 
$((a, 
l]\cap \sfL_{T}(r, c)) \cup \{ a \}$ in $(r, c+1)$, and appending $k$ to 
$\sfA_{T}(r, c+1)$. By assumption, there was no $i$ in $(r, c+1)$ to begin 
with. 
Thus, we have that the rightmost unpaired $i$ of $\mathcal{V}_{b}(T)$ is the 
$i$ in $(r, c+1)$ that replaced $k$.  Since $k \geqslant i+1$ or $k$ is the empty 
character, we have that the cell $(r+1, c+1)$ does not contain an $i+1$ and the 
cell $(r, c+2)$ does not contain an $i$. Hence, $f_{i}$ acts on 
$\mathcal{V}_{b}(T)$ by replacing the $i$ in $\sfL_{\mathcal{V}_{b}(T)}(r, 
c+1)$ 
with an $i+1$. We now consider $\mathcal{V}_{b}(f_{i}(T))$. We have that 
$f_{i}$ acts on $T$ by changing the $i$ in $\sfL_{T}(r, c)$ to an $i+1$. We see 
that $a'$ and $k'$ are the same as $a$ and $k$ respectively. Since $i>a$, we 
have that $i+1>a$ or in other words $i+1 \in (a', \ell']\cap \sfL_{T}(r, c)$. 
This 
implies that $(a', \ell']\cap \sfL_{f_{i}(T)}(r', c')= (((a', \ell']\cap 
\sfL_{T}(r, 
c)) 
\cup \{ i+1 \}) - \{ i \}$. We have $\mathcal{V}_{b}$ acts on $f_{i}(T)$ by 
removing $(a', \ell']\cap \sfL_{f_{i}(T)}(r, c)$ from $\sfL_{f_{i}(T)}(r, c)$ 
and 
$a$ 
from $\sfA_{f_{i}(T)}(r, c)$, replacing $k$ with $(a', \ell']\cap 
\sfL_{f_{i}(T)}(r, c)$ 
in $(r, c+1)$, and appending $k$ to $\sfA_{f_{i}(T)}(r, c+1)$. We see that 
$f_i(\mathcal{V}_{b}(T)) = \mathcal{V}_{b}(f_i(T))$.\\

Assume that the rightmost unpaired $i$ is in $\sfA_{T}(r,c)$ and $\tilde{r} < 
r$ 
or $(a, \ell]\cap \sfL_{T}(r, c) = \emptyset$ for this entire paragraph. Under 
this 
assumption, $f_{i}$ acts on $T$ by changing the rightmost $i$ in the arm of 
$(r, c)$ to an $i+1$. Also, $\mathcal{V}_{b}$ acts on $T$ by removing $a$ from 
$\sfA_{T}(r, c)$, replacing $k$ in $(\tilde{r}, c+1)$ with $a$, and appending 
$k$ 
to $\sfA_{T}(\tilde{r}, c+1)$. First, we make the additional assumption that 
$i< 
a$. Since we assume the rightmost unpaired $i$ is in the arm of $(r, c)$ and $i 
< a$, we have the rightmost unpaired $i$ in $\mathcal{V}_{b}(T)$ is in the same 
position as in $T$. Note that the cell $(r+1, c)$ still does not contain an 
$i+1$ and the cell $(r, c+1)$ still does not contain an $i$. Thus, we have that 
$f_i$ acts on $\mathcal{V}_{b}(T)$ by changing the rightmost $i$ in 
$\sfA_{\mathcal{V}_{b}}(r, c)$ into an $i+1$. We now consider 
$\mathcal{V}_{b}(f_{i}(T))$. We see that $a'$ and $k'$ are the same as $a$ and 
$k$ respectively. This implies that $\mathcal{V}_{b}$ acts on $f_{i}(T)$ by 
removing $a$ from $(r,c)$, replacing $k$ with $a$ in $(\tilde{r}, c)$, and 
appending $k$ to $\sfA_{f_{i}(T)}(\tilde{r}, c+1)$. We see that 
$f_i(\mathcal{V}_{b}(T)) = \mathcal{V}_{b}(f_i(T))$. Next, we make the 
assumption that $a = i$ and column $c+1$ does not contain both an $i$ and an 
$i+1$. We have that the rightmost unpaired $i$ in $\mathcal{V}_{b}(T)$ is 
precisely the $i$ that replaced $k$ in $(\tilde{r}, c+1)$. We also have that $k 
\geqslant i+1$ or $k$ is the empty character implying that the cell $(\tilde{r}+1, 
c+1)$ does not contain an $i+1$ and the cell $(\tilde{r}, c+2)$ does not 
contain an $i$. This implies that $f_{i}$ acts on $\mathcal{V}_{b}(T)$ by 
changing the $i$ in $\sfL^{+}_{\mathcal{V}_{b}(T)}(\tilde{r}, c+1)$ to an $i+1$. We 
now consider $\mathcal{V}_{b}(f_{i}(T))$. We see that $a'$ is the $i+1$ in $(r, 
c)$ created by appying $f_{i}$ and $k'$ is the same as $k$. Thus, 
$\mathcal{V}_{b}$ acts on $f_{i}(T)$ by removing the $i+1$ from $(r, c)$, 
replacing $k$ with an $i+1$ in $(\tilde{r}, c)$, and appending $k$ to 
$\sfA_{f_{i}(T)}(\tilde{r}, c+1)$. We see that $f_i(\mathcal{V}_{b}(T)) = 
\mathcal{V}_{b}(f_i(T))$. Next, we assume that $a = i$ and column $c+1$ 
contains both an $i$ and an $i+1$ in the same cell. Note that this implies that 
$k = i$. Since $a$ corresponded to the rightmost unpaired $i$ in $T$ and the 
only cell in column $c+1$ that contains an $i+1$ or an $i$ is $(\tilde{r}, 
c+1)$, we have that the rightmost unpaired $i$ in $\mathcal{V}_{b}(T)$ 
corresponds to the $i$ appended to $\sfA_{T}(\tilde{r}, c+1)$. Since 
$(\tilde{r}, 
c+1)$ contains an $i+1$ in $\mathcal{V}_{b}(T)$, we have that $(\tilde{r}+1, c+1)$ 
cannot contain an $i+1$ and $(\tilde{r}, c+2)$ cannot contain an $i$. Thus, 
$f_{i}$ acts on $\mathcal{V}_{b}(T)$ by changing the $i$ in 
$\sfA_{\mathcal{V}_{b}(T)}(\tilde{r}, c+1)$ to an $i+1$. We now consider 
$\mathcal{V}_{b}(f_{i}(T))$. We see that $a'$ is the $i+1$ in $(r, c)$ obtained 
after applying $f_{i}$ and $k'$ is the $i+1$ in cell $(\tilde{r}, c+1)$. This 
implies that $\mathcal{V}_{b}$ acts on $f_{i}(T)$ by removing the $i+1$ from 
$(r, c)$, replacing $k'$ with an $i+1$ in $(\tilde{r}, c+1)$, and appending 
$k'$ to $\sfA_{f_{i}(T)}(\tilde{r}, c+1)$. We see that $f_i(\mathcal{V}_{b}(T)) 
= 
\mathcal{V}_{b}(f_i(T))$. Finally, we make the assumption that $a = i$ and 
column $c+1$ contains both an $i$ and an $i+1$ but in different cells. We once 
again have that $k = i$, but now we have that $(\tilde{r}+1, c+1)$ contains an 
$i+1$. We have that the rightmost unpaired $i$ in $\mathcal{V}_{b}(T)$ is the 
$i$ that was appended to $\sfA_{T}(\tilde{r}, c+1)$. Since $(\tilde{r}+1, c+1)$ 
contains an $i+1$, we have that $f_{i}$ acts on $\mathcal{V}_{b}(T)$ by 
removing the $i$ from $\sfA_{\mathcal{V}_{b}(T)}(\tilde{r}, c+1)$ and appending 
an 
$i+1$ to $\sfA_{\mathcal{V}_{b}(T)}(\tilde{r}+1, c+1)$. We now consider 
$\mathcal{V}_{b}(f_{i}(T))$. We see that $a'$ is the $i+1$ in $(r, c)$ obtained 
after applying $f_{i}$ and $k'$ the $i+1$ in cell $(\tilde{r}+1, c+1)$. This 
implies that $\mathcal{V}_{b}$ acts on $f_{i}(T)$ by removing the $i+1$ from 
$(r, c)$, replacing $k'$ with an $i+1$ in $(\tilde{r}+1, c+1)$, and appending 
$k'$ to $\sfA_{f_{i}(T)}(\tilde{r}+1, c+1)$. We see that 
$f_i(\mathcal{V}_{b}(T)) 
= \mathcal{V}_{b}(f_i(T))$.\\

Assume that the rightmost unpaired $i$ is in the arm of $(r,c)$, $\tilde{r} = 
r$, and $(a, \ell]\cap \sfL_{T}(r, c) \not= \emptyset$ for this entire 
paragraph. 
First, we make the additional assumption that $i < a$. This gives us that 
$\mathcal{V}_{b}(T)$ is attained from $T$ by removing $(a, \ell]\cap 
\sfL_{T}(r, 
c)$ 
from $\sfL_{T}(r, c)$ and $a$ from $\sfA_{T}(r, c)$, replacing $k$ in cell $(r, 
c+1)$ with $((a, \ell]\cap \sfL_{T}(r, c)) \cup \{ a \}$, and appending $k$ to 
$\sfA_{T}(r, c+1)$. Since $k, a >i$, we have that the rightmost unpaired $i$ in 
$\mathcal{V}_{b}(T)$ remains the same as in $T$. We also have that the cell 
$(r+1, c)$ does not contain an $i+1$ and the cell $(r, c+1)$ does not contain 
an $i$. Thus, $f_i$ acts on $\mathcal{V}_{b}(T)$ by changing the rightmost $i$ 
in $\sfA_{\mathcal{V}_{b}(T)}(r, c)$ to an $i+1$. We now consider 
$\mathcal{V}_{b}(f_i(T))$. We have that $f_i$ acts on $T$ by changing the 
rightmost $i$ in $\sfA_{T}(r, c)$ to an $i+1$. We see that $a'$, $k'$, and 
$(a', 
l']\cap \sfL_{f_{i}(T)}(r', c')$ are the same as $a$, $k$, and $(a, \ell]\cap 
\sfL_{T}(r, c)$ respectively. This implies that $\mathcal{V}_{b}$ acts on 
$f_i(T)$ by removing $(a, \ell]\cap \sfL_{T}(r, c)$ from $\sfL_{f_{i}(T)}(r, 
c)$ 
and 
$a$ from $\sfA_{f_{i}(T)}(r, c)$, replacing $k$ in cell $(r, c+1)$ with $((a, 
l]\cap \sfL_{T}(r, c))\cup \{ a \}$, and appending $k$ to $\sfA_{f_{i}(T)}(r, 
c+1)$. 
We see that $f_i(\mathcal{V}_{b}(T)) = \mathcal{V}_{b}(f_i(T))$. Next, we 
assume that $a = i$ and $(r, c)$ contains an $i+1$. Since $a = i$, the $i+1$ in 
$(r, c)$ must be in its leg. Also as $a$ is the rightmost unpaired $i$ of $T$, 
we must have that $(r, c)$ contains another $i$ besides $a$. This gives us that 
$\mathcal{V}_{b}(T) $ is attained from $T$ by removing $(a, \ell]\cap 
\sfL_{T}(r, 
c)$ from $\sfL_{T}(r, c)$ and $a$ from $\sfA_{T}(r, c)$, replacing $k$ in cell 
$(r, 
c+1)$ with $((a, \ell]\cap \sfL_{T}(r, c)) \cup \{ a \}$, and appending $k$ to 
$\sfA_{T}(r, c+1)$. Note that the $i$ inserted into $(r, c+1)$ becomes 
$i$-paired 
while an $i$ in $(r, c)$ becomes unpaired. This implies that the rightmost 
unpaired $i$ in $\mathcal{V}_{b}(T)$ still sits in the cell $(r, c)$. We see 
that the cell $(r+1, c)$ still does not contain an $i+1$; however, the cell 
$(r, c+1)$ now contains an $i$. This implies that $f_i$ acts on 
$\mathcal{V}_{b}(T)$ by removing the $i$ from the cell $(r, c+1)$ and appending 
an $i+1$ to $\sfL_{\mathcal{V}_{b}(T)}(r, c)$. We now consider 
$\mathcal{V}_{b}(f_i(T))$. We have that $f_i$ acts on $T$ by changing $a$ into 
an $i+1$. We have that $a'$ is the $i+1$ obtained from applying $f_i$ and $k'$ 
is the same as $k$. We see that $(a', \ell']\cap \sfL_{f_{i}(T)}(r', c')$ is the 
same as $(a, \ell]\cap \sfL_{T}(r, c)$ excluding the $i+1$. We have that 
$\mathcal{V}_{b}$ 
acts on $f_i(T)$ by removing $(a', \ell']\cap \sfL_{f_{i}(T)}(r', c')$ from 
$\sfL_{f_{i}(T)}(r, c)$ and $i+1$ from $\sfA_{f_{i}(T)}(r, c)$, leaving the 
$i+1$ in 
$\sfL_{f_{i}(T)}(r, c)$, replacing $k$ in $(r, c+1)$ with $((a', \ell']\cap 
\sfL_{f_{i}(T)}(r', c')) \cup \{a' \} $, and appending $k$ to 
$\sfA_{f_{i}(T)}(r, c+1)$. We see that $f_i(\mathcal{V}_{b}(T)) = 
\mathcal{V}_{b}(f_i(T))$. 
Finally, we assume that $a = i$ and $i+1$ is not in the cell $(r, c)$. This 
gives us that $\mathcal{V}_{b}(T)$ is attained from $T$ by removing $(a, 
\ell]\cap 
\sfL_{T}(r, c)$ from $\sfL_{T}(r, c)$ and $a$ from $\sfA_{T}(r, c)$, replacing 
$k$ in 
cell $(r, c+1)$ with $((a, \ell]\cap \sfL_{T}(r, c)) \cup \{ a\}$, and 
appending 
$k$ 
to $\sfA_{T}(r, c+1)$. Since $k \geqslant j > i+1$ for all $j \in (a, \ell]\cap \sfL_{T}(r,c)$, we have that the 
$i$ inserted into 
the cell $(r, c+1)$ is the rightmost unpaired $i$ in $\mathcal{V}_{b}(T)$. Note 
that the cell $(r+1, c+1)$ does not contain an $i+1$ and the cell $(r, c+2)$ 
does not contain an $i$. Thus, $f_i$ acts on $\mathcal{V}_{b}(T)$ by changing 
the $i$ in $(r, c+1)$ to an $i+1$. We now consider $\mathcal{V}_{b}(f_i(T))$. 
We have that $f_i$ acts on $T$ by changing $a$ into an $i+1$. We have that $a'$ 
is the $i+1$ obtained from applying $f_i$ and $k'$ is the same as $k$. We see 
that $(a', \ell']\cap \sfL_{f_{i}(T)}(r', c') = (a, \ell]\cap \sfL_{T}(r, c)$. 
We 
have 
that $\mathcal{V}_{b}$ acts on $f_i(T)$ by removing $(a, \ell]\cap \sfL_{T}(r, 
c)$ 
from $\sfL_{f_{i}(T)}(r, c)$ and $i+1$ from $\sfA_{f_{i}(T)}(r, c)$, replacing 
$k$ in 
$(r, c+1)$ with $((a, \ell]\cap \sfL_{T}(r, c)) \cup \{ a' \}$, and appending 
$k$ 
to 
$\sfA_{f_{i}(T)}(r, c+1)$. We see that $f_i(\mathcal{V}_{b}(T)) = 
\mathcal{V}_{b}(f_i(T))$.
		\end{itemize}

	\item[Case 2] Assume that $\hat{r} < r$ and $\hat{c} = c$.\\
	Note that $a > i$. By Lemma~\ref{lem: knuth} we have that
	$R_{i}(T) = R_{i}(\mathcal{V}_{b}(T))$ unless $a = i+1$ and column $c+1$ contains both an $i$ and an 
	$i+1$. However, even in this special case, we see that the rightmost unpaired $i$ of $\mathcal{V}_{b}(T)$ is 
	in the same position as the rightmost unpaired $i$ of $T$. We also see that 
	$\mathcal{V}_{b}(T)$ does not change whether or not cell $(\hat{r}+1, c)$ 
	contains an $i+1$ and whether or not cell $(\hat{r}, c+1)$ contains an $i$.	Thus, $f_{i}$ acts on the same $i$ and in 
	the same way for both $T$ and 
	$\mathcal{V}_{b}(T)$. Since $a> i$, we have that $k'$ is the same as $k$. 
	Note that the only way for $f_{i}$ to affect the cell $(r, c)$ in $T$ is if 
	$\hat{r} = r-1$ and $(r, c)$ contains an $i+1$. However, even in this 
	special case, we see that $(r', c')$, $a'$, $l'$, and $(a', \ell'] \cap 
	\sfL_{f_{i}(T)}(r', c')$ are the same as $(r, c)$, $a$, $\ell$, and $(a, 
	\ell] 
	\cap \sfL_{T}(r, c)$. Thus, $\mathcal{V}_{b}$ acts on $T$ and $f_{i}(T)$ in 
	the same way. Therefore, we have that the actions of $f_{i}$ and 
	$\mathcal{V}_{b}$ on $T$ are independent and $f_i(\mathcal{V}_{b}(T)) = 
	\mathcal{V}_{b}(f_i(T))$.

	\item[Case 3] Assume that $\hat{c} < c$. \\
	Let $\tilde{i}$ denote the rightmost unpaired $i$ of $T$. From the proof of Lemma~\ref{lem: knuth}, we have that $\mathcal{V}_{b}$ 
	does not change whether or not the $i$'s to the right of $\tilde{i}$ in $R_{i}(T)$ are $i$-paired. Thus, the rightmost unpaired $i$ in 
	$R_{i}(T)$ and $R_{i}(\mathcal{V}_b(T))$ are in the same position. As $\mathcal{V}_{b}$ does not affect any column to the left of 
	column $c$, we have that the rightmost unpaired $i$ for $\mathcal{V}_{b}(T)$ is
	in the same position as the rightmost unpaired $i$ for $T$. Note that $\mathcal{V}_{b}$ also does not affect whether or not cell 
	$(\hat{r}+1, \hat{c})$ contains an $i+1$ and whether or not cell $(\hat{r}, \hat{c}+1)$ contains an $i$. Thus, $f_{i}$ acts on 
	the rightmost unpaired $i$ in $T$ and $\mathcal{V}_{b}(T)$ in exactly the 
	same way. Next, we note that $(r', c')$, $a'$, $k'$, and $(a', \ell'] \cap 
	\sfL_{f_{i}(T)}(r', c')$ are the same as $(r, c)$, $a$, $k$, and $(a, \ell] \cap 
	\sfL_{T}(r, c)$ respectively. Thus, $\mathcal{V}_{b}$ acts on $T$ and 
	$f_{i}(T)$ in the same way. Therefore, we have that the actions of $f_{i}$ 
	and $\mathcal{V}_{b}$ on $T$ are independent and $f_i(\mathcal{V}_{b}(T)) = 
	\mathcal{V}_{b}(f_i(T))$.

	\item[Case 4] Assume that $\hat{r} \leqslant r$ and $\hat{c} = c+1$.\\
Under this assumption, we have that column $c+1$ does not contain an $i+1$ and $a \not = i+1$ since the
cells in column $c+1$ do not contain any arms. We break into subcases.
		\begin{itemize}
			\item Assume that $k \not = i$. This implies that the rightmost 
			unpaired $i$ in $\mathcal{V}_{b}(T)$ is in the same position as the 
			rightmost unpaired $i$ in $T$. We see that $\mathcal{V}_{b}$ does 
			not change whether or not cell $(\hat{r}+1, c+1)$ contains an $i+1$ 
			and whether or not cell $(\hat{r}, c+2)$ contains an $i$. Thus, 
			$f_{i}$ acts on the rightmost unpaired $i$ in $T$ and 
			$\mathcal{V}_{b}(T)$ in exactly the same way. We also observe that 
			$(r', c')$, $a'$, $\ell'$, $k'$, and $(a', \ell'] \cap 
			\sfL_{f_{i}(T)}(r', c')$ are the same as $a$, $\ell$, $k$, and $(a, \ell] \cap 
			\sfL_{f_{i}(T)}(r, c)$ respectively. Thus, $\mathcal{V}_{b}$ acts on 
			$T$ and $f_{i}(T)$ in the same way. Therefore, we have that the 
			actions of $f_{i}$ and $\mathcal{V}_{b}$ on $T$ are independent and 
			$f_i(\mathcal{V}_{b}(T)) = \mathcal{V}_{b}(f_i(T))$.

			\item Assume that $k = i$. We see that the rightmost unpaired $i$ in $\mathcal{V}_{b}(T)$ 
			is the $i$ that was appended to $\sfA_{T}(\hat{r}, c+1)$. Note that 
			$\mathcal{V}_{b}$ does not change whether or not cell $(\hat{r}+1, 
			c+1)$ contains an $i+1$ and whether or not cell $(\hat{r}, c+2)$ 
			contains an $i$. We first make the extra assumption that $(\hat{r}, c+2)$ in $T$ contains an $i$. 
                       This implies that $f_{i}$ acts on $\mathcal{V}_{b}(T)$ and $T$ in the same way by removing
			the $i$ from the hook entry of $(\hat{r}, c+2)$ and appending an $i+1$ to the leg of $(\hat{r}, c+1)$.
			We also have that $(r', c')$, $a'$, $\ell'$, $k'$, and $(a', \ell'] \cap \sfL_{f_{i}(T)}(r', c')$ are 
			equal to $(r, c)$, $a$, $\ell$, $k$, and $(a, \ell] \cap \sfL_{f_{i}(T)}(r, c)$ 
			respectively. Thus, $\mathcal{V}_{b}$ acts on $T$ and $f_{i}(T)$ in 
			the same way. Therefore, we have that the actions of $f_{i}$ and 
			$\mathcal{V}_{b}$ on $T$ are independent and 
			$f_i(\mathcal{V}_{b}(T)) = \mathcal{V}_{b}(f_i(T))$.
                        We now assume that $(\hat{r}, c+2)$ does not contain an $i$.
                        This implies that $f_{i}$ acts on $\mathcal{V}_{b}(T)$ by changing the $i$ in 
			$\sfA_{\mathcal{V}_{b}(T)}(\hat{r}, c+1)$ to an $i+1$ and 
			acts on $T$ similarly by changing the $i$ in 
			$\sfL_{\mathcal{V}_{b}(T)}(\hat{r}, c+1)$ to an $i+1$. Note that 
			$(r', c')$, $a'$, $\ell'$, and $(a', \ell'] \cap \sfL_{f_{i}(T)}(r', c')$ 
			are equal to $(r, c)$, $a$, $\ell$, and $(a, \ell] \cap 
			\sfL_{f_{i}(T)}(r, c)$ respectively while $k'$ is the $i+1$ in 
			$\sfL_{f_{i}(T)}(\hat{r}, c+1)$. Thus, besides the value of the number that is 
                        bumped from the leg of $(\hat{r}, c+1)$ to its arm, we have $\mathcal{V}_{b}$ 
			acts on $T$ and $f_{i}(T)$ in the same way. Looking at 
			$f_i(\mathcal{V}_{b}(T))$ and $\mathcal{V}_{b}(f_i(T))$, we see 
			that $f_i(\mathcal{V}_{b}(T)) = \mathcal{V}_{b}(f_i(T))$.
		\end{itemize}

	\item[Case 5] Assume that $\hat{r} >r$ and $\hat{c} = c$ or $c+1$.\\
	Under this assumption, we have that $\mathcal{V}_{b}$ does not change the cells $(\hat{r}, \hat{c})$,  $(\hat{r}+1, \hat{c})$, 
	and $(\hat{r}, \hat{c}+1)$. We also have that $R_{i}(T) = R_{i}(\mathcal{V}_{b}(T))$ implying that the rightmost unpaired $i$
	in $\mathcal{V}_{b}(T)$ is in the same position as the rightmost unpaired $i$ in $T$. Thus, $f_{i}$ acts on the rightmost 
	unpaired $i$ in $\mathcal{V}_{b}(T)$ and $T$ in the same way. Note that $i+1$ cannot be in column $\hat{c}$ implying 
	that $f_{i}$ can only make changes to the legs and hook entries of $(\hat{r}, \hat{c})$ and $(\hat{r}, \hat{c}+1)$. Since these 
	changes only affect the legs and hook entries of cells outside of the possible cells that $\mathcal{V}_{b}$ can change, we 
	have that $\mathcal{V}_{b}$ acts on $T$ and $f_{i}(T)$ in the same way. Therefore, we have that the actions of $f_{i}$ 
	and $\mathcal{V}_{b}$ on $T$ are independent and $f_i(\mathcal{V}_{b}(T)) = \mathcal{V}_{b}(f_i(T))$.

	\item[Case 6] Assume that $\hat{c} \geqslant c+2$.\\
	Let $\tilde{i}$ denote the rightmost unpaired $i$ of $T$. From the proof of Lemma~\ref{lem: knuth}, we have that $\mathcal{V}_{b}$ 
	does not change whether or not the $i+1$'s to the left of $\tilde{i}$ are $i$-paired. Thus, the rightmost unpaired $i$ in $R_{i}(T)$ and 
	$R_{i}(\mathcal{V}_b(T))$ are in the same position. As $\mathcal{V}_{b}$ does not affect any column to the right of column $c+1$, 
	we have that the rightmost unpaired $i$ for $\mathcal{V}_{b}(T)$ is 
	in the same position as the rightmost unpaired $i$ for $T$. Note that $\mathcal{V}_{b}$ also does not affect whether or not cell 
	$(\hat{r}+1, \hat{c})$ contains an $i+1$ and whether or not cell $(\hat{r}, \hat{c}+1)$ contains an $i$. Since the cells that $f_{i}$ 
	and $\mathcal{V}_{b}$ could change are different and the rightmost unpaired $i$ 
	does not change, we have that the actions of $f_{i}$ and $\mathcal{V}_{b}$ on $T$ are independent and 
	$f_i(\mathcal{V}_{b}(T)) = \mathcal{V}_{b}(f_i(T))$.
\end{description}

\end{proof}

\begin{theorem}
\label{theorem.main}
    Let $T \in \HVT$. 
    \begin{enumerate}
    \item
    If $f_i(T) \neq 0$, we have $f_i(P(T)) = P(f_i(T))$ and $Q(T) = Q(f_i(T))$.
    \item
    If $e_i(T) \neq 0$, we have $e_i(P(T)) = P(e_i(T))$ and $Q(T) = Q(e_i(T))$.
    \end{enumerate}
\end{theorem}

\begin{proof}
Part (2) follows from part (1) since $e_i$ and $f_i$ are partial inverse. We prove part (1) here.

Let $T \in \HVT$ with arm excess $\alpha$ such that $f_i(T)\neq 0$ for some 
$i$. Then $f_i(P(T))=P(f_i(T))$ follows from Lemma~\ref{intertwine_lemma}, as 
$P(T)$ is obtained by successive applications of $\mathcal{V}$ on $T$ and each 
application of $\mathcal{V}$ is a string of applications of $\mathcal{V}_b$.

Since crystal operators do not change arm excess, we may employ the notation 
in Definition~\ref{def: uncrowd_map} and denote the pair of insertion and recording tableaux produced 
at the $j$-th step for $0\leqslant j\leqslant \alpha$ of the uncrowding map $\mathcal{U}$ for 
$T$ and $f_i(T)$ as $(P_j(T),Q_j(T))$ and $(P_j(f_i(T)),Q_j(f_i(T)))$, respectively. 
As crystal operators do not change the shape of $T$, we have 
$\shape(P_j(f_iT)) = \shape(f_i(P_j(T))) = \shape(P_j(T))$ for all $0 \leqslant j \leqslant \alpha$. 
Hence
\begin{equation}
\label{equation.PP}
	\shape(P_{j+1}(T))/\shape(P_{j}(T)) = \shape(P_{j+1}(f_i(T)))/\shape(P_{j}(f_i(T))) \qquad \text{for all $0 \leqslant j\leqslant \alpha -1$.}
\end{equation}

Next we show $Q_j(T) = Q_j(f_i(T))$ for all $0\leqslant j\leqslant \alpha$ by induction. When $j=0$, $Q_0(T)=Q_0(f_i(T))$ since 
$\shape(P_0(T)) = \shape(P_0(f_i(T))) = \shape(T)$. 

Suppose $Q_j(T) = Q_j(f_i(T))$ for a given $j\geqslant 0$. It suffices to show that the cells 
\[
\begin{split}
	\shape(Q_{j+1}(T))/\shape(Q_j(T)) &= \shape(P_{j+1}(T))/\shape(P_{j}(T)) \qquad \text{and}\\
	\shape(Q_{j+1}(f_i(T)))/\shape(Q_j(f_i(T))) &= \shape(P_{j+1}(f_i(T)))/\shape(P_{j}(f_i(T)))
\end{split}
\]
 in $Q_{j+1}(T)$ and $Q_{j+1}(f_i(T))$ are at the same position with the same 
entry. By~\eqref{equation.PP}, the cells are in the same position, say in column $\tilde{c}$. By 
Definition~\ref{def: hvt_crystal}, $f_i$ does not move elements in 
the arm to a different column, so the columns in which we start the uncrowding 
insertion $\mathcal{V}$ on $P_j(T)$ and $P_j(f_i(T))$ are the same, say $c$, by 
Definition~\ref{def: uncrowd_map}. Hence the cells $\shape(Q_{j+1}(T))/\shape(Q_j(T))$ and 
$\shape(Q_{j+1}(f_i(T)))/\shape(Q_j(f_i(T)))$ are at the same position with 
entry $\tilde{c}-c$. The theorem follows.
\end{proof}

Hawkes and Scrimshaw~\cite[Theorem 4.6]{HS.2019} proved that $\HVT^m(\lambda)$ is a Stembridge crystal by 
checking the Stembridge axioms. This also follows directly from our analysis above.

\begin{corollary}
The crystal $\HVT^m(\lambda)$ of Definition~\ref{def: hvt_crystal} is a Stembridge crystal of type $A_{m-1}$.
\end{corollary}

\begin{proof}
According to~\cite{MPS.2018}, $\SVT^m(\mu)$ is a Stembridge crystal of type $A_{m-1}$. By Theorem~\ref{theorem.main}, the map
\[
	\mathcal{U} \colon \HVT^m(\lambda) \to \bigsqcup_{\mu \supseteq \lambda} \SVT^m(\mu) \times \hat{\mathcal{F}}(\mu/\lambda),
\]
is a strict crystal morphism (see for example \cite[Chapter~2]{BumpSchilling.2017}). The statement follows.
\end{proof}

\subsection{Uncrowding map on multiset-valued tableaux}
\label{section.uncrowding MVT}

The uncrowding map on hook-valued tableaux described above turns out to be a 
generalization of the uncrowding map on multiset-valued tableaux by Hawkes and 
Scrimshaw \cite[Section 3.2]{HS.2019}.
We will prove that this is indeed the case in this section.
Let us recall the definition of the uncrowding map in \cite[Section 3.2]{HS.2019}.

\begin{definition}\label{uncrowd: Hawkes-Scrimshaw}
Let $T \in\MVT(\lambda)$. The \defn{uncrowding map} 
\[
	\Upsilon: \MVT(\lambda) \to \bigsqcup_{\mu\supseteq \lambda} \SSYT(\mu)\times \hat{\mathcal{F}}(\mu/\lambda)
\]
sends $T$ to a pair of tableaux using the following algorithm:
	\begin{enumerate}
		\item Set $U_{\lambda_1+1}=\emptyset$ and $F_{\lambda_1+1}$ be the 
		unique column-flagged increasing tableau of shape $\emptyset/\emptyset$.
		\item Let $1\leqslant k\leqslant \lambda_1$ and assume that the pair $(U_{k+1},F_{k+1})$
		is defined. The pair $(U_k, F_k)$ is defined recursively from $(U_{k+1},F_{k+1})$ using the following two steps:
		\begin{enumerate}
		\item Define $U_k$ as the RSK row insertion tableau from the word 
		\[
			R(C_k) R(C_{k+1}) \cdots R(C_{\lambda_1}),
		\]
		where $C_j$ is the $j$-th column of $T$ for every $1\leqslant j\leqslant \lambda_1$.
		In other words, if we denote by $T_{\geqslant k}$ the tableau formed by the 
		columns weakly to the right of the $k$-th column of $T$, $U_k$ is 
		obtained by performing the RSK row insertion using the 
		column reading word of $T_{\geqslant k}$. 
		\item Form the tableau $F_k$ of shape $\shape(U_k)/\shape(T_{\geqslant k})$ as 
		follows. Shift $F_{k+1}$ by one column to the right and fill the boxes in the 
		same positions into $F_k$; for every unfilled box in the shape 
		$\shape(U_k)/\shape(U_{k+1})$, label each box in column $i$ with entry 
		$i-1$.
		\end{enumerate}
	\end{enumerate}
	Define $\Upsilon(T)=(U,F):=(U_1, F_1)$.
\end{definition}

\begin{example}
	Let $T$ be the multiset-valued tableau
	\[
	\ytableausetup{notabloids,boxsize=2em}
	T =\, 
	\raisebox{7mm}{\scriptsize{\begin{ytableau}
			\begin{array}{l} 45 \end{array}\\
			\begin{array}{l} 233 \end{array} & \begin{array}{l} 345 
			\end{array}\\
			1 & \begin{array}{l} 11 \end{array} & \begin{array}{l} 4 
			\end{array}
			\end{ytableau}}}\,.
	\]
	Then, we obtain the following pairs of tableaux for the
	uncrowding map $\Upsilon$:
	\[
	\begin{aligned}
	(U_4, F_4) &= \left(\emptyset, \emptyset\right)\\
	(U_3, F_3) &= \left(
	\ytableausetup{notabloids,boxsize=1.8em}
	\raisebox{-2mm}{\scriptsize{\begin{ytableau}
		4
		\end{ytableau}}} \,,\, 
	\raisebox{-2mm}{\scriptsize{\begin{ytableau}
		*(gray)
		\end{ytableau}}}
	\right)\\
	(U_2, F_2) &= \left(
	\raisebox{1mm}{\scriptsize{\begin{ytableau}
	3 & 5\\
	1 & 1 & 4 & 4
	\end{ytableau}}} \,,\, 
	\scriptsize{\begin{ytableau}
		*(gray) & 1\\
		*(gray) & *(gray) & 2 & 3
		\end{ytableau}}
	\right)\\
	(U_1, F_1) &= \left(
	\raisebox{1mm}{\scriptsize{\begin{ytableau}
			4 & 5\\
			2 & 3 & 3 & 5\\
			1 & 1 & 1 & 3 & 4 & 4
			\end{ytableau}}} \,,\, 
	\scriptsize{\begin{ytableau}
		*(gray) & 1\\
		*(gray) & *(gray) & 1 & 3\\
		*(gray) & *(gray) & *(gray) & 2 & 3 & 5
		\end{ytableau}}
	\right)& = (U,F) = \Upsilon(T).
	\end{aligned}\]
\end{example}

\begin{proposition}
	Let $T\in\MVT(\lambda)$. Then $\mathcal{U}(T) = \Upsilon(T)$.
	In other words, the uncrowding map as defined in Definition \ref{def: 
	uncrowd_map} is equivalent to the uncrowding map of Definition~\ref{uncrowd: Hawkes-Scrimshaw}
	in~\cite[Section 3.2]{HS.2019}.
\end{proposition}

\begin{proof}
	Recall from Definition~\ref{def: uncrowd_map}, that the pair of uncrowding and recording tableaux for $\mathcal{U}(T)$
	is denoted by $(P(T),Q(T)) = \mathcal{U}(T)$. Similarly, let us denote $(U(T),F(T)) := \Upsilon(T)$.
	
	Assume that $S\in \MVT(\lambda)$ is highest weight, that is, $e_i(S) = 0$ for $i \geqslant 1$.
	By~\cite[Proposition 3.10]{HS.2019}, row $i$ of $S$ only contains the letter $i$. Thus its weight is some partition 
	$\mu = (\mu_1,\mu_2,\dots,\mu_{\ell})$ if $\lambda = (\lambda_1,\lambda_2,\dots,\lambda_{\ell})$. 
	By Proposition~\ref{proposition.main} and Theorem~\ref{theorem.main}, $P(S)\in \SSYT$ is highest weight. 
	As weights of tableaux are preserved under uncrowding, the weight of $P(S)$ is equal to $\mu$.
	By a similar argument using~\cite[Theorem 3.17]{HS.2019}, $U(S) \in \SSYT$ is also highest weight with weight $\mu$. 
	Since highest weight semistandard Young tableaux are uniquely determined by 
	their weights, we have $P(S) = U(S)$.
	
	Recall that as long as $f_iT \neq 0$ for $T\in \MVT(\lambda)$, we have $U(f_i T) = f_i U(T)$ by~\cite[Theorem 3.17]{HS.2019} 
	and  $P(f_i T) = f_i P(T)$ by Theorem~\ref{theorem.main}.
	Now let $T\in \MVT(\lambda)$ be arbitrary. Then $T=f_{i_1} \cdots f_{i_k}(S)$ for some sequence of $i_1,\ldots,i_k$ and $S$
	highest weight. Hence,
	\[
		P(T) = P(f_{i_1} \cdots f_{i_k} S) = f_{i_1} \cdots f_{i_k} P(S) = f_{i_1} \cdots f_{i_k} U(S)
		= U(f_{i_1} \cdots f_{i_k} S) = U(T).
	\]
	
	It remains to show that $Q(T)=F(T)$ 	for all $T\in\MVT(\lambda)$. 
	To do this, we show that the newly created boxes of the 
	uncrowding map up to a specified column in Definition \ref{uncrowd: 
	Hawkes-Scrimshaw} are in the 
	same positions as those for the uncrowding insertion in Definition~\ref{def: uncrowd_map}. 
	For every $Y\in\MVT(\mu)$ and for every $1\leqslant j\leqslant \mu_1$, denote by
	$Y_{\geqslant j}$ the tableau formed by the rightmost $j$ columns of $Y$; 
	here $Y_{\geqslant {\mu_1+1}}$ is the empty tableau.
	
	Let $T \in \MVT(\lambda)$ be arbitrary.	
	For $1\leqslant k\leqslant \lambda_1+1$, let $P^{(k)}$ be the tableau obtained by 
	performing the uncrowding map $\mathcal{U}$ on $T$ on the columns from right to left up to and including
	the $k$-th column of $T$; here $P^{(\lambda_1+1)}=T$. 
	In other words, $P^{(k)}=\mathcal{V}^{\alpha_k}(T)$ as in Definition 
	\ref{def: uncrowd_insert}, where $\alpha_k$ is the arm excess of $T_{\geqslant k}$.
	As the entries to the left of column $k$ of $T$ are untouched by the 
	uncrowding insertion in Definition~\ref{def: uncrowd_insert}, for every 
	$1\leqslant k\leqslant \lambda_1+1$, we have $(P^{(k)})_{\geqslant k}=
	P(T_{\geqslant k})=U(T_{\geqslant k})$.
	It follows that for every $1\leqslant k\leqslant \lambda_1$, up to horizontal shifts, 
	the newly formed boxes in $\shape(P^{(k)})/\shape(P^{(k+1)})
	=\shape[(P^{(k)})_{\geqslant k+1}]/\shape[(P^{(k+1)})_{\geqslant k+1}]$ and 
	$\shape([U(T_{\geqslant k})]_{\geqslant k+1})/\shape([U(T_{\geqslant k+1})]_{\geqslant k+1})$ 
	are in the same positions.	
	Since the entries in these boxes both record the difference in column 
	indices relative to the $k$-th column for each $1\leqslant k\leqslant \lambda_1$ and 
	since the recording tableaux for both maps are formed from the union of 
	these boxes, we conclude that $Q(T)=F(T)$, completing the proof.
\end{proof}

\subsection{Crowding map}
\label{section.crowding}

In this section, we give a description of the ``inverse'' of the uncrowding map.
\smallskip

We begin by introducing some notation. Let $F\in \hat{\mathcal{F}}$ with $e$ entries. For each cell $(r,c)$ in 
$F$ with entry $F(r,c)$, define the corresponding \defn{destination column} to be $d(r,c) = c-F(r,c)$. 
Define the \defn{crowding order} on $F$ by ordering all the cells in $F$ with a filling, first determined by their 
destination column (smallest to largest) and then by column index (largest to smallest). Denote the order by 
$(r_1,c_1),(r_2,c_2),\dots, (r_e,c_e)$. Set $\alpha(F) = (\alpha_1,\alpha_2, \dots, \alpha_e)$, where $\alpha_i = F(r_i,c_i)$. 
Let the arm excess for a column of a hook-valued tableau
be the sum of arm excesses of all its cells.

\begin{definition}\label{def: cb}
Let $h\in \HVT$ and let $(r,c)$ be a cell in $h$ with $c>1$ and with at most one element in $\sfA_h(r,c)$. If $\sfA_h(r,c)$ is empty, we also 
require that the cell $(r,c)$ is a corner cell in $h$. Then we
define the \defn{crowding bumping} $\mathcal{C}_b$ on the pair $[h,(r,c)]$ by the following algorithm:
\begin{enumerate}
\item If $\sfA_h(r,c)$ is nonempty, set $m$ to be the only element in
$\sfA_h(r,c)$ and $b=\max\{x\in\sfL^+_h(r,c) \mid x\leqslant m\}$. Otherwise, set
$m=\sfH_h(r,c)$ and $b=\max(\sfL^+_h(r,c))$.
\label{alg: m,b}
\item Find the largest $r'$ such that $\sfH_h(r',c-1)\leqslant b$.
If $r'=r$, set $q=\sfH_h(r,c)$. Otherwise, set $q=b$.
\label{alg: r'}
In either case, append $q$ to $\sfA_h(r',c-1)$.
\item
\begin{enumerate}
	\item If $r'$ from Step \ref{alg: r'} equals $r$, perform either of the following:
	\begin{enumerate}
	\item If $\sfA_h(r,c)$ is nonempty, move the set $\{x\in \sfL_h(r,c) \mid q<x\leqslant m\}$
	from $\sfL_h(r,c)$ to $\sfL_h(r',c-1)$ and keep it strictly increasing. Remove $m$ from $\sfA_h(r,c)$ and 
	set $H_h(r,c)= m$.\label{alg:3-a-i}
	\item Otherwise, $\sfA_h(r,c)$ is empty, so move $\sfL_h(r,c)$ into
	$\sfL_h(r',c-1)$ and keep it to be strictly increasing. Remove cell $(r,c)$ from $h$.\label{alg:3-a-ii}
	\end{enumerate}
	\item Otherwise, $r'\neq r$ and perform either of the following:
	\begin{enumerate}
		\item Suppose that $\sfA_h(r,c)$ is nonempty.
		Replace $q$ in $\sfL^+_h(r,c)$ with $m$.
		Remove $m$ from $\sfA_h(r,c)$.\label{alg:3-b-i}
		\item If instead $\sfA_h(r,c)$ is empty, then remove cell $(r,c)$ from $h$.\label{alg:3-b-ii}
	\end{enumerate}
\end{enumerate}
\label{alg: move}
\end{enumerate}
Denote the resulting (not necessarily semistandard) hook-valued tableau by $h'$. We write $\mathcal{C}_b([h,(r,c)]) = [h',(r',c-1)]$. We 
also define the \defn{projections} $p_1$ and $p_2$ by $p_1 \circ \mathcal{C}_b([h,(r,c)])=h'$ and $p_2 \circ \mathcal{C}_b([h,(r,c)])=(r',c-1)$. 
See Figures~\ref{fig: cb-a} and \ref{fig: cb-b} for illustration.
\begin{figure}[h!]
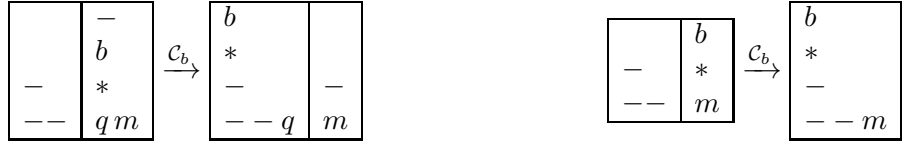

	\begin{minipage}{0.45\textwidth}
		\[
			{\def\mc#1{\makecell[lb]{#1}}
			{\begin{array}[lb]{*{2}{|l}|}\cline{1-2}
				\mc{- \\--}&\mc{-\\b\\\ast\\q \,m}\\\cline{1-2}
				\end{array}} \xrightarrow[]{\mathcal{C}_b}{\begin{array}[lb]{*{2}{|l}|}\cline{1-2}
				\mc{b\\\ast \\-\\--q}&\mc{-\\m}\\\cline{1-2}
				\end{array}}
			 }
		\]
	\end{minipage}
	\begin{minipage}{0.45\textwidth}
		\[
			{\def\mc#1{\makecell[lb]{#1}}
			{{\begin{array}[lb]{*{2}{|l}|}\cline{1-2}
				\mc{- \\--}&\mc{b\\\ast\\m}\\\cline{1-2}
				\end{array}}
				 \xrightarrow[]{\mathcal{C}_b}\begin{array}[lb]{*{2}{|l}|}\cline{1-1}
					\mc{b\\\ast \\-\\--m}\\\cline{1-1}
					\end{array}}
				}
		\]
	\end{minipage}
	\caption{When $r'=r$. Left: (i) $\sfA_{h}(r,c) \neq \emptyset$.\quad Right: (ii) $\sfA_{h}(r,c) = \emptyset$.}\label{fig: cb-a}
\end{figure}
\begin{figure}[h!]
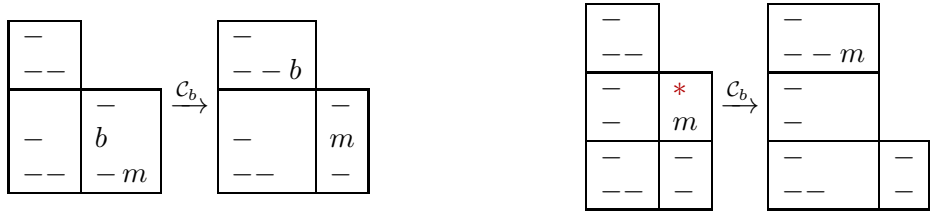

\begin{minipage}{0.45\textwidth}
	\[
		{\def\mc#1{\makecell[lb]{#1}}
		{\begin{array}[lb]{*{2}{|l}|}\cline{1-1}
			\mc{-\\--}\\\cline{1-2}
			\mc{- \\- -}&\mc{-\\b\\-\,m}\\\cline{1-2}
			\end{array}} \xrightarrow[]{\mathcal{C}_b}{\begin{array}[lb]{*{2}{|l}|}\cline{1-1}
			\mc{-\\--b}\\\cline{1-2}
			\mc{- \\- -}&\mc{-\\m\\-}\\\cline{1-2}
			\end{array}}}
	\]
\end{minipage}
\begin{minipage}{0.45\textwidth}
	\[
		{\def\mc#1{\makecell[lb]{#1}}
		{\begin{array}[lb]{*{2}{|l}|}\cline{1-1}
			\mc{-\\--}\\\cline{1-2}
			\mc{- \\-}&\mc{\color{darkred}{*}\\m}\\\cline{1-2}
			\mc{- \\ - -}&\mc{-\\-}\\\cline{1-2}
			\end{array}} \xrightarrow[]{\mathcal{C}_b}{\begin{array}[lb]{*{2}{|l}|}\cline{1-1}
			\mc{-\\--m}\\\cline{1-1}
			\mc{- \\-}\\\cline{1-2}
			\mc{- \\ - -}&\mc{-\\-}\\\cline{1-2}
			\end{array}}}
	\]
\end{minipage}
	\caption{When $r' \neq r$. Left: $\sfA_{h}(r,c) \neq \emptyset$.\quad Right: $\sfA_{h}(r,c) = \emptyset$.}\label{fig: cb-b}
	\end{figure}
\end{definition}

\begin{example}\label{social-distancing}
	We compute $\mathcal{C}_b$ in two examples:
	\[
	T =\,
	{\def\mc#1{\makecell[lb]{#1}}
	{\begin{array}[lb]{*{2}{|l}|}\cline{1-1}
	\mc{5}\\\cline{1-2}
	\mc{1\;1}&\mc{5\\4\\3\\2\;4}\\\cline{1-2}
	\end{array}}\, ,\quad
	\mathcal{C}_b([T,(1,2)]) = 	[\;
	{\begin{array}[lb]{*{2}{|l}|}\cline{1-1}
	\mc{5}\\\cline{1-2}
	\mc{4\\3\\1\;1\;2}&\mc{5\\4}\\\cline{1-2}
	\end{array}}\,,(1,1)\,] = [\;T'\,,(1,1)\,].}
	\]
	\[
	S =\,
	{\def\mc#1{\makecell[lb]{#1}}
	{\begin{array}[lb]{*{2}{|l}|}\cline{1-1}
	\mc{3}\\\cline{1-2}
	\mc{2\\1}&\mc{3\\2}\\\cline{1-2}
	\end{array}}\, ,\quad
	\mathcal{C}_b([S,(1,2)]) = 	[\;
	{\begin{array}[lb]{*{2}{|l}|}\cline{1-1}
	\mc{33}\\\cline{1-1}
	\mc{2\\1}\\\cline{1-1}
	\end{array}}\,,(2,1)\,] = [\;S'\,,(2,1)\,].}
	\]
\end{example}

\smallskip

\begin{remark}\label{remark: loss}
In Definition~\ref{def: cb},
\begin{itemize}
	\item if $r'= r$, then $h'$ is always semistandard and has the same weight as $h$;
	\item if $r'\neq r$ and $\sfA_h(r,c)$ is empty, then $h'$ might have fewer letters than $h$. In Example~\ref{social-distancing}, 
	$S$ contains $5$ letters while $S'$ only contains $4$. This happens precisely when $\sfL_h(r,c)$ is nonempty.
\end{itemize}
In principle, the arm in cell $(r',c-1)$ could be greater than the $q$ that is to be inserted. However, we only consider the cases as defined 
in the order described by the next paragraph. We refer to Proposition~\ref{prop: semistandard} which states that all tableaux we deal with in 
this section are indeed semistandard hook-valued tableaux.
\end{remark}

Let $(S,F)\in \SVT(\mu) \times \mathcal{\hat{F}}(\mu/\lambda)$ with crowding order $(r_1,c_1),(r_2,c_2),\dots, (r_e,c_e)$ and
$\alpha(F) = (\alpha_1,\alpha_2, \dots, \alpha_e)$. For all $0 \leqslant j \leqslant e-1$ and for all $0\leqslant 
s\leqslant\alpha_{j+1}$, define 
$T^{(s)}_j$ recursively by setting $T^{(0)}_0:=S$ and
\begin{equation*}
T^{(s)}_j := \begin{cases}
p_1 \circ \mathcal{C}_b([T^{(s-1)}_j,(r_{j+1},c_{j+1})]) &\text{when 
$s>0$,}\\
T^{(\alpha_j)}_{j-1} &\text{when $s=0$ and $j>0$.}
	\end{cases}
\end{equation*}
Additionally, define $T^{(0)}_e:= T^{(\alpha_e)}_{e-1}$.

Thus we obtain the following sequence
\[
S = T^{(0)}_0\, \xrightarrow[(r_1,c_1)]{p_1\circ\,\mathcal{C}^{\alpha_1}_b}\,	
T^{(0)}_1 \,
\xrightarrow[(r_2,c_2)]{p_1\circ\,\mathcal{C}^{\alpha_2}_b}\, 
T^{(0)}_2\,
\xrightarrow[(r_3,c_3)]{p_1\circ\,\mathcal{C}^{\alpha_3}_b}\,\dots 
\xrightarrow[(r_e,c_e)]{p_1\circ\,\mathcal{C}^{\alpha_e}_b} \, T^{(0)}_e.
\]

\begin{remark}
The tableaux $T^{(s)}_j$ are well-defined. 
We check the conditions in Definition~\ref{def: cb}. 
Let $h= T^{(s)}_j$ for some $0\leqslant j\leqslant e-1$ and for some $0\leqslant
s<\alpha_{j+1}$, with cell $(r,c)$.
\begin{itemize}
	\item Since $F\in \mathcal{\hat{F}}$, we always have $c>1$.\
	\item The case that $\sfA_h(r,c)$ is empty can only occur in $T_{j-1}^{(0)}$ for 
	some $j>0$.
	In this case, $(r,c)=(r_j,c_j)$, which is a corner cell.
	\item 
		Consider the $\alpha_j$ steps in $T^{(0)}_{j-1}\, 
		\xrightarrow[(r_j,c_j)]{p_1\circ\,\mathcal{C}^{\alpha_j}_b}\, 
		T^{(0)}_{j}$. We first delete cell $(r_j,c_j)$, which has no arm. Then at every step 
		after that, we move leftward one 
		column at a time. Before we reach column $d(r_j,c_j)$, there is exactly one column with arm excess being 1 and the rest has zero 
		arm excess among columns to the right of $d(r_j,c_j)$ since recall that the cells $(r_j,c_j)$ are ordered from smallest to largest
		destination column. Once we reach column $d(r_j,c_j)$, the cell there may contain more than 
		one arm element, but we then go to $(r_{j+1},c_{j+1})$, which is a corner cell instead. Thus 
		there is at most one element in $\sfA_h(r,c)$.
\end{itemize}
\end{remark}

\begin{definition}
	With the same notation as above, define the \defn{insertion path} of 
	$T^{(0)}_{j-1} \to T^{(0)}_j$
	for $1\leqslant j \leqslant e$ to be
	\[
		\mathsf{path}_j:= \left((r^{(0)}_j,c^{(0)}_j),(r^{(1)}_j,c^{(1)}_j),\dots, (r^{(\alpha_j)}_j,c^{(\alpha_j)}_j)\right),
	\]
	where $(r^{(s)}_j,c^{(s)}_j):= p_2 \circ 
	\mathcal{C}^s_b([T^{(0)}_{j-1},(r_j,c_j)])$ for $0\leqslant s \leqslant 
	\alpha_j$.
\end{definition}

\begin{example}\label{eg: crowding}
	Consider the following pair of tableaux $(S,F) \in \HVT((5,3,2)) \times \mathcal{\hat{F}}((5,3,2)/((3,2,1)))$,
	\[
		{\def\mc#1{\makecell[lb]{#1}}
		{S\, = \begin{array}[lb]{*{5}{|l}|}\cline{1-2}
		\mc{5\\4}&\mc{5}\\\cline{1-3}
		\mc{2}&\mc{3}&\mc{4\\3}\\\cline{1-5}
		\mc{1}&\mc{1}&\mc{2\\1}&\mc{4}&\mc{4}\\\cline{1-5}
		\end{array}}}, \quad F = \begin{ytableau}
			*(gray) & 1 \\
			*(gray) & *(gray) & 1 \\
			*(gray) & *(gray) & *(gray) & 3 & 4
		\end{ytableau}\,.
	\]
The crowding order is $(1,5),(1,4),(3,2),(2,3)$. The insertion path and destination column for each of them are:
\begin{align*}
\mathsf{path}_1 =  ((1,5),(1,4),(2,3),(2,2),(2,1))&, \;d(1,5)=1,\\
\mathsf{path}_2 =  ((1,4),(2,3),(2,2),(3,1))&, \; d(1,4)=1,\\
\mathsf{path}_3 = ((3,2),(3,1))&, \; d(3,2)=1, \\
\mathsf{path}_4 = ((2,3),(2,2))&,\; d(2,3)=2.
\end{align*}
We obtain the sequence from the algorithm:
\[
	{\def\mc#1{\makecell[lb]{#1}}
	{\begin{array}[lb]{*{5}{|l}|}\cline{1-2}
	\mc{5\\4}&\mc{5}\\\cline{1-3}
	\mc{2}&\mc{3}&\mc{4\\3}\\\cline{1-5}
	\mc{1}&\mc{1}&\mc{2\\1}&\mc{4}&\mc{4}\\\cline{1-5}
	\end{array}}}	 \xrightarrow[(1,5)]{p_1\circ\,\mathcal{C}^4_b}{\def\mc#1{\makecell[lb]{#1}}
	{\begin{array}[lb]{*{4}{|l}|}\cline{1-2}
	\mc{5\\4}&\mc{5}\\\cline{1-3}
	\mc{23}&\mc{4\\3}&\mc{4}\\\cline{1-4}
	\mc{1}&\mc{1}&\mc{2\\1}&\mc{4}\\\cline{1-4}
	\end{array}}}  \xrightarrow[(1,4)]{p_1\circ\,\mathcal{C}^3_b}{\def\mc#1{\makecell[lb]{#1}}
	{\begin{array}[lb]{*{3}{|l}|}\cline{1-2}
	\mc{5\\44}&\mc{5}\\\cline{1-3}
	\mc{23}&\mc{4\\3}&\mc{4}\\\cline{1-3}
	\mc{1}&\mc{1}&\mc{2\\1}\\\cline{1-3}
	\end{array}}}  \xrightarrow[(3,2)]{p_1\circ\,\mathcal{C}_b}{\def\mc#1{\makecell[lb]{#1}}
	{\begin{array}[lb]{*{3}{|l}|}\cline{1-1}
	\mc{5\\445}\\\cline{1-3}
	\mc{23}&\mc{4\\3}&\mc{4}\\\cline{1-3}
	\mc{1}&\mc{1}&\mc{2\\1}\\\cline{1-3}
	\end{array}}}  \xrightarrow[(2,3)]{p_1\circ\,\mathcal{C}_b}{\def\mc#1{\makecell[lb]{#1}}
{\begin{array}[lb]{*{3}{|l}|}\cline{1-1}
\mc{5\\445}\\\cline{1-2}
\mc{23}&\mc{4\\34}\\\cline{1-3}
\mc{1}&\mc{1}&\mc{2\\1}\\\cline{1-3}
\end{array}}}.
\]
\end{example}

\begin{lemma}\label{lem: path}
If $d(r_j,c_j) = d(r_{j+1},c_{j+1})$, then $\mathsf{path}_{j+1}$ is weakly above $\mathsf{path}_j$.
\end{lemma}

\begin{proof}
By the definition of crowding order, $d(r_j,c_j) = d(r_{j+1},c_{j+1})$ implies $c_j>c_{j+1}$. Set $z_j := c_j-c_{j+1}$. Then we have
$c_j^{(s+z_j)} = c_j-z_j-s = c_{j+1}-s = c_{j+1}^{(s)}$ for $0\leqslant s \leqslant \alpha_{j+1}$. We need to show that 
$r_{j+1}^{(s)} \geqslant r_j^{(s+z_j)}$ for $0\leqslant s \leqslant \alpha_{j+1}$.
Computing $T^{(s)}_{j-1}$ from $T^{(s-1)}_{j-1}$ for $1\leqslant s \leqslant 
\alpha_j$, 
we denote $b$ and $q$ in Step~\ref{alg: m,b} and Step~\ref{alg: r'} of 
Definition~\ref{def: cb} by $b_j^{(s)}$ and $q_j^{(s)}$.

Since $(r_{j+1},c_{j+1})$ is a corner cell in $T^{(z_j)}_{j-1}$, we have 
$r^{(0)}_{j+1}\geqslant r^{(z_j)}_{j}$. 
We prove that, for $1\leqslant s \leqslant \alpha_{j+1}$, we have that 
$q_{j+1}^{(s)}\geqslant q_{j}^{(s+z_j)}$, which implies 
$b_{j+1}^{(s)}\geqslant b_{j}^{(s+z_j)}$ and thus $r_{j+1}^{(s)}\geqslant r_{j}^{(s+z_j)}$.

We prove $q_{j+1}^{(s)}\geqslant q_{j}^{(s+z_j)}$ by induction on $s$. First we check the case $k=1$. If $r^{(0)}_{j+1} > r_{j}^{(z_j)}$, then 
it is obvious that $q_{j+1}^{(1)} > q_j^{(z_j+1)}$. Otherwise if $r^{(0)}_{j+1} 
= r_{j}^{(z_j)}$, we consider the following cases. 
$q_j^{(z_j)}$ is the only element in $\sfA_{T^{(z_j)}_{j-1}}(r_{j+1},c_{j+1})$.
Let $x = \sfH_{T^{(z_j)}_{j-1}}(r_{j+1},c_{j+1})$, $y = 
\max(\sfL_{T^{(z_j)}_{j-1}}(r_{j+1},c_{j+1}))$ and 
$y' = \max\{z \in 
\sfL^+_{T^{(z_j)}_{j-1}}(r_{j+1},c_{j+1})\mid  z \leqslant 
q_j^{(z_j)}\}$.
See Figure~\ref{fig: insertion_0} for illustration.

\smallskip\noindent
\textbf{Case (1):} If $r_j^{(z_j+1)} = r_j^{(z_j)}$, then $q_j^{(z_j+1)} = x$. If $r_{j+1}^{(1)} = r_{j+1}^{(0)}$, then $q_{j+1}^{(1)} = q_j^{(z_j)}$. 
If $r_{j+1}^{(1)} \neq r_{j+1}^{(0)}$, then $q_{j+1}^{(1)}$ equals $y$ when $y>y'$ and $q_j^{(z_j)}$ when $y=y'$. In both cases 
$q_{j+1}^{(1)} \geqslant x = q_j^{(z_j+1)}$.

\begin{figure}[h!]
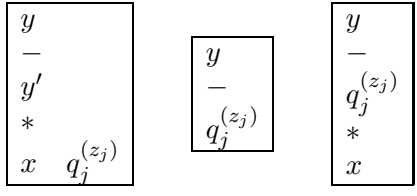

	\[
	{\def\mc#1{\makecell[lb]{#1}}
	{\begin{array}[lb]{*{1}{|l}|}\cline{1-1}
	\mc{y\\-\\y'\\\ast\\x\quad q_j^{(z_j)}}\\\cline{1-1}
	\end{array}}\qquad {\begin{array}[lb]{*{1}{|l}|}\cline{1-1}
		\mc{y\\-\\q_j^{(z_j)}}\\\cline{1-1}
		\end{array}}\qquad {\begin{array}[lb]{*{1}{|l}|}\cline{1-1}
			\mc{y\\-\\q_j^{(z_j)}\\\ast\\x}\\\cline{1-1}
			\end{array}}}\]
		\caption{Cell $(r^{(0)}_{j+1},c^{(0)}_{j+1}) = (r^{(z_j)}_{j},c^{(z_j)}_{j})$ in $T^{(z_j)}_{j-1}$ (left);
		\protect\\ in $T^{(0)}_{j}$, case(1) (middle), \protect\; case(2) (right).}
	\label{fig: insertion_0}
\end{figure}
\smallskip\noindent
\textbf{Case (2):} If $r_j^{(z_j+1)} \neq r_j^{(z_j)}$, then $q_j^{(z_j+1)} = y'$. 
In this case we have $\sfH_{T^{(z_j)}_{j-1}}(r_{j+1}+1,c_{j+1}-1) \leqslant y' \leqslant y$. 
Since $\sfH_{T^{(0)}_j}(r_{j+1}+1,c_{j+1}-1)$ is smaller or equal to $y'$, we 
have that $r_{j+1}^{(1)} \neq r_{j+1}^{(0)}$. Therefore $q_{j+1}^{(1)}$ equals 
$y$ when $y>y'$ and $q_j^{(z_j)}$ when $y=y'$. In this 
case $q_{j+1}^{(1)} \geqslant y' = q_j^{(z_j+1)}$.
\smallskip

Now we have proved the base case $s=1$. Next, suppose it holds for some $s\geqslant 1$ that $q_{j+1}^{(s)}\geqslant q_j^{(s+z_j)}$ 
and $r_{j+1}^{(s)}\geqslant r_j^{(s+z_j)}$. The statement is similar to the argument of the base case. If $r_{j+1}^{(s)}>r_j^{(z_j+s)}$, it
is obvious that $q_{j+1}^{(s+1)}> q_j^{(s+1+z_j)}$ and thus $r_{j+1}^{(s+1)}\geqslant r_j^{(s+1+z_j)}$. If $r_{j+1}^{(s)} = r_j^{(z_j+s)}$, 
we discuss the following cases. $q_j^{(s+z_j)}$ is the only element in $\sfA_{T^{(s+z_j)}_{j-1}}(r^{(s+z_j)}_{j},c^{(s+z_j)}_{j})$.
Let $x = \sfH_{T^{(s+z_j)}_{j-1}}(r^{(s+z_j)}_{j},c^{(s+z_j)}_{j})$, 
$y = \max(\sfL_{T^{(s+z_j)}_{j-1}}
(r^{(s+z_j)}_{j},c^{(s+z_j)}_{j}))$ and $y' = \max\{z \in 
\sfL^+_{T^{(s+z_j)}_{j-1}}(r^{(s+z_j)}_{j},c^{(s+z_j)}_{j})\mid z \leqslant q_j^{(s+z_j)}\}$. 
See Figure~\ref{fig: insertion} for illustration.

\smallskip\noindent
\textbf{Case (1):} If $r_j^{(s+1+z_j)} = r_j^{(s+z_j)}$, then $q_j^{(s+1+z_j)} = x$. If $r_{j+1}^{(s+1)} = r_{j+1}^{(s)}$, then $q_{j+1}^{(s+1)} 
= q_j^{(s+z_j)} \geqslant x$. 
If $r_{j+1}^{(s+1)} \neq r_{j+1}^{(s)}$, then $q_{j+1}^{(s+1)} = 
\max\{z\in \sfL^+_{T^{(s)}_j}(r_{j+1}^{(s)},c_{j+1}^{(s)}) \mid z\leqslant 
q_{j+1}^{(s)}\}\geqslant q_j^{(s+z_j)}\geqslant x$. 
So in either case we have $q_{j+1}^{(s+1)}\geqslant q_j^{(s+1+z_j)}$.

\begin{figure}[h!]
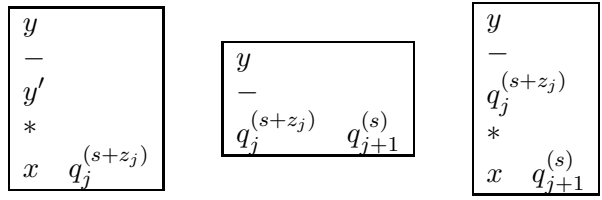

	\[
	{\def\mc#1{\makecell[lb]{#1}}
	{\begin{array}[lb]{*{1}{|l}|}\cline{1-1}
	\mc{y\\-\\y'\\\ast\\x\quad q_j^{(s+z_j)}}\\\cline{1-1}
	\end{array}}\qquad {\begin{array}[lb]{*{1}{|l}|}\cline{1-1}
		\mc{y\\-\\q_j^{(s+z_j)}\quad q_{j+1}^{(s)}}\\\cline{1-1}
		\end{array}}\qquad {\begin{array}[lb]{*{1}{|l}|}\cline{1-1}
			\mc{y\\-\\q_j^{(s+z_j)}\\\ast\\x\quad q_{j+1}^{(s)}}\\\cline{1-1}
			\end{array}}}\]
		\caption{Cell $(r^{(s)}_{j+1},c^{(s)}_{j+1}) = 
		(r^{(s+z_j)}_{j},c^{(s+z_j)}_{j})$ in $T^{(s+z_j)}_{j-1}$ (left); 
		\protect\\in $T^{(s)}_j$, case(1) (middle),\protect\; case(2) (right).}
	\label{fig: insertion}
\end{figure}

\smallskip\noindent
\textbf{Case (2):} If $r_j^{(s+1+z_j)} \neq r_j^{(s+z_j)}$, then 
$q_j^{(s+1+z_j)} = y'$. In this case we have 
$\sfH_{T^{(s+z_j)}_{j-1}}(r_{j}^{(s+z_j)}+1,c_{j}^{(s+z_j)}-1)\leqslant 
y'\leqslant q_j^{(s+z_j)}$. 
Since $\sfH_{T^{(s)}_j}(r^{(s)}_{j+1}+1,c^{(s)}_{j+1}-1)$ is smaller or equal 
to $q_j^{(s+z_j)}$, we have that 
$r^{(s+1)}_{j+1} \neq r_{j+1}^{(s)}$. 
Therefore $q_{j+1}^{(s+1)} = \max\{z\in 
\sfL^+_{T^{(s)}_j}(r^{(s)}_{j+1},c^{(s)}_{j+1})\mid z\leqslant 
q^{(s)}_{j+1}\}$. 
By induction we 
have $q_j^{(s+z_j)}\leqslant q_{j+1}^{(s)}$, thus $q_{j+1}^{(s+1)} \geqslant q_j^{(s+z_j)}\geqslant y' = q_j^{(s+1+z_j)}$.
This completes the proof.
\end{proof}

\begin{lemma}\label{lem: leftmost-topmost}
With the notations as above, let $0 \leqslant j \leqslant e-1$, $0\leqslant  s<\alpha_{j+1}$ 
and $\mathcal{C}_b([T^{(s)}_{j},(r,c)]) = [\,T^{(s+1)}_{j},(r',c-1)]$ for some 
$r,c, r'$. Then in $T^{(s+1)}_{j}$, 
column $c-1$ is the rightmost column with nonzero arm excess and $(r',c-1)$ is the topmost cell in column $c-1$ with nonzero arm excess.
\end{lemma}
	
\begin{proof}
In any $\mathsf{path}_j$, consider the arm excess of its columns. Those with column index $c$ such that $d(r_j,c_j)<c<c_j$ started 
with arm excess $0$, then changed to arm excess $1$ when the insertion path 
passed through that column, and immediately decreased to $0$.

Thus the $q_{j}^{(s)}$ that is being moved to cell $(r',c-1)$ is always at the rightmost column containing nonzero arm excess. 
When $c-1>d(r_j,c_j)$, the arm excess of the column $c-1$ is exactly $1$, $(r',c-1)$ is also the topmost cell containing an arm.
For $c-1=d(r_j,c_j)$, the path $\mathsf{path}_j$ has reached its destination. At that point, any column to the right of $d(r_j,c_j)$ has $0$ 
arm excess. It follows from Lemma~\ref{lem: path} that the cell 
$(r_j^{(\alpha_j)},c_j^{(\alpha_j)})$ is also the topmost cell containing an 
arm.
\end{proof}

\begin{proposition}\label{prop: semistandard}
The tableau $T^{(s+1)}_j$ is a semistandard hook-valued tableau for all 
$0 \leqslant j \leqslant e-1$ and for all $0\leqslant s <\alpha_{j+1}$.
\end{proposition}

\begin{proof}
We only need to check that the $q$ in Step~\ref{alg: r'} of 
Definition~\ref{def: cb} is greater or equal to the hook entry and arm of the 
cell $q$ is to be inserted into.
When $q$ is the only arm element, it is obvious that $q$ is greater or equal to 
the hook entry.

The case when $q$ is not the only arm element can only happen when we reach the 
destination column of the path. 
By the proof of Lemma~\ref{lem: path}, we have that for $q_{j+1}^{(s)}\geqslant q_{j}^{(s+z_j)}$ for $s\geqslant 1$ and for $j$ such 
that $d(r_j,c_j) = d(r_{j+1},c_{j+1})$. Hence the statement follows by setting $k = \alpha_{j+1}$.
\end{proof}

Before we define the ``inverse'' of the uncrowding map $\mathcal{U}: \HVT(\lambda) \to \sqcup_{\mu\supseteq \lambda}\SVT(\mu)\times \hat{\mathcal{F}}(\mu/\lambda)$, we need to restrict our domain to a subset $\mathsf{K}_\lambda$ of 
$\sqcup_{\mu\supseteq \lambda}\SVT(\mu)\times \hat{\mathcal{F}}(\mu/\lambda)$, as the image of $\mathcal{U}$ is not all of $\sqcup_{\mu\supseteq \lambda}\SVT(\mu)\times \hat{\mathcal{F}}(\mu/\lambda)$. We define:
\begin{align*}
\mathsf{K}_\lambda(\mu) :=& \{ (S,F)\in \SVT(\mu) \times \mathcal{\hat{F}}(\mu/\lambda)\,\mid \,\mathsf{weight}(T^{(s)}_j) = \mathsf{weight}(S), \forall \, 0\leqslant j 
\leqslant e-1, \forall \, 0\leqslant s \leqslant \alpha_{j+1}\},\\
\mathsf{K}_\lambda := & \bigsqcup_{\mu\supseteq \lambda}\mathsf{K}_\lambda(\mu).
\end{align*}

\smallskip
\begin{remark}
From the perspective of the uncrowding map, the set-valued tableau $S$ in Example~\ref{social-distancing} cannot be obtained from a 
shape $(1,1)$ hook-valued tableau via the uncrowding map as explained in Remark~\ref{remark: loss}. We say the cell $(1,2)$ in $S$ 
\defn{practices social distancing}. In this case,
\[
	\ytableausetup{notabloids,boxsize=1.6em}
	\left(\,{\def\mc#1{\makecell[lb]{#1}}
	{\begin{array}[lb]{*{2}{|l}|}\cline{1-1}
	\mc{3}\\\cline{1-2}
	\mc{2\\1}&\mc{3\\2}\\\cline{1-2}
	\end{array}}}\;,
	\begin{ytableau}
		*(gray) \\
		*(gray) & 1
	\end{ytableau}\,\right) \notin \mathsf{K}_{(1,1)}.
\]
The $(S,F)$ in Example~\ref{eg: crowding} is in $\mathsf{K}_{(3,2,1)}(5,3,2)$.
\end{remark}

\begin{definition}\label{crowding}
We can now define the \defn{crowding map} $\mathcal{C}$ for any partition $\lambda$ as follows,
\begin{align*}
	\mathcal{C} \colon \mathsf{K}_\lambda &\longrightarrow \HVT(\lambda) \\
		(S,F) & \mapsto T^{(0)}_e.
\end{align*}
\end{definition}

\begin{proposition}
The image of the uncrowding map $\mathcal{U}: \HVT(\lambda) \to \sqcup_{\mu\supseteq \lambda}\SVT(\mu)\times \hat{\mathcal{F}}(\mu/\lambda)$ is a subset of $\mathsf{K}_\lambda$. Moreover, we have 
$\mathcal{C} \circ \mathcal{U} = \mathbf{1}_{\HVT(\lambda)}$.
\end{proposition}

\begin{proof}
 We show that if $\tilde{h} = \mathcal{V}_b(h)$, where $h\in \HVT$, $\mathcal{V}_b$ is as defined in Definition~\ref{def: uncrowd_bump} 
 and $\tilde{h}$ is obtained by moving some letter(s) from the cell $(r,c)$ to $(\tilde{r},c+1)$ (potentially adding a box), then 
 $\mathcal{C}_b([\tilde{h},(\tilde{r},c+1)])=[h',(r',c)]$ satisfies $[h',(r',c)]=[h,(r,c)]$.

We follow the notation used in Definitions~\ref{def: uncrowd_bump} and~\ref{def: cb}. Thus $a = \max(\sfA_h(r,c))$. We have 
that $\sfH_h(\tilde{r},c)\leqslant a$. If cell $(r+1,c)$ is in $h$, then $\sfH_h(r+1,c)>a$.

\smallskip \noindent
\textbf{Case (1):} $\tilde{r} \neq r$.

\smallskip\noindent
\textbf{Case (1A):} If cell $(\tilde{r},c+1)$ is not in $h$, then $h'$ is obtained by adding cell $(\tilde{r},c+1)$ and moving $a$ from 
$\sfA_h(r,c)$ to $\sfH_h(\tilde{r},c+1)$. 
Under the action of $\mathcal{C}_b$, by 
Step~\ref{alg: m,b}, $b=a$ and $r'=r$. $\mathcal{C}_b$ appends $a$ to 
$\sfA_{\tilde{h}}(r,c)$ and removes cell $(\tilde{r},c+1)$, which recovers $h$.
\begin{figure}[h!]
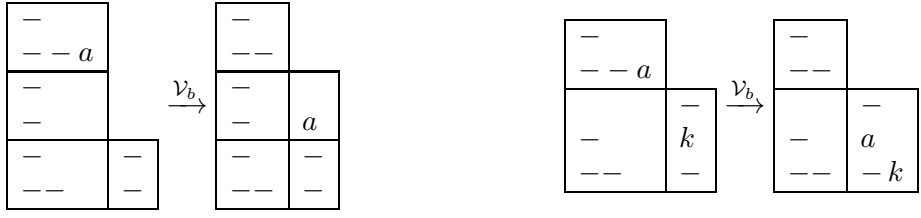

	\begin{minipage}{0.45\textwidth}
		\[
			{\def\mc#1{\makecell[lb]{#1}}
			{\begin{array}[lb]{*{2}{|l}|}\cline{1-1}
			\mc{-\\--a}\\\cline{1-1}
			\mc{- \\-}\\\cline{1-2}
			\mc{- \\ - -}&\mc{-\\-}\\\cline{1-2}
			\end{array}} \xrightarrow[]{\mathcal{V}_b}{\begin{array}[lb]{*{2}{|l}|}\cline{1-1}
				\mc{-\\--}\\\cline{1-2}
				\mc{- \\-}&\mc{a}\\\cline{1-2}
				\mc{- \\ - -}&\mc{-\\-}\\\cline{1-2}
				\end{array}}}
		\]
	\end{minipage}
	\begin{minipage}{0.45\textwidth}
		\[
			{\def\mc#1{\makecell[lb]{#1}}
			{\begin{array}[lb]{*{2}{|l}|}\cline{1-1}
			\mc{-\\--a}\\\cline{1-2}
			\mc{- \\- -}&\mc{-\\k\\-}\\\cline{1-2}
			\end{array}} \xrightarrow[]{\mathcal{V}_b}{\begin{array}[lb]{*{2}{|l}|}\cline{1-1}
				\mc{-\\--}\\\cline{1-2}
				\mc{- \\- -}&\mc{-\\a\\-\,k}\\\cline{1-2}
				\end{array}}}
		\]
	\end{minipage}
	\caption{Left: case (1A): $(\tilde{r},c+1)$ is not in $h$.\quad Right: case (1B): $(\tilde{r},c+1)$ is in $h$.}
\end{figure}

\smallskip\noindent
\textbf{Case (1B):} If cell $(\tilde{r},c+1)$ is in $h$, then $k\in \sfL^{+}_h(\tilde{r},c+1)$ is the smallest number that is greater than or 
equal to $a$ in column $c+1$. $h'$ is obtained by removing $a$ from $\sfA_h(r,c)$, replacing $k$ with $a$, and attaching $k$ to 
$\sfA_h(\tilde{r},c+1)$. 
Under the action of $\mathcal{C}_b$, by Step~\ref{alg: m,b}, we can see that 
$m=k$, $b=a$ and $r'=r$. By Step~\ref{alg:3-b-i}, 
$q=b=a$, and $a$ is appended to $\sfA_{\tilde{h}}(r,c)$ and $q=a$ in $\sfL_{\tilde{h}}(\tilde{r}, c+1)$ is replaced with $m=k$. In the 
end, $m$ is removed from $\sfA_{\tilde{h}}(\tilde{r}, c+1)$. We recover $h$.

\smallskip \noindent
\textbf{Case (2):} $\tilde{r} = r$. Let $\ell = \max (\sfL^+_h(r,c))$.

\smallskip\noindent
\textbf{Case (2A):} If cell $(r,c+1)$ is not in $h$, $\mathcal{V}_b$ adds cell 
$(r,c+1)$, removes the part of $\sfL_h(r,c)$ that is greater than $a$ to 
$\sfL_h(r,c+1)$ and moves $a$ from $\sfA_h(r,c)$ to $\sfH_h(r,c+1)$. 
Under the action of $\mathcal{C}_b$, by Step~\ref{alg: m,b}, $m=a$ and 
$b=\ell$. Thus $r'=r$. By Step~\ref{alg:3-a-ii}, we move $\sfL_{\tilde{h}}(r, 
c+1)$ into $\sfL_{\tilde{h}}(r,c)$ and we recover $h$.
\begin{figure}[h!]
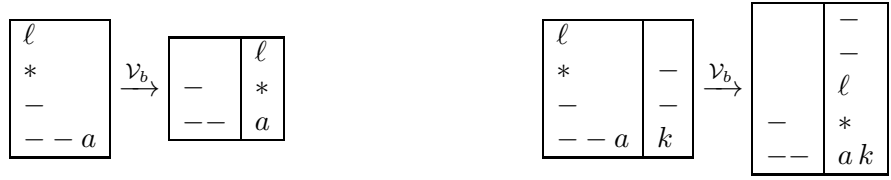

	\begin{minipage}{0.45\textwidth}
	\[
		{\def\mc#1{\makecell[lb]{#1}}
		{\begin{array}[lb]{*{2}{|l}|}\cline{1-1}
		\mc{\ell\\\ast \\-\\--a}\\\cline{1-1}
		\end{array}} \xrightarrow[]{\mathcal{V}_b}{\begin{array}[lb]{*{2}{|l}|}\cline{1-2}
			\mc{- \\--}&\mc{\ell\\\ast\\a}\\\cline{1-2}
			\end{array}}}
	\]
	\end{minipage}
	\begin{minipage}{0.45\textwidth}
		\[
			{\def\mc#1{\makecell[lb]{#1}}
			{\begin{array}[lb]{*{2}{|l}|}\cline{1-2}
			\mc{\ell\\\ast \\-\\--a}&\mc{-\\-\\k}\\\cline{1-2}
			\end{array}} \xrightarrow[]{\mathcal{V}_b}{\begin{array}[lb]{*{2}{|l}|}\cline{1-2}
				\mc{- \\--}&\mc{-\\-\\\ell\\\ast\\a \,k}\\\cline{1-2}
				\end{array}}}
		\]
	\end{minipage}
	\caption{Left: Case (1A): $(r,c+1)$ is not in $h$.\quad Right: Case (1B): $(r,c+1)$ is in $h$.}
\end{figure}

\smallskip\noindent
\textbf{Case (2B):} If cell $(r,c+1)$ is in $h$, $\tilde{h}$ is obtained by 
moving the part of $\sfL_h(r,c)$ that is greater than $a$ to $\sfL_h(r,c+1)$, 
moving $a$ from $\sfA_h(r,c)$ to $\sfH_h(r,c+1)$, and appending $k$ to 
$\sfA_h(r,c+1)$. 
Under the action of $\mathcal{C}_b$, by Step~\ref{alg: m,b}, $m=k$ and 
$b=\ell$. Then $r'=r$ and $q=a$. By Step~\ref{alg:3-a-i}, we move the set 
$\{x\in 
\sfL_{\tilde{h}}(r,c) \mid a<x\leqslant k\}$ from $\sfL_{\tilde{h}}(r,c+1)$ 
into $\sfL_{\tilde{h}}(r,c)$, which is the set that was moved from cell $(r,c)$ 
by $\mathcal{V}_b$. Removing $k$ from $\sfA_{\tilde{h}}(r,c+1)$ and setting 
$\sfH_{\tilde{h}}(r,c+1)=k$, we recover $h$.

Now we have proven $\mathcal{C}_b([\tilde{h},(\tilde{r},c+1)])=[h',(r',c)]=[h,(r,c)]$. It follows that for any $(S,F) = \mathcal{U}(h)$, we 
have that $T^{(s)}_j$ is semistandard and has the same weight as $S$ for all 
$0 \leqslant j \leqslant e-1$, for all $0\leqslant s \leqslant \alpha_{j+1}$. Thus 
$\mathsf{image}(\mathcal{U})\subset \mathsf{K}_\lambda$ and $\mathcal{C} \circ \mathcal{U} = \mathbf{1}_{\HVT(\lambda)}$.
\end{proof}

\begin{proposition}
$\mathsf{K}_\lambda$ is a subset of the image of 
$\mathcal{U}: \HVT(\lambda) \to \sqcup_{\mu\supseteq \lambda}\SVT(\mu)\times \hat{\mathcal{F}}(\mu/\lambda)$.
Moreover, $\mathcal{U}\circ \mathcal{C} = \mathbf{1}_{\mathsf{K}_\lambda}$.
\end{proposition}

\begin{proof}
Let $(S,F) \in \mathsf{K}_\lambda$, then for all $0\leqslant j < e$ and for all $0 \leqslant s <\alpha_{j+1}$, 
$\mathcal{C}_b([\,T^{(s)}_j,(r,c)]) = 
[\,T^{(s+1)}_j,(r',c-1)]$ for some $r,c,r'$.
We show that $\mathcal{V}_b(T^{(s+1)}_j) = T^{(s)}_j$ for all $0\leqslant j <e$ and for all $0 \leqslant s <\alpha_{j+1}$. 
Following the notation in 
Definition~\ref{def: uncrowd_bump}, we first locate the rightmost column that contains nonzero arm excess, then determine 
the topmost cell in row $\tilde{r}$ in that column with nonzero arm excess. We denote by $a$ the largest arm element in that cell.

By Lemma~\ref{lem: leftmost-topmost}, in $T^{(s+1)}_j$, 
column $c-1$ is the rightmost column with nonzero arm excess and 
$(r',c-1)$ is the topmost cell in column $c-1$ with nonzero arm excess.

\smallskip \noindent
\textbf{Case (1):} $r'= r$. In this case either cell $(r+1,c-1)$ does not exist in $T^{(s)}_j$, or $\sfH_{T^{(s)}_j}(r+1,c-1)>b$.

\smallskip \noindent
\textbf{Case (1A):} $\sfA_{T^{(s)}_j}(r,c) = \emptyset$. 
$m=\sfH_{T^{(s)}_j}(r,c)$ 
and $b = \max(\sfL^+_{T^{(s)}_j}(r,c))$. 
Since $r'=r$, $q=m$, 
$T^{(s+1)}_j$ is obtained by appending $m$ to $\sfA_{T^{(s)}_j}(r,c-1)$, moving 
$\sfL_{T^{(s)}_j}(r,c)$ into $\sfL_{T^{(s)}_j}(r,c-1)$, and 
removing cell $(r,c)$ from $T^{(s)}_j$. Note that everything in 
$\sfL_{T^{(s)}_j}(r,c)$ is greater than $m$ and everything in 
$\sfL_{T^{(s)}_j}(r,c-1)$ 
is smaller or equal to $m$.

For the $\mathcal{V}_b$ action, we have $a = m$ and $b$ is the greatest letter 
in $\sfL_{T^{(s+1)}_j}(r,c-1)$. Since every letter in 
$T^{(s+1)}_j(r'',c)$ is smaller than $m$ for $r''<r$, we have 
$\tilde{r}=r$. 
$\mathcal{V}_b$ acts on $T^{(s+1)}_j$ by adding the cell $(r,c)$, 
setting the hook entry to be $m$, and moving $(m,b]\cap 
\sfL_{T^{(s+1)}_j}(r,c-1)$ to $\sfL_{T^{(s+1)}_j}(r,c)$. Then we recover $T^{(s)}_j$.
\begin{figure}[h!]
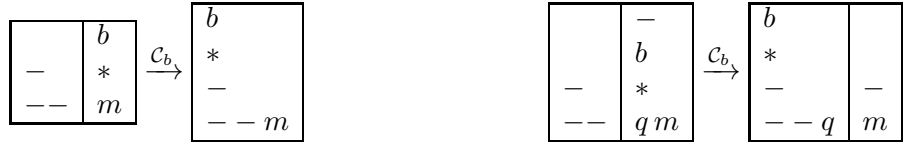

	\begin{minipage}{0.45\textwidth}
		\[
			{\def\mc#1{\makecell[lb]{#1}}
			{{\begin{array}[lb]{*{2}{|l}|}\cline{1-2}
				\mc{- \\--}&\mc{b\\\ast\\m}\\\cline{1-2}
				\end{array}}
				 \xrightarrow[]{\mathcal{C}_b}\begin{array}[lb]{*{2}{|l}|}\cline{1-1}
					\mc{b\\\ast \\-\\--m}\\\cline{1-1}
					\end{array}}
				}
		\]
	\end{minipage}
	\begin{minipage}{0.45\textwidth}
		\[
			{\def\mc#1{\makecell[lb]{#1}}
			{\begin{array}[lb]{*{2}{|l}|}\cline{1-2}
				\mc{- \\--}&\mc{-\\b\\\ast\\q \,m}\\\cline{1-2}
				\end{array}} \xrightarrow[]{\mathcal{C}_b}{\begin{array}[lb]{*{2}{|l}|}\cline{1-2}
				\mc{b\\\ast \\-\\--q}&\mc{-\\m}\\\cline{1-2}
				\end{array}}
			 }
		\]
	\end{minipage}
	\caption{Left: Case (1A): $\sfA_{T^{(s)}_j}(r,c) = \emptyset$.\quad Right: 
	Case (1B): $\sfA_{T^{(s)}_j}(r,c) \neq \emptyset$.}
\end{figure}

\smallskip \noindent
\textbf{Case (1B):} $\sfA_{T^{(s)}_j}(r,c) \neq \emptyset$. $m$ is the only 
element in $\sfA_{T^{(s)}_j}(r,c)$, $q=\sfH_{T^{(s)}_j}(r,c)$ and 
$b = \max \{x\in \sfL^+_{T^{(s)}_j}\mid x\leqslant m\}$. $T^{(s+1)}_j$ is obtained 
by appending $q$ to $\sfA_{T^{(s)}_j}(r,c-1)$, setting 
$\sfH_{T^{(s)}_j}(r,c)$ to be $m$, deleting $\sfA_{T^{(s)}_j}$, and moving $\{x\in 
\sfL_{T^{(s)}_j(r,c)} \mid q<x\leqslant m\}$ to $\sfL_{T^{(s)}_j}(r,c-1)$.

For the $\mathcal{V}_b$ action, $a = q$ and $b$ is the greatest letter in 
$\sfL_{T^{(s+1)}_j}(r,c-1)$. Since every letter in $T^{(s+1)}_j(r'',c)$ 
is smaller than $q$ for $r''<r$ and $m\geqslant q$, $\tilde{r}=r$. 
$\mathcal{V}_b$ acts on $T^{(s+1)}_j$ by setting $\sfH_{T^{(s+1)}_j}(r,c)=q$, 
$\sfA_{T^{(s+1)}_j}(r,c)=m$, and moving $(q,b]\cap \sfL_{T^{(s+1)}_j}(r,c-1)$ 
to $\sfL_{T^{(s+1)}_j}(r,c)$. We recover $T^{(s)}_j$.

\smallskip \noindent
\textbf{Case (2):} $r' \neq r$.

\smallskip \noindent
\textbf{Case (2A):} $\sfA_{T^{(s)}_j}(r,c) = \emptyset$. Note that in this 
case, $\mathcal{C}_b$ will move $m$ somewhere else and remove 
the cell $(r,c)$. Since $\mathsf{weight}(T^{(s+1)}_j) = 
\mathsf{weight}(T^{(s)}_j)$, we must have that $\sfL_{T^{(s)}_j}(r, c) = 
\emptyset$. So
$b = q = m$. $T^{(s+1)}_j$ is obtained from $T^{(s)}_j$ by appending $m$ to 
$\sfA_{T^{(s)}_j}(r',c-1)$ and removing the cell $(r,c)$.

For the $\mathcal{V}_b$ action, $a=m$. Since every letter in 
$T^{(s+1)}_j(r'',c)$ is smaller than $m$ for $r''<r$, a new cell $(r,c)$ is added, 
$\tilde{r}=r$. $\mathcal{V}_b$ acts on $T^{(s+1)}_j$ by moving $m$ to 
$\sfH_{T^{(s+1)}_j}(r, c)$. We recover $T^{(s)}_j$.

\begin{figure}[h!]
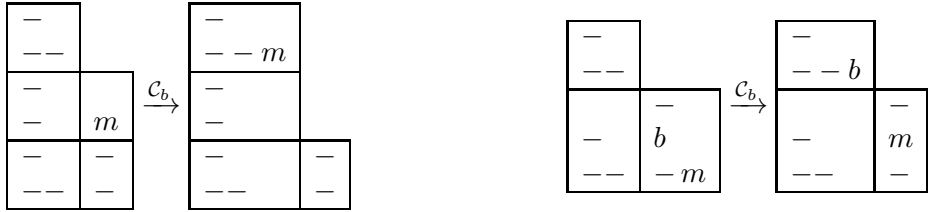

	\begin{minipage}{0.45\textwidth}
		\[
			{\def\mc#1{\makecell[lb]{#1}}
			{\begin{array}[lb]{*{2}{|l}|}\cline{1-1}
				\mc{-\\--}\\\cline{1-2}
				\mc{- \\-}&\mc{m}\\\cline{1-2}
				\mc{- \\ - -}&\mc{-\\-}\\\cline{1-2}
				\end{array}} \xrightarrow[]{\mathcal{C}_b}{\begin{array}[lb]{*{2}{|l}|}\cline{1-1}
				\mc{-\\--m}\\\cline{1-1}
				\mc{- \\-}\\\cline{1-2}
				\mc{- \\ - -}&\mc{-\\-}\\\cline{1-2}
				\end{array}}}
		\]
	\end{minipage}
	\begin{minipage}{0.45\textwidth}
		\[
			{\def\mc#1{\makecell[lb]{#1}}
			{\begin{array}[lb]{*{2}{|l}|}\cline{1-1}
				\mc{-\\--}\\\cline{1-2}
				\mc{- \\- -}&\mc{-\\b\\-\,m}\\\cline{1-2}
				\end{array}} \xrightarrow[]{\mathcal{C}_b}{\begin{array}[lb]{*{2}{|l}|}\cline{1-1}
				\mc{-\\--b}\\\cline{1-2}
				\mc{- \\- -}&\mc{-\\m\\-}\\\cline{1-2}
				\end{array}}}
		\]
	\end{minipage}
	\caption{Left: case (2A): $\sfA_{T^{(s)}_j}(r,c) = \emptyset$.\quad Right: 
	case (2B): $\sfA_{T^{(s)}_j}(r,c) \neq \emptyset$.}
\end{figure}

\smallskip \noindent
\textbf{Case (2B):} $\sfA_{T^{(s)}_j}(r,c) \neq \emptyset$. $m$ is the only 
element in $\sfA_{T^{(s)}_j}(r,c)$, 
$q = b = \max\{x\in \sfL^+_{T^{(s)}_j}(r,c) \mid x\leqslant m\}$. $T^{(s+1)}_j$ 
is obtained by appending $b$ to $\sfA_{T^{(s)}_j}(r',c-1)$, 
replacing $b$ in $\sfL_{T^{(s)}_j}(r,c)$ with $m$, and removing $m$ from 
$\sfA_{T^{(s)}_j}(r,c)$.

For the $\mathcal{V}_b$ action, $a = b$. 
Since every letter in $T^{(s+1)}_j(r'',c)$ is smaller than $b$ for $r''<r$, $m$ 
is the smallest letter 
that is greater or equal to $b$ in column $c$. Hence $\tilde{r}=r$.
$\mathcal{V}_b$ acts on $T^{(s+1)}_j$ by removing $b$ from 
$\sfA_{T^{(s+1)}_j}(r',c-1)$, replacing $m$ in $\sfL_{T^{(s+1)}_j}(r,c)$ with 
$b$, and attaching $m$ to $\sfA_{T^{(s+1)}_j}(r,c)$. We recover $T^{(s)}_j$.

\smallskip

Therefore we have $\mathcal{V}_b(T^{(s+1)}_j) = T^{(s)}_j$ for all $0\leqslant 
j \leqslant e-1$, for all $0\leqslant s < \alpha_j$, and 
$\mathcal{V}(T^{(0)}_{j+1}) = T^{(0)}_j$. It follows that we also recover the 
recording tableau $F$. Thus $\mathcal{U}(T^{(0)}_e) = (S,F)$.
\end{proof}

\begin{corollary}\label{corollary.main}
	The uncrowding map $\mathcal{U}$ is a bijection between $\HVT(\lambda)$ and $\mathsf{K}_\lambda$ with inverse $\mathcal{C}$.
\end{corollary}

\subsection{Alternative uncrowding on hook-valued tableaux}
\label{section.alternative uncrowding}

In Section \ref{section.uncrowding HVT}, we defined an uncrowding map sending hook-valued tableaux to pairs of tableaux with one being 
set-valued and the other being column-flagged increasing. As hook-valued tableaux were introduced as a generalization of both set-valued 
tableaux and multiset-valued tableaux, it is natural to ask if there is an uncrowding map taking hook-valued tableaux to pairs of tableaux 
with one being multiset-valued. In this section we provide such a map.

\begin{definition} \label{def: alt_uncrowd_bump}
The \defn{multiset uncrowding bumping} $\tilde{\mathcal{V}}_{b} \colon \HVT \rightarrow \HVT$ is defined by the following algorithm:
\renewcommand\labelenumi{(\arabic{enumi})}
\renewcommand\theenumi\labelenumi
\begin{enumerate}
    \item \label{def: alt_uncrowd_bump_1} Initialize $T$ as the input.
    \item \label{def: alt_uncrowd_bump_2} If the leg excess of T equals zero, return T.
    \item \label{def: alt_uncrowd_bump_3} Find the topmost row that contains a cell with nonzero leg 
    excess. Within this row, find the cell with the largest value in its 
    leg. (This is the rightmost cell with nonzero leg 
    excess in the specified row.) Denote the row index and column index of 
    this cell by $r$ and $c$, respectively. Denote the cell as $(r,c)$,  its 
    largest leg entry by $\ell$, and its rightmost arm entry by $a$.
    \item \label{def: alt_uncrowd_bump_4} Look at the row above $(r,c)$ (i.e. row $r+1$) and find the leftmost
    number that is strictly greater than $\ell$. 
        \begin{itemize}
            \item If no such number exists, attach an empty cell to the end of row $r+1$ and label the cell as $(r+1, \tilde{c})$,
             where $\tilde{c}$ is its column index. Let $k$ be the empty character.
            \item If such a number exists, label the value as $k$ and the cell containing $k$ as  $(r+1,\tilde{c})$ where $\tilde{c}$ is the 
            cell's column index.
        \end{itemize}
    We now break into cases:
    \begin{enumerate}
        \item \label{def: alt_uncrowd_bump_a} If $\tilde{c} \not = c$, then remove 
        $\ell$ from $\sfL_{T}(r, c)$, replace $k$ with $\ell$, and attach $k$ to the 
        leg of $\sfL_{T}(r+1,\tilde{c})$.
        \item \label{def: alt_uncrowd_bump_b} If $\tilde{c} = c$ then remove $[\ell, 
        a] \cap \sfA_{T}(r, c)$ from $\sfA_{T}(r,c)$ where $[\ell, a] \cap \sfA_{T}(r, c)$ is the multiset 
        $\{z \in \sfA_{T}(r, c) \mid \ell \leqslant z \leqslant a\}$. Remove $\ell$ from 
        $\sfL_{T}(r,c)$, insert $[\ell, a] \cap \sfA_{T}(r, c)$ into $\sfA_{T}(r+1,\tilde{c})$, replace the hook entry 
        of $(r+1, \tilde{c})$ with $\ell$, and attach $k$ to $\sfL_{T}(r+1, \tilde{c})$.
    \end{enumerate}
    \item \label{def: alt_uncrowd_bump_5} Output the resulting tableau.
\end{enumerate}
\end{definition}

\begin{definition} \label{def: alt_uncrowd_insert}
The \defn{multiset uncrowding insertion} $\tilde{\mathcal{V}} \colon \HVT \rightarrow \HVT$ is defined as
$\tilde{\mathcal{V}}(T) = \tilde{\mathcal{V}}_b^d(T)$, where the integer $d\geqslant 1$ is minimal such that 
$\shape(\tilde{\mathcal{V}}_b^d(T))/\shape(\tilde{\mathcal{V}}_b^{d-1}(T)) \neq \emptyset$ or 
$\tilde{\mathcal{V}}_b^d(T) = \tilde{\mathcal{V}}_b^{d-1}(T)$.
\end{definition}

\begin{definition} \label{def: alt_uncrowd_map}
Let $T \in \HVT(\lambda)$ with leg excess $\alpha$. The \defn{multiset uncrowding map}
\[
	\tilde{\mathcal{U}} \colon \HVT(\lambda) \rightarrow \bigsqcup_{\mu\supseteq \lambda}\MVT(\mu) \times \mathcal{F}(\mu/\lambda)
\]
 is defined by the following algorithm:
\renewcommand\labelenumi{(\arabic{enumi})}
\renewcommand\theenumi\labelenumi
\begin{enumerate}
	\item Let $\tilde{P}_{0}=T$ and let $\tilde{Q}_{0}$ be the flagged increasing tableau of shape 
	$\lambda/\lambda$.
	\item
	For $1 \leqslant i \leqslant \alpha$, let $\tilde{P}_{i+1} = \tilde{\mathcal{V}}(\tilde{P}_{i})$. Let $r$ be the index of the topmost 
	row of $\tilde{P}_{i}$ containing a cell with nonzero leg excess and let $\tilde{r}$ be the row index of the cell
	$\shape(\tilde{P}_{i+1})/\shape(\tilde{P}_i)$. Then $\tilde{Q}_{i+1}$ is obtained from $\tilde{Q}_{i}$ by appending the cell 
	$\shape(\tilde{P}_{i+1})/\shape(\tilde{P}_i)$ to $\tilde{Q}_{i}$ and filling this cell with $\tilde{r}- r$.
\end{enumerate}
Define $\tilde{\mathcal{U}}(T) = (\tilde{P}(T), \tilde{Q}(T)) := (\tilde{P}_{\alpha}, \tilde{Q}_{\alpha})$.
\end{definition}

\begin{example}
	Let $T$ be the hook-valued tableau
	\[
	\ytableausetup{notabloids,boxsize=3em}
	T =\, 
	{\def\mc#1{\makecell[lb]{#1}}
	{\begin{array}[lb]{*{3}{|l}|}\cline{1-1}
	\mc{79}\\\cline{1-2}
	\mc{233}&\mc{8\\78}\\\cline{1-3}
	\mc{1}&\mc{3\\223}&\mc{7\\4}\\\cline{1-3}
	\end{array}}\,.}
	\]
	Then, we obtain the following sequence of tableaux $\tilde{\mathcal{V}}_b^i(T)$ for 
	$0\leqslant i\leqslant 2=d$ when computing
	the first multiset uncrowding insertion:
\[
	{\def\mc#1{\makecell[lb]{#1}}
	{\begin{array}[lb]{*{3}{|l}|}\cline{1-1}
	\mc{79}\\\cline{1-2}
	\mc{233}&\mc{8\\78}\\\cline{1-3}
	\mc{1}&\mc{3\\223}&\mc{7\\4}\\\cline{1-3}
	\end{array}}\,
	\rightarrow
	{\begin{array}[lb]{*{3}{|l}|}\cline{1-1}
		\mc{9\\78}\\\cline{1-2}
		\mc{233}&\mc{78}\\\cline{1-3}
		\mc{1}&\mc{3\\223}&\mc{7\\4}\\\cline{1-3}
		\end{array}}\,
		\rightarrow
	{\begin{array}[lb]{*{3}{|l}|}\cline{1-1}
		\mc{9}\\\cline{1-1}
		\mc{78}\\\cline{1-2}
		\mc{233}&\mc{78}\\\cline{1-3}
		\mc{1}&\mc{3\\223}&\mc{7\\4}\\\cline{1-3}
		\end{array}}\,= \tilde{\mathcal{V}}(T).
	}
\]

	Continuing with the remaining multiset uncrowding insertions, we obtain the 
	following sequences of tableaux for the multiset uncrowding map:
	\[
	\begin{aligned}
	{\def\mc#1{\makecell[lb]{#1}}
	{\begin{array}[lb]{*{3}{|l}|}\cline{1-1}
	\mc{79}\\\cline{1-2}
	\mc{233}&\mc{8\\78}\\\cline{1-3}
	\mc{1}&\mc{3\\223}&\mc{7\\4}\\\cline{1-3}
	\end{array}}\,}
	& \rightarrow\; &
	{\def\mc#1{\makecell[lb]{#1}}
	{\begin{array}[lb]{*{3}{|l}|}\cline{1-1}
	\mc{9}\\\cline{1-1}
	\mc{78}\\\cline{1-2}
	\mc{233}&\mc{78}\\\cline{1-3}
	\mc{1}&\mc{3\\223}&\mc{7\\4}\\\cline{1-3}
	\end{array}}\,}
	& \rightarrow\; &
	{\def\mc#1{\makecell[lb]{#1}}
	{\begin{array}[lb]{*{3}{|l}|}\cline{1-1}
	\mc{9}\\\cline{1-2}
	\mc{78}&\mc{8}\\\cline{1-2}
	\mc{233}&\mc{77}\\\cline{1-3}
	\mc{1}&\mc{3\\223}&\mc{4}\\\cline{1-3}
	\end{array}}\,}
	&\rightarrow\; &
	{\def\mc#1{\makecell[lb]{#1}}
	{\begin{array}[lb]{*{3}{|l}|}\cline{1-1}
	\mc{9}\\\cline{1-1}
	\mc{8}\\\cline{1-2}
	\mc{77}&\mc{8}\\\cline{1-2}
	\mc{233}&\mc{337}\\\cline{1-3}
	\mc{1}&\mc{22}&\mc{4}\\\cline{1-3}
	\end{array}}\,} & = \tilde{P}(T),
	\end{aligned}\]
	\[
	\ytableausetup{notabloids,boxsize=1.5em}
	\begin{aligned}
	&\raisebox{5mm}{\begin{ytableau}
		*(gray)\\
		*(gray) & *(gray)\\
		*(gray) & *(gray) & *(gray)
		\end{ytableau}} & \rightarrow\; &
	\raisebox{8mm}{\begin{ytableau}
		2\\
		*(gray)\\
		*(gray) & *(gray)\\
		*(gray) & *(gray) & *(gray)
		\end{ytableau}} & \rightarrow\; &
	\raisebox{8mm}{\begin{ytableau}
		2\\
		*(gray) & 2\\
		*(gray) & *(gray)\\
		*(gray) & *(gray) & *(gray)
		\end{ytableau}} & \rightarrow \; &
	\raisebox{10mm}{\begin{ytableau}
		4\\
		2\\
		*(gray) & 2\\
		*(gray) & *(gray)\\
		*(gray) & *(gray) & *(gray)
		\end{ytableau}} & = \tilde{Q}(T).
	\end{aligned}\]
\end{example}

\begin{proposition}
Let $T \in \HVT$. Then $\tilde{\mathcal{U}}(T)$ is well-defined.
\end{proposition}
\begin{proof}
The statement follows from a similar argument to the proofs found in Corollary~\ref{well_defined_p} and Lemma ~\ref{well_defined_q}.
\end{proof}

Similar to the uncrowding map $\mathcal{U}$, the multiset uncrowding map $\tilde{\mathcal{U}}$ interwines with the corresponding 
crystal operators.
\begin{theorem}
\label{theorem.alternate_main}
    Let $T \in \HVT$. 
    \begin{enumerate}
   \item If $f_{i}(T) = 0$, then $f_{i}(\tilde{P}(T)) = 0$.
   \item If $e_{i}(T) = 0$, then $e_{i}(\tilde{P}(T)) = 0$.
    \item
    If $f_i(T) \neq 0$, we have $f_i(\tilde{P}(T)) = \tilde{P}(f_i(T))$ and $\tilde{Q}(T) = \tilde{Q}(f_i(T))$.
    \item
    If $e_i(T) \neq 0$, we have $e_i(\tilde{P}(T)) = \tilde{P}(e_i(T))$ and $\tilde{Q}(T) =\tilde{Q}(e_i(T))$.
    \end{enumerate}
\end{theorem}

\begin{proof}
The proof follows similarly to those found in Proposition~\ref{proposition.main}, Lemma~\ref{intertwine_lemma}, 
and Theorem~\ref{theorem.main}.
\end{proof}

\section{Applications}
\label{section.application}
In this section, we provide the expansion of the canonical Grothendieck polynomials $G_\lambda(x;\alpha,\beta)$ in terms of
the stable symmetric Grothendieck polynomials $G_\mu(x;\beta=-1)$ and in terms of the dual stable symmetric Grothendieck polynomials 
$g_\mu(x;\beta=1)$ using techniques developed in~\cite{BandlowMorse.2012}. We first review the basic definitions and Schur expansions 
of the two polynomials.

Recall from~\eqref{equation.G stable}, that the stable symmetric Grothendieck polynomial is the generating function of set-valued tableaux
\[
	G_{\mu}(x;-1) = \sum_{S \in \SVT(\mu)} (-1)^{|S|-|\mu|} x^{\mathsf{weight}(S)}.
\]
Its Schur expansion can be obtained from the crystal structure on set-valued tableaux~\cite{MPS.2018}
\[
	 G_{\mu}(x;-1) = \sum_{\substack{ S\in \SVT(\mu)\\ e_i(S) = 0\;\; \forall i}} (-1)^{|S|-|\mu|} \; s_{\mathsf{weight}(S)}.
\]

\begin{definition}
The \defn{reading word} $\mathsf{word}(S)=w_1 w_2 \cdots w_n$ of a set-valued tableau $S\in \SVT(\mu)$ is obtained by reading the elements
in the rows of $S$ from the top row to the bottom row in the following way. In each row, first ignore the smallest element of each cell 
and read all remaining elements in descending order. Then read the smallest elements of each cell in ascending order.
\end{definition}

\begin{example}
The reading word of $P(T)$ in Example~\ref{eg: uncrowding} is $\mathsf{word}(P(T))=8675423362111567$.
\end{example}

\begin{example}
The highest weight set-valued tableaux of shape $(2)$ are
\[{
\def\mc#1{\makecell[lb]{#1}}
\begin{array}[lb]{*{3}{|l}|}\cline{1-2}
	\mc{1}&\mc{1}\\\cline{1-2}
\end{array}\,, \quad
\begin{array}[lb]{*{3}{|l}|}\cline{1-2}
	\mc{1}&\mc{2\\1}\\\cline{1-2}
	\end{array}\,, \quad 
\begin{array}[lb]{*{3}{|l}|}\cline{1-2}
	\mc{1}&\mc{3\\2\\1}\\\cline{1-2}
\end{array}\,, \quad
\begin{array}[lb]{*{3}{|l}|}\cline{1-2}
	\mc{1}&\mc{4\\3\\2\\1}\\\cline{1-2}
\end{array}\,, \quad \dots ,
}\]
which gives the Schur expansion
\[
	G_{(2)}(x;-1) = s_{2}-s_{21}+s_{211}-s_{2111}\pm\cdots.
\]
\end{example}

The \defn{dual stable symmetric Grothendieck polynomials} $g_{\mu}(x;1)$ are dual to $G_\mu(x;-1)$ under the Hall inner product on the ring
of symmetric functions. 

\begin{definition}
A \defn{reverse plane partition} of shape $\mu$ is a filling of the cells in the Ferrers diagram of $\mu$ with positive integers, such that 
the entries are weakly increasing in rows and columns. We denote the collection of all reverse plane partitions of shape $\mu$ by
$\RPP(\mu)$ and the set of all reverse plane partitions by $\RPP$.

The \defn{evaluation} $\ev(R)$ of a reverse plane partition $R\in \RPP$ is a composition $\alpha=(\alpha_i)_{i\geqslant 1}$, 
where $\alpha_i$ is the total number of columns in which $i$ appears. The \defn{reading word} $\mathsf{word}(R)$ is obtained by first circling 
the bottommost occurrence of each letter in each column, and then reading the circled letters row-by-row from top to bottom and left 
to right within each row.
\end{definition}

\begin{example}
Consider the reverse plane partition
\[
R = \begin{ytableau}
1 & 2 \\
1 & 1 & 3
\end{ytableau}\, \in \RPP((3,2)).
\] 
By circling the bottommost occurrence of each letter in each column, we obtain
\[
R = \begin{ytableau}
1 & \circled{2} \\
\circled{1} & \circled{1} & \circled{3}
\end{ytableau},
\; \ev(R) = (2,1,1), \; \mathsf{word}(R) = 
2113.
\]
\end{example}

Lam and Pylyavskyy~\cite{LP.2007} showed that the dual stable symmetric Grothendieck polynomials $g_{\mu}(x;1)$ are generating 
functions of reverse plane partitions of shape $\mu$
\[
	g_{\mu}(x;1) = \sum_{R \in \RPP(\mu)} x^{\ev(R)}.
\]
They also provided the Schur expansion of the dual stable symmetric Grothendieck polynomials~\cite[Theorem 9.8]{LP.2007}
\[
 	g_{\mu}(x;1) = \sum_{F} s_{\mathsf{innershape}(F)},
\]
where the sum is over all flagged increasing tableaux whose outer shape is $\mu$.

\begin{example}
When $\mu = (\mu_1)$ is a partition with only one row, we have $g_{(\mu_1)}(x;1) = s_{(\mu_1)}$.

The flagged increasing tableaux of outer shape $(2,1,1)$ are
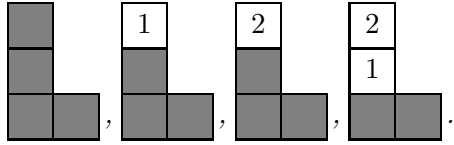
\[
\begin{ytableau}
*(gray) \\
*(gray) \\
*(gray) & *(gray) \\
\end{ytableau}\,\raisebox{-1cm}{,}\;
\begin{ytableau}
1 \\
*(gray) \\
*(gray) & *(gray) \\
\end{ytableau}\,\raisebox{-1cm}{,}\;
\begin{ytableau}
2 \\
*(gray) \\
*(gray) & *(gray) \\
\end{ytableau}\,\raisebox{-1cm}{,}\;
\begin{ytableau}
2 \\
1 \\
*(gray) & *(gray) \\
\end{ytableau}\,\raisebox{-1cm}{.}
\]
Hence $g_{211}(x;1) = s_{211} + 2s_{21}+s_{2}$.
\end{example}

According to~\cite{BandlowMorse.2012}, a symmetric function $f_\alpha$ over the ring $R$ is said to have a \defn{tableaux Schur expansion} 
if there is a set of tableaux $\mathbb{T}(\alpha)$ and a weight function $\mathsf{wt}_\alpha \colon \mathbb{T}(\alpha) \to R$ so that 
\[
	f_\alpha = \sum_{T \in \mathbb{T}(\alpha)} \mathsf{wt}_\alpha(T)s_{\mathsf{shape}(T)}.
\]
Furthermore, any symmetric function with such a property has the following expansion in terms of $G_{\mu}(x;-1)$ and $g_{\mu}(x;1)$.

\begin{theorem} \cite[Theorem 3.5]{BandlowMorse.2012}
Let $f_\alpha$ be a symmetric function with a tableaux Schur expansion
$f_\alpha = \sum_{T\in \mathbb{T}(\alpha)}\mathsf{wt}_{\alpha}(T)s_{\mathsf{shape}(T)}$ for some $\mathbb{T}(\alpha)$. Let
$\mathbb{S}(\alpha)$ and $\mathbb{R}(\alpha)$ be defined as sets of set-valued tableaux and reverse plane partitions, respectively, by
\begin{align*}
	S \in \mathbb{S}(\alpha) &\text{ if and only if } P(\mathsf{word}(S)) \in \mathbb{T}(\alpha), \text{ and}\\
	R \in \mathbb{R}(\alpha) &\text{ if and only if } P(\mathsf{word}(R)) \in \mathbb{T}(\alpha),
\end{align*}
where $P(w)$ is the RSK insertion tableau of the word $w$.
We also extend $\mathsf{wt}_\alpha$ to $\mathbb{S}(\alpha)$ and $\mathbb{R}(\alpha)$ by setting 
$\mathsf{wt}_\alpha(X):= \mathsf{wt}_\alpha(P(\mathsf{word}(X)))$ for any $X \in \mathbb{S}(\alpha)$ or $\mathbb{R}(\alpha)$. Then we have
\begin{align*}
	f_\alpha &= \sum_{R\in \mathbb{R}(\alpha)}\mathsf{wt}_\alpha(R)G_{\mathsf{shape}(R)}(x;-1), \text{ and}\\
	f_\alpha &= \sum_{S\in \mathbb{S}(\alpha)}\mathsf{wt}_\alpha(S)(-1)^{|S|-|\mathsf{shape}(S)|}g_{\mathsf{shape}(S)}(x;1).
\end{align*}
\end{theorem}

\begin{proposition}
\label{proposition.Schur expansion}
The canonical Grothendieck polynomials have a tableaux Schur expansion.
\end{proposition}

\begin{proof}
Recall the uncrowding map on set-valued tableaux of Definition~\ref{definition.uncrowding SVT}
\[
	\mathcal{U}_{\SVT}: \, \SVT(\mu) \longrightarrow \bigsqcup_{\nu\supseteq \mu}\SSYT(\nu)\times \mathcal{F}(\nu/\mu).
\]
By Corollary~\ref{corollary.main}, we have a bijection 
\[
	\mathcal{U}: \HVT(\lambda) \to \mathsf{K}_\lambda = \bigsqcup_{\mu \supseteq \lambda}\mathsf{K}_\lambda(\mu).
\]
Note that $\mathsf{K}_\lambda \subseteq \bigsqcup_{\mu \supseteq \lambda} \SVT(\mu)\times \mathcal{\hat{F}}(\mu/\lambda)$.
Denote 
\[
	\phi_\lambda(S) = |\{F\in \hat{\mathcal{F}} \mid (S,F)\in \mathsf{K}_\lambda\}|.
\]
Note that sometimes $\phi_\lambda(S) = 0$.

Given $H\in \HVT(\lambda)$, we have $\mathcal{U}(H) = (S,F)\in \SVT(\mu)\times \mathcal{\hat{F}}(\mu/\lambda)$ for some 
$\mu \supseteq \lambda$ and $|\mu|= |\lambda|+a(H)$. We can also obtain $\mathcal{U}_{\SVT}(S) = (T,Q)\in 
\SSYT(\nu)\times \mathcal{F}(\nu/\mu)$ for some $\nu\supseteq \mu$ and $|\nu|= |H|$. The weights of $H,S$ and $T$ are the same.
When $H$ is highest weight, that is $e_i(H)=0$ for all $i$, then $S$ and $T$ are also of highest weight and 
$\mathsf{weight}(H) = \mathsf{shape}(T)$. Denote by $\HVT_h(\lambda), \SVT_h(\lambda), \SSYT_h(\lambda)$ the subset 
of highest weight elements in $\HVT(\lambda), \SVT(\lambda), \SSYT(\lambda)$, respectively. 

Applying~\cite[Theorem 4.6]{HS.2019} and the above correspondence, we obtain
\begin{align*}
	G_{\lambda}(x;\alpha,\beta) &= \sum_{H\in \HVT_h(\lambda)} \alpha^{a(H)}\beta^{\ell(H)}s_{\mathsf{weight}(H)} 
	= \sum_{\mu\supseteq \lambda}\sum_{(S,F)\in \mathsf{K}_\lambda(\mu)}\alpha^{|\mu|-|\lambda|}\beta^{|S|-|\mu|}s_{\mathsf{weight}(S)}\\
	&= \sum_{\mu\supseteq \lambda} \sum_{S\in \SVT_h(\mu)} \phi_\lambda(S)\alpha^{|\mu|-|\lambda|}\beta^{|S|-|\mu|}s_{\mathsf{weight}(S)}\\
	&= \sum_{\mu\supseteq \lambda} \sum_{\nu \supseteq \mu}\sum_{T\in \SSYT_h(\nu)}\sum_{Q\in \mathcal{F}(\nu/\mu)}
	\phi_\lambda(\mathcal{U}^{-1}_{\SVT}(T,Q))\alpha^{|\mu|-|\lambda|}\beta^{|\nu|-|\mu|}s_{\mathsf{weight}(T)}\\
	&= \sum_{\mu\supseteq \lambda}\sum_{\nu \supseteq \mu}\sum_{T\in \SSYT_h(\nu)}\alpha^{|\mu|-|\lambda|}\beta^{|\nu|-|\mu|}
	\sum_{Q\in \mathcal{F}(\nu/\mu)}\phi_\lambda(\mathcal{U}^{-1}_{\SVT}(T,Q))s_{\mathsf{shape}(T)}\\
	&= \sum_{T\in \mathbb{T}(\lambda)} \mathsf{wt}_{\lambda}(T)s_{\mathsf{shape}(T)},
\end{align*}
where $\mathbb{T}(\lambda) = \{T\in \SSYT_h(\nu) \mid \nu\supseteq \lambda\}$ and 
\[
	\mathsf{wt}_{\lambda}(T) = \sum_{\mu: \lambda \subseteq \mu \subseteq \mathsf{shape}(T)}\alpha^{|\mu|-|\lambda|}
	\beta^{|\mathsf{shape}(T)|-|\mu|}\sum_{Q\in \mathcal{F}(\mathsf{shape}(T)/\mu)}\phi_\lambda(\mathcal{U}^{-1}_{\SVT}(T,Q)).
\]
\end{proof}

Note that Proposition~\ref{proposition.Schur expansion} in particular implies that the canonical Grothendieck polynomials are Schur positive. 
This was known from~\cite{HS.2019}, but here an explicit tableaux formula is given. 

\begin{corollary}
The canonical Grothendieck polynomials have $G_\mu(x;-1)$ and $g_\mu(x;1)$ expansions:
\begin{align*}
	G_\lambda(x;\alpha,\beta) &= \sum_{R\in \mathbb{R}(\lambda)}\mathsf{wt}_\lambda(R)G_{\mathsf{shape}(R)}(x;-1),\\
	G_\lambda(x;\alpha,\beta) &= \sum_{S\in \mathbb{S}(\lambda)}\mathsf{wt}_\lambda(S)(-1)^{|S|-|\mathsf{shape}(S)|}
	g_{\mathsf{shape}(S)}(x;1).	
\end{align*}
\end{corollary}

\begin{example}
We compute the first two terms in $G_{(2)}(x;\alpha,\beta) = s_{2} + \beta s_{21} + 2\alpha s_{3} + 2\alpha\beta s_{31}+ \cdots$. 
The semistandard Young tableaux involved are 
\[
\mathbb{T}((2)) = \left\{
\raisebox{-0.65cm}{
\begin{ytableau}
	1 & 1
\end{ytableau}}\,, 
\begin{ytableau}
	2 \\
	1 & 1
\end{ytableau}\,,
\raisebox{-0.65cm}{
\begin{ytableau}
	1 & 1 & 1
\end{ytableau}}\,,
\begin{ytableau}
	2\\
	1 & 1 & 1
\end{ytableau}\,, \dots\right\}.	
\]
Labelling the tableaux $T_1,T_2,T_3,T_4,\dots$, we have $\mathsf{wt}_{(2)}(T_1)= 1, \mathsf{wt}_{(2)}(T_2)=\beta, 
\mathsf{wt}_{(2)}(T_3) = 2\alpha, \mathsf{wt}_{(2)}(T_4) = 2\alpha\beta$. Next we compute the elements in $\mathbb{R}((2))$ and 
$\mathbb{S}((2))$ that correspond to $T_1$ and $T_2$:
\begin{align*}
	\{R\in \mathbb{R}((2)) \mid P(\mathsf{word}(R))=T_1\} =&\bigl\{\,
	\raisebox{-0.65cm}{
\begin{ytableau}
	1 & 1
\end{ytableau}}\,,
\begin{ytableau}
	1 \\
	1 & 1
\end{ytableau}\,,
\begin{ytableau}
	1 & 1\\
	1 & 1
\end{ytableau}\,,
\raisebox{0.65cm}{
\begin{ytableau}
	1 \\
	1 \\
	1 & 1
\end{ytableau}}\,,\dots\bigr\}\\
\{R\in \mathbb{R}((2)) \mid P(\mathsf{word}(R))=T_2\} =&\bigl\{\,
\begin{ytableau}
	2\\
	1 & 1
\end{ytableau}\,,
\begin{ytableau}
	1 & 2\\
	1 & 1
\end{ytableau}\,,
\raisebox{0.65cm}{
\begin{ytableau}
	2\\
	1 \\
	1 & 1
\end{ytableau}}\,,
\raisebox{0.65cm}{
\begin{ytableau}
	2 \\
	2 \\
	1 & 1
\end{ytableau}}\, ,\dots \bigr\}\\
\{S\in \mathbb{S}((2)) \mid P(\mathsf{word}(S))=T_1\} =&\bigl\{\,
\begin{ytableau}
	1 & 1
\end{ytableau}\,\bigr\}\\
\{S\in \mathbb{S}((2)) \mid P(\mathsf{word}(S))=T_2\} =&\bigl\{\,
\begin{ytableau}
	2\\
	1 & 1
\end{ytableau}\,,
{\def\mc#1{\makecell[lb]{#1}}
{\begin{array}[lb]{*{2}{|l}|}\cline{1-2}
\mc{1}&\mc{2\\1}\\\cline{1-2}
\end{array}}}\bigr\}.
\end{align*}
Applying the expansion formulas, we obtain
\begin{align*}
	G_{(2)}(x;\alpha,\beta) =& (G_{(2)}(x;-1)+G_{(21)}(x;-1)+G_{(22)}(x;-1)+G_{(211)}(x;-1)+\cdots) \\
	&+ \beta(G_{(21)}(x;-1)+G_{(22)}(x;-1)+2G_{(211)}(x;-1)+\cdots) + \cdots\\
	G_{(2)}(x;\alpha,\beta) =& g_{(2)}(x;1)+\beta(g_{(21)}(x;1)-g_{(2)}(x;1)) + \cdots .
\end{align*}
\end{example}

\bibliographystyle{alpha}
\bibliography{main}{}

\end{document}